\documentclass[12pt]{amsart}
\usepackage{amsthm,amssymb}
\usepackage{epsfig}
\usepackage[arrow]{xy}
\usepackage{float}
\restylefloat{figure}

\usepackage{amssymb,latexsym,cite,epsf,graphics}
\usepackage[dvips,centering,includehead,width=15.1cm,height=22.7cm]{geometry}
\usepackage{graphicx,color}

\definecolor{darkred}{rgb}{1,0,0}
\newcommand{\darkred}[1]{\color{darkred}{#1}\color{black}}

\newtheorem{theorem}{Theorem}[section]

\newtheorem{lemma}[theorem]{Lemma}

\newtheorem{proposition}[theorem]{Proposition}
\newtheorem{prop}[theorem]{Proposition}
\newtheorem{corollary}[theorem]{Corollary}

\newtheorem{conjecture}[theorem]{Conjecture}

\theoremstyle{remark}

\theoremstyle{definition}

\newtheorem{remark}[theorem]{Remark}
\newtheorem{example}[theorem]{Example}
\newtheorem{definition}[theorem]{Definition}

\def\Acal{\mathcal{A}} 
\def\Bcal{\mathcal{B}} 
\def\Ccal{\mathcal{C}} 
\def\Fcal{\mathcal{F}} 
\def\CC{\mathbb{C}} 
 
\def\cc{\mathbf{c}}
\def\xx{\mathbf{x}}
\def\zz{\mathbf{z}}
\def\ee{\mathbf{e}} 
\def\vol{\operatorname{vol}}
\def\SL{\operatorname{SL}}
\def\QF{\operatorname{QF}}

\newcommand{\usection}[1]{\section*{\textbf{#1}}}

\title{
Tensor diagrams and cluster algebras
}

\numberwithin{equation}{section}

\begin{document}

\author{Sergey Fomin}
\address{\hspace{-.3in} Department of Mathematics, University of Michigan,
Ann Arbor, MI 48109, USA}
\email{fomin@umich.edu}

\author{Pavlo Pylyavskyy}
\address{\hspace{-.3in} Department of Mathematics, University of Minnesota,
Minneapolis, MN 55414, USA}
\email{ppylyavs@umn.edu}

\date{May 7, 2015
}

\thanks{Partially supported by NSF grants DMS-1101152 (S.~F.)
and DMS-1068169 (P.~P.).
}

\dedicatory{To the memory of Andrei Zelevinsky}



\begin{abstract}
The rings of $\SL(V)$ invariants of configurations of vectors
and linear forms in a finite-dimensional complex vector space $V$ 
were explicitly described by Hermann Weyl in the 1930s. 
We show that when $V$ is $3$-dimensional, each of these rings carries
a natural cluster algebra structure (typically, many of them) whose cluster
variables include Weyl's generators. 
We describe and explore these cluster structures using the
combinatorial machinery of tensor diagrams.
A~key role is played by the web bases introduced by G.~Kuperberg. 
\end{abstract}

\ \vspace{-.1in}

\maketitle

\vspace{-.3in}

\tableofcontents

\usection{Introduction}

Homogeneous coordinate rings of Grassmannians 
are among the most important examples of 
cluster algebras. 
Cluster structures in these rings \cite{gsv1, scott}
play a prominent role in applications of
cluster theory arising in connection with
integrable systems, 
algebraic Lie theory,
Poisson geometry,
Teichm\"uller theory,
total positivity, 
and beyond; 
see, e.g.,  \cite{gls-lie, gsv-book, grabowski-launois,
keller-triangulated, kodama-williams} and references therein. 
Within cluster algebra theory proper, Grassmannians provide the most concrete
and accessible examples of naturally defined cluster algebras 
of infinite mutation type. 

Despite their importance, cluster structures on 
Grassmannians are not well understood at all,
apart from a few special cases.  
Just a tiny subset of their cluster variables have been explicitly
described;
we do not know which quivers appear in their seeds;
we do not understand the structure of their underlying cluster
complexes; and so on.

Let $\operatorname{Gr}_{k,N}$ denote the Grassmann manifold 
of $k$-subspaces in an $N$-dimensional complex vector space. 
The corresponding cluster algebra 
 has finite type (i.e., has finitely
many seeds) if and only if $(k-2)(N-k-2)\le 3$. 
All of the problems mentioned above are open for any Grassmannian of
infinite cluster type, so in particular for 
$k=3$, $N\ge 9$. 
(The case $k=2$ has been well understood since the early days of
cluster algebras, see~\cite[Section~12.2]{ca2}.) 

We advocate the point of view that many aspects of cluster 
structures on Grassmannians are best understood within a broader
range of examples coming from classical invariant theory. 
Recall that the homogeneous coordinate ring
of~$\operatorname{Gr}_{k,N}$ 
(with respect to a Pl\"ucker embedding) 
is isomorphic to the ring of $\SL(V)$ invariants of 
$N$-tuples of vectors in a 
$k$-dimensional complex vector space~$V$. 
More general rings of $\SL(V)$ invariants of 
collections of vectors and linear forms have been thoroughly studied
by classical invariant theory. 
We conjecture that \emph{every such ring carries a natural cluster algebra
structure}. 
In fact, there are typically many such structures, 
depending on a choice of a cyclic ordering of the tuples of vectors and
covectors. 

In this paper, we prove this claim in the special case when $V$ is
$3$-dimensional. 
Our main result (Theorem~\ref{th:main}) describes a family of cluster
structures in the rings of $\SL_3$-invariants of collections of
vectors and covectors in~$V\cong\CC^3$.
The cluster structure (and moreover the mutation type of the
corresponding quiver) depends on the choice of a \emph{cyclic
  signature}, a binary word describing the order in which vectors
and covectors are arranged around a circle. 
An accurate formulation of this result requires a description of (some
of) the seeds defining the cluster structure in question;
this in turn relies on the development of fairly technical 
combinatorial vocabulary, not to mention the requisite background on
cluster algebras, tensor calculus, and basic invariant theory. 
This makes it impractical to include the precise statement of our main
result in this introduction. 

What makes the case $k=3$ special, and amenable to combinatorial
approaches that cannot be straightforwardly generalized to higher
dimensions? 
{}From our perspective, 
the distinguishing characteristic of the $3$-dimensional case 
is the existence of a beautiful
\emph{web basis} discovered by Greg Kuperberg~\cite{kuperberg}. 
Our investigations were largely motivated by the desire to understand 
the cluster-theoretic significance of Kuperberg's basis.
The main new feature of our approach is an emphasis on the
multiplicative properties of the web basis,
which along with its compatibility with tensor contraction 
play a central role in the study of
cluster algebra structures in these rings of invariants. 

While all of our results hold for arbitrary rings of
$\SL_3$-invariants, with an (almost) arbitrary cyclic ordering of
vectors and covectors, 
many of our theorems are new already in the case of Grassmannians 
$\operatorname{Gr}_{3,N}$.
The list of such results includes: 
\begin{itemize}
\item a large family of new ``non-Pl\"ucker'' cluster variables;
\item
compatibility of cluster structures across different values of~$N$; 
\item
examples of ``imaginary'' elements in the web basis;
\item
examples of negative
  structure constants; 
\item
a description of clusters associated with arbitrary
  triangulations of an $N$-gon.
\end{itemize}
The latter description extends a construction given by
  J.~Scott~\cite{scott}, and
  suggests an extension to 
 other Riemann surfaces, which we plan to pursue in a separate~paper. 

We describe conjectural ways in which each cluster structure in a
classical ring of $\SL_3$-invariants is
intrinsically determined by 
the corresponding web basis. 
We believe that this type of relationship extends to many other
settings in which 
multiplicative properties of some distinguished additive
basis in a given commutative ring 
dictate a ``canonical'' choice of a cluster structure
in the ring. 
We also believe that, conversely, each cluster algebra
of geometric type has an
additive basis with certain remarkable properties; see 
Conjecture~\ref{conj:A-and-B}. 

Perhaps most significantly, we formulate a conjectural combinatorial
description (see Conjecture~\ref{conj:cluster-variables-are-trees})
of \emph{all} cluster variables in each of our cluster algebras.
To the best of our knowledge, 
this is the first class of cluster algebras of infinite
mutation type for which such a description has been proposed. 

\smallskip
\centerline{------------}
\smallskip

We next outline the general plan of the paper, 
and review the contents of each section. 

Sections~\ref{sec:tensors}--\ref{sec:webs} cover the requisite
background, some of which may be familiar to the reader.
Section~\ref{sec:tensors} is a reminder 
on the basic notions of tensor calculus. 
Section~\ref{sec:rings-of-invariants} introduces our main object of
study, the ring $R_{a,b}(V)$ of $\SL(V)$ invariants of collections of
$b$~vectors and $a$~covectors 
(i.e., linear forms) in a complex vector space~$V$. 
Fundamentals of cluster algebras are reviewed in 
Section~\ref{sec:cluster-algebras}.
We limit ourselves to the case of cluster algebras defined by
quivers, 
as this is all the generality we need.
At the end of the section, we state sufficient conditions 
that ensure that a particular seed, i.e., a quiver whose vertices are
labeled by elements of a given ring~$R$, 
defines a cluster algebra structure in~$R$;
see Proposition~\ref{prop:cluster-criterion}
and Corollary~\ref{cor:cluster-criterion}.

Starting with Section~\ref{sec:tensor-diagrams},
we assume that the space~$V$ is $3$-dimensional. 
Section~\ref{sec:tensor-diagrams} is a primer on \emph{tensor
  diagrams}, 
a particular kind of combinatorial gadgets 
that can be used to define $\SL(V)$ invariants. 
This diagrammatic calculus, 
popular among physicists, has the advantage of reducing cumbersome
calculations with tensors and invariants 
to repeated applications of certain (purely
combinatorial) local transformation rules called skein relations. 

Section~\ref{sec:webs} presents 
the crown jewel of the theory of tensor diagrams: Kuperberg's
construction of the web basis in $R_{a,b}(V)$. 
This basis, which depends on
a choice of a cyclic signature, consists of \emph{web invariants},
the $\SL(V)$ invariants defined 
by planar tensor diagrams without short internal cycles. 

Sections~\ref{sec:special-invariants} and~\ref{sec:special-seeds}
introduce the key combinatorial construction of the paper:
a family of ``special'' seeds designed to define a cluster structure 
in~$R_{a,b}(V)$. 
In Section~\ref{sec:special-invariants}, we describe the ``special''
invariants that appear in those seeds, and state their basic
properties. 
We then explain in Section~\ref{sec:special-seeds} how each triangulation
of a convex $(a+b)$-gon gives rise to an extended cluster consisting
of $\dim(R_{a,b}(V))=3(a+b)-8$ special invariants,
and to a quiver whose vertices are labeled by them. 
An impatient reader unwilling to labor through the rather intricate
technicalities of this construction may decide to skip the details,
and go directly to Section~\ref{sec:main-theorem}. 

Sections~\ref{sec:main-theorem}--\ref{sec:zoo} present the main
results and conjectures of the paper.
(The proofs are deferred until later.)
The main theorem (Theorem~\ref{th:main}) asserts that each of the
special seeds described in Section~\ref{sec:special-seeds} defines a cluster
algebra structure in the ring of invariants~$R_{a,b}(V)$.
This cluster structure is independent of the choice of
such a seed; in other words, all special seeds are mutation
equivalent. The cluster structure does however depend on the choice of
a signature. In fact, this choice affects the (cluster) type of the resulting
cluster algebra, and whether it is of finite or infinite type 
(resp., finite or infinite mutation type). 
This can be seen in Figure~\ref{fig:boundary-signatures-5678} that lists 
the cluster types of~$R_{a,b}(V)$ for all signatures with
$a+b\le 8$. 

Our main construction has nice functoriality properties:
it is preserved by the duality between vectors and covectors,
and respected by
the natural embeddings of smaller invariant rings into larger ones, 
induced by the forgetful maps (dropping a (co)vector 
from a collection) or the maps defined by taking a
cross product of two consecutive (co)vectors. 
As a corollary, we establish that if a tensor
diagram is a planar tree, then the corresponding web invariant is a
cluster or coefficient variable. 


In Section~\ref{sec:main-conjectures}, 
we discuss many conjectural connections between our main construction
and Kuperberg's web basis.
In particular, we expect this basis to contain all cluster monomials.
Furthermore, 
we expect two cluster variables to be compatible if and only if their
product is a web invariant. 
We formulate criteria that should distinguish cluster
  variables/monomials among more general web invariants. 

While the web basis has many wonderful properties, 
it \emph{may} have negative structure constants
(cf.\ Conjecture~\ref{conj:strong-positivity} 
and Proposition~\ref{prop:negative-structure-const}). 
This is perhaps not so surprising in light of the discovery, made by
M.~Khovanov and G.~Kuperberg~\cite{khovanov-kuperberg},
that the web basis is generally different from the \emph{dual canonical
basis}, i.e., the basis dual to G.~Lusztig's canonical
basis~\cite{lusztig-quantum-groups}. 
Putting things into a cluster-theoretic context allows us to
``explain'' the Khovanov-Kuperberg smallest counterexample:
it corresponds to the square of the simplest web invariant which is
not a cluster variable. 
As established by B.~Leclerc \cite{leclerc-reims} and
P.~Lampe~\cite{lampe}, for the cluster type at hand the appropriate
element of the dual canonical basis is given by the quadratic
Chebyshev polynomial of the second kind. Using the Chebyshev
polynomial of the first kind recovers the corresponding element of the
``atomic'' basis of P.~Sherman and
A.~Zelevinsky~\cite{sherman-zelevinsky}, while the square lies in the
\emph{dual semicanonical basis}, see~\cite{gls-semicanonical,
lusztig-semicanonical}. It is
then natural to ask: Does the web basis always coincide with the
appropriate dual semicanonical (or ``generic'') basis, in the sense of
\cite{gls-semicanonical, gls-jams}?

Section~\ref{sec:arborization} begins by our favorite conjecture of
the paper  
(Conjecture~\ref{conj:cluster-variables-are-trees}) that 
describes, in simple combinatorial terms, the entire set of
 cluster monomials
in~$R_{a,b}(V)$. 
According to this conjecture, a cluster monomial is an $\SL(V)$ 
invariant that possesses two alternative presentations by a single
tensor diagram:
first, by a (planar, non-elliptic) web;
second, by a (possibly non-planar) forest. 
If, in addition, this invariant does not factor 
(so that this forest is actually a tree),
then the invariant is a cluster or coefficient variable. 
We then present an explicit \emph{arborization} algorithm that
conjecturally detects whether a given web defines a cluster monomial
(resp., a cluster or coefficient variable) by applying a sequence of
skein relations that either transform the web into a forest (resp., tree),
or else determine that this is impossible. 

Section~\ref{sec:zoo} presents a gallery of fairly complicated
examples of webs illustrating various phenomena that we
discovered. These include: non-arborizable webs, both 
``real'' and ``imaginary'' (in the sense of
B.~Leclerc~\cite{leclerc});
``fake'' exchange relations involving web invariants;
and the aforementioned negative structure constants. 

Sections~\ref{sec:special-proofs}--\ref{sec:other-proofs} contain the
proofs of the results stated in
Sections~\ref{sec:special-invariants}--\ref{sec:main-theorem} 
and~\ref{sec:arborization}. 
Specifically, the properties of special invariants and special seeds
formulated in
Sections~\ref{sec:special-invariants}--\ref{sec:special-seeds} are
proved in
Sections~\ref{sec:special-proofs}--\ref{sec:properties-of-special-seeds}. 
The construction of the quiver associated with a special seed defined
by an arbitrary triangulation is outlined in 
Section~\ref{sec:building-a-quiver}. 
The main Theorem~\ref{th:main} is proved in
Section~\ref{sec:proof-main}. 
Other results stated in Sections~\ref{sec:main-theorem}
and~\ref{sec:arborization} are proved in
Section~\ref{sec:other-proofs}. 

The 
last two sections
play the role of appendices. 
Section~\ref{sec:T-fan} describes, in precise combinatorial terms, 
the construction of a special seed associated with a particular choice
of a triangulation. This is in principle sufficient to identify a
cluster structure in the ring~$R_{a,b}(V)$ once the main theorem has
been proved. 
Section~\ref{sec:examples-of-seeds} shows a couple of additional
examples of special seeds, illustrating the main construction. 


The paper contains a large number of pictures,
which are best viewed in \darkred{color}. 

\medskip

The first version of this paper~\cite{fp-arxiv} was circulated in
October 2012. 
While the text subsequently underwent extensive editorial revisions of
expository nature, 
the main results of the paper, and the arguments used in their proofs,
remained essentially unchanged. 
We also refrained from describing the developments that took place in this research
subfield after the first \texttt{arXiv} posting. 

\subsection*{Acknowledgments}
We were influenced in our work 
by the research of many friends and colleagues 
on the subject of cluster algebras. 
We are grateful to 
Arkady Berenstein, 
Thomas Lam,  
Bernard Leclerc, 
Jan Schr\"oer, 
David Speyer,
and Andrei Zelevinsky
for helpful advice and stimulating discussions,
and to 
Igor Dolgachev,
Rob Lazarsfeld, 
Ivan Losev, 
Ezra Miller, 
and
Mircea Musta\c{t}\u{a}
for answering our questions on invariant theory. 
We thank Bernard Leclerc, Gregg Musiker, and the anonymous referees
whose multiple comments on the earlier versions of the paper
were critical in improving the quality of exposition. 

Our main sources of inspiration outside cluster theory
included the timeless texts by H.~Weyl \cite{weyl} 
and V.~Popov--E.~Vinberg~\cite{popov-vinberg}
on the foundations of classical invariant theory;  
the pioneering work by G.~Kuperberg~\cite{kuperberg} on the web bases; 
and the groundbreaking paper by V.~Fock and A.~Goncharov~\cite{fg-convex} 
on 
$\operatorname{PGL}_3$-Teich\-m\"uller~spaces.


Our work on this paper began and ended at the Mathematical Sciences
Research Institute in Berkeley, CA.
It started in 2008 during MSRI's 
Combinatorial Representation Theory program, 
and was completed in 2012 during the Cluster Algebras
program.
We are grateful to MSRI for an excellent work environment. 

The main results of this paper were first reported at 
the Hausdorff Institute (Bonn) 
and the Abel Symposium (Balestrand) in April and June 2011, 
respectively. 
The first author thanks the organizers of these events
for their hospitality and support. 

\smallskip
\centerline{------------}
\smallskip

This paper is dedicated to the memory of Andrei Zelevinsky,
a dear friend, collaborator, and mentor. 
His beautiful theorems, his penetrating insights, and 
his unique approach to combinatorial and representation-theoretic
problems made him one of the leaders of his generation of
researchers. 
His deep and original ideas influenced, and will continue to inspire,
a great many mathematicians, ourselves included.

\pagebreak[3]

\usection{Preliminaries}


\section{Tensors}
\label{sec:tensors}


Let $V$ be a $k$-dimensional complex vector space.
The word \emph{vector} will always refer to an element of~$V$. 
The elements of the dual space $V^*$ (the linear forms on~$V$) are 
called \emph{covectors}. 
A \emph{tensor} $T$ of \emph{type} $(a,b)$ is an element of the tensor product 
\[
\underbrace{V\otimes \cdots \otimes V}_\text{$a$ copies} \otimes 
\underbrace{V^*\otimes \cdots \otimes V^*}_\text{$b$ copies}\,, 
\]
or alternatively a multilinear map 
\begin{equation}
\label{eq:multilinear-map}
T: \underbrace{V^*\times \cdots \times V^*}_\text{$a$ copies} \times 
\underbrace{V\times \cdots \times V}_\text{$b$ copies}
\longrightarrow \CC. 
\end{equation}
In this paper, we consistently use the latter viewpoint.  
We will also need a slight notational variation 
that allows an arbitrary ordering 
of the \emph{contravariant} arguments of $T$ (vectors in~$V$) 
and its \emph{covariant} arguments (covectors in~$V^*$). 

The simplest examples of tensors are vectors (linear forms on $V^*$)
and covectors (linear forms on~$V$); 
they have types $(1,0)$ and $(0,1)$, respectively.
Tensors of type~$(1,1)$ are cryptomorphic to linear operators in
$\operatorname{End(V)}$.
Indeed, a tensor $T:V^*\times V\to\CC$ corresponds to a linear map 
$V\to (V^*)^*\cong V$ that sends $v\in V$ 
to the linear form on $V^*$ given by $u^*\mapsto T(u^*,v)$. 
The \emph{identity tensor} $I$ is the type $(1,1)$ tensor that corresponds 
to the identity operator on~$V$. 

One can accordingly define the \emph{trace} of a type $(1,1)$ tensor, 
denoted by $\operatorname{Tr}(T)$. 
For example, $\operatorname{Tr}(I)=\dim(V)=k$. 

The $k$th exterior power of $V^*$ is one-dimensional. 
Let us fix an element in $\wedge^k V^*$, 
a \emph{volume form} on~$V$. 
The corresponding \emph{volume tensor} $\vol:V^k\to\CC$ has type $(0,k)$. 
Its evaluation $\vol(v_1,\dots,v_k)$ 
is the oriented volume of the parallelotope with sides $v_1,\dots,v_k$. 
The \emph{dual volume tensor} $\vol^*$ of type $(k,0)$ is defined by the
identity
\begin{equation}
\label{eq:two-volume-forms}
\vol(v_1,\dots,v_k) \vol^*(u^*_1,\dots,u^*_k)
=\det(u^*_j(v_i))_{\substack{1\le i\le k\\ 1\le j\le k}}
\,.
\end{equation}

Let $T$ and $U$ be tensors of types $(a,b)$ and $(c,d)$, respectively.
The \emph{tensor product} $T\otimes U$ is a tensor of type $(a+c,b+d)$ 
defined by 
\begin{align*}
&(T\otimes U)(u_1^*,\dots,u_{a+c}^*; v_1,\dots,v_{b+d})\\
&=T(u_1^*,\dots,u_a^*; v_1,\dots,v_b)
\,\cdot\, U(u_{a+1}^*,\dots,u_{a+c}^*; v_{b+1},\dots,v_{b+d}). 
\end{align*}
In some instances, the arguments of $T$ and $U$ are re-shuffled in the process,
resulting in different variants of the notion of a tensor product. 

We will also need the concept of a \emph{contraction} of a tensor with respect
to a pair of arguments of opposite variance, i.e., a vector argument and a
covector argument. 
For a tensor $T$ of type $(a,b)$ given by~\eqref{eq:multilinear-map},
its contraction with respect to, say, 
the first covariant and the first contravariant arguments is the type
$(a-1,b-1)$ 
tensor $T'$ obtained by viewing $T$ as the function of those two arguments
(temporarily fixing all the rest) and evaluating the trace of the resulting 
tensor: 
\[
T'(v_2^*,\dots,v_a^*;u_2,\dots,u_b)
=\operatorname{Tr}(T(\bullet,v_2^*,\dots,v_a^*;\bullet,u_2,\dots,u_b)).
\]
One can also define a contraction of two tensors with respect to 
a contravariant argument of one of them and a covariant argument of another.
The result is obtained by first computing the tensor product
of the two tensors, then contracting as above. 

\begin{example}
\ \\
\textbf{1.}
Contracting a vector $v$ and a covector $u^*$ produces a scalar $u^*(v)$. 
More generally, an $(a+b)$-fold contraction of a type $(a,b)$ tensor~$T$
against 
an ordered collection of $a$ covectors and $b$ vectors returns the evaluation
of $T$ at this collection. 
\\
\textbf{2.}
Contracting the two arguments of a type $(1,1)$ tensor against each other
yields 
the trace of the corresponding linear operator. \\
\textbf{3.}
Given two tensors of type $(1,1)$, contracting a contravariant argument 
of one tensor against the covariant argument of the other tensor corresponds to 
composition of respective linear operators. \\
\textbf{4.}
Two-fold contraction of two such tensors (matching each contravariant argument
to the covariant argument of the other tensor) yields the trace of the product
of the corresponding linear operators. \\
\textbf{5.}
A $k$-fold contraction of the volume tensor against its dual yields the
scalar~$k!$. 
\end{example}

All of the above can be recast in the basis-dependent 
language of coordinates and matrices, once we designate  
a ``standard'' basis $(\ee_1,\dots,\ee_k)$ in~$V$ satisfying
\[
\vol(\ee_1,\dots,\ee_k)=1. 
\]
(Note that in view of~\eqref{eq:two-volume-forms}, the dual basis
$(\ee^*_1,\dots,\ee^*_k)$ in~$V^*$ has the similar property:
\[
\vol^*(\ee^*_1,\dots,\ee^*_k)=1.)
\]
After a standard basis has been chosen, 
tensors of type $(a,b)$ become $(a+b)$-dimensional arrays of complex scalars
(the \emph{components} of a tensor) 
indexed by tuples of $a$ ``row indices'' 
and $b$ ``column indices;'' 
each index takes values from 1 to~$k$. 
In particular, vectors become column vectors (with $k$ components),
covectors become row vectors,
tensors of type~$(1,1)$ become $k\times k$ matrices, and so on. 
The trace becomes the usual trace of a matrix.
The evaluation of a volume tensor at a $k$-tuple of column vectors
is given by the determinant of a square matrix formed by placing them
side-by-side
in the given order. 
The contraction of a tensor with respect to a row index and a column index is
obtained
by summing over all assignments of matching values to these two indices. 

In coordinate notation described above, the identity tensor is given by 
the identity matrix $(\delta_{ij})$, also known as 
the \emph{Kronecker symbol}.  
In \emph{Einstein notation}, commonly used in physics, the row indices are
placed as superscripts while column indices appear as subscripts; 
one then writes the Kronecker symbol as~$\delta^i_j\,$. 

The volume tensor on $V$ is given in the coordinates associated with the
standard
basis as the \emph{Levi-Civita symbol} $(\varepsilon_{i_1,\dots,i_k})$ 
defined by
\[
 \varepsilon_{i_1,\dots,i_k}
=\begin{cases}
  \text{the sign of the permutation $(i_1,\dots,i_k)$} & \text{if 
$i_1,\dots,i_k$ are all distinct;}\\
0 & \text{otherwise. }
 \end{cases}
\]
The dual volume tensor on $V^*$ is given by the same symbol
(denoted $(\varepsilon^{i_1,\dots,i_k})$ in Einstein notation). 

\section{
$\SL(V)$ invariants}
\label{sec:rings-of-invariants}

The special linear group $\SL(V)$ naturally acts 
on both~$V$ and~$V^*$;
to make the latter a left action, one sets $(gu^*)(v)=u^*(g^{-1}(v))$, 
for $v\in V$, $u^*\in V^*$, and $g\in\SL(V)$. 
The group $\SL(V)$ consequently acts on the vector space 
\begin{equation}
\label{eq:product-of-spaces}
(V^*)^a\times V^b
=\underbrace{V^*\times \cdots \times V^*}_\text{$a$ copies} \times 
\underbrace{V\times \cdots \times V}_\text{$b$ copies}\,, 
\end{equation}
and therefore on its coordinate (polynomial) ring. 

In this paper, we focus our attention 
on the ring 
\[
R_{a,b}(V)=\CC[(V^*)^a\times V^b]^{\SL(V)}
\]
of $\SL(V)$-invariant polynomials on $(V^*)^a\times V^b$. 
In coordinate notation, this ring is described as follows.
Consider two matrices filled with indeterminates:
the matrix $y\!=\!(y_{ij})$ of size $a\!\times\! k$ and
the matrix $x\!=\!(x_{ij})$ of size $k\!\times\! b$. 
The ring $R_{a,b}(V)$ consists of polynomials in these $(a\!+\!b)k$
variables which are invariant under the transformation that
simultaneously replaces $x$ by $gx$, and $y$ by $yg^{-1}$,
for any matrix $g\in\SL_k(\CC)$. 
One example of such a polynomial is a \emph{Pl\"ucker coordinate}~$P_J$
where $J$ is a $k$-element subset of columns in~$x$;
by definition, $P_J$ is the $k\times k$ minor of~$x$ obtained by
selecting the columns in~$J$. 
(In coordinate-free terms, we evaluate the volume tensor
on the $k$-element subset, labeled by~$J$, 
of the last $b$ factors in~\eqref{eq:product-of-spaces}.)
One similarly defines the ``dual'' Pl\"ucker coordinate~$P_I^*$,
for $I$ a $k$-element subset of rows,
as the corresponding maximal minor of~$y$. 
Another example of an $\SL(V)$ invariant is a bilinear
form $Q_{ij}=\sum_\ell y_{i \ell} x_{\ell j} $  
where $i$ is a row index for~$y$, and $j$ a column index for~$x$;
thus, this invariant is obtained by pairing a particular
$V^*$ factor in~\eqref{eq:product-of-spaces} with a~$V$ factor. 

The ring $R_{a,b}(V)$ is one of the archetypal objects of classical 
invariant theory; see, e.g., \cite[Chapter~2]{dolgachev}, 
\cite{kraft-procesi, li, olver}, 
\hbox{\cite[\S9]{popov-vinberg}},
\cite[Section~11]{procesi-textbook},
\cite{stur, weyl}. 
It was explicitly described by Hermann Weyl~\cite{weyl} 
in the 1930s in terms of generators and relations. 
The First Fundamental Theorem of invariant theory states that 
the ring $R_{a,b}(V)$ is generated by the following multilinear
polynomials (tensors):
\begin{itemize}
\item
the $\binom{a}{k}$ (``dual'') Pl\"ucker coordinates~$P_I^*$, 
\item
the $\binom{b}{k}$ Pl\"ucker coordinates~$P_J$, and 
\item
the $ab$ pairings $Q_{ij}$\,. 
\end{itemize}
Thus, in a sense, all $\SL(V)$ invariants are
obtained from three fundamental ones:
the volume tensors on $V$ and~$V^*$, and the identity tensor. 

Weyl also described the ideal of relations among the generators
$P_I^*$, $P_J$, and~$Q_{ij}$. 
His Second Fundamental Theorem, which we will not rely upon,
states that this ideal is generated by the quadratic 
\emph{Grassmann-Pl\"ucker relations} for the~$P_I^*$'s, 
the similar relations for the~$P_J$'s, 
and the non-homogeneous (for $k\neq 2$) relations 
$
P_I^* P_J=\det(Q_{ij})_{i\in I, j\in J}
$
(cf.~\eqref{eq:two-volume-forms}).

The First Fundamental Theorem implies that for $b \geq k$ the ring $R_{0,b}(V)$ 
is isomorphic to
the homogeneous coordinate ring $\CC[\operatorname{Gr}_{k,b}]$ of the 
Grassmann manifold $\operatorname{Gr}_{k,b}$ of $k$-dimensional subspaces
in~$\CC^b$ with respect to its Pl\"ucker
embedding (see, e.g., \cite[Corollary~2.3]{dolgachev}). 
It is not hard to show that $R_{1,b}(V)$ is isomorphic to the
  homogeneous coordinate ring of the two-step partial flag manifold 
\[
\{
(V_1,V_k) : V_1\in\operatorname{Gr}_{1,b}, V_k\in\operatorname{Gr}_{k,b}
\}. 
\]
We do not know a similar interpretation of $R_{a,b}(V)$ for
  general values of $a$ and~$b$. 

In our discussions of the spaces $R_{a,b}(V)$,
it will be important to distinguish between their incarnations that
involve different orderings of the contravariant and covariant
arguments. 
To this end, we will use the notion of a \emph{signature}, 
a word in the alphabet $\{\bullet,\circ\}$ that reflects the
order of the $a+b$ arguments of an invariant~$f$. 
Specifically, each of the $a$ factors $V^*$ will be represented by the
symbol~$\circ$, and each of the $b$ factors $V$ by the
symbol~$\bullet$. 
Thus for example an $SL(V)$-invariant polynomial 
\[
f: V^*\times V\times V\times V^* \to \CC
\]
is an invariant of signature $\sigma=[\circ\bullet\bullet\,\circ]$.
We denote by $R_\sigma(V)$ the ring of $\SL(V)$ invariants of
signature~$\sigma$. 
If $\sigma$ consists of $a$ copies of~$\circ$ and $b$ copies of~$\bullet$, 
we say that it has \emph{type} $(a,b)$. 
Thus for $\sigma$ of type~$(a,b)$, we have
$R_\sigma(V)\cong R_{a,b}(V)$. 

The natural action of the $(a+b)$-dimensional torus defines 
a multi-grading of the ring $R_\sigma(V)\cong R_{a,b}(V)$. 
If an invariant~$f\in R_\sigma(V)$ is multi-homogeneous of degrees 
$d_1,\dots,d_{a+b}$ in its $a+b$ arguments,
we then call the tuple
\begin{equation}
\label{eq:multideg}
\operatorname{multideg}(f)=
(d_1,\dots,d_{a+b})
\end{equation}
the \emph{multidegree} of~$f$. 
For example, 
the Pl\"ucker coordinate $P_{\{1,\dots,k\}}\in R_{0,b}(V)$ 
has multidegree $(1,\dots,1,0,\dots,0)$ (the first $k$ entries are~1's). 

There are known combinatorial formulas 
for the dimensions of multigraded components of the ring~$R_{a,b}(V)$
in terms of the iterated 
Littlewood-Richardson Rule. 
In some cases, more explicit formulas can be given. 
For example, the multilinear component of $R_{0,b}(V)$
(i.e., the space of $\SL_k$-invariant $b$-covariant 
tensors in~$\CC^k$) is only nontrivial when $k$ divides~$b$;
if it does, then the dimension of this space is equal to the number of standard Young
tableaux of rectangular shape $k\times \frac{b}{k}$. 
As such, this dimension is given by the hooklength formula. 

Being a subring of a polynomial ring, $R_{a,b}(V)$ is a domain. 
By~(the modern version of) Hilbert's theorem, 
it is finitely generated; as mentioned above, 
H.~Weyl~\cite{weyl} described the generators explicitly.  
Finally,  the ring $R_{a,b}(V)$ is factorial: 

\begin{lemma}[see \cite{popov-vinberg}, Theorem~3.17]
\label{lem:properties-of-RabV}
The algebra of invariants $R_{a,b}(V)$ is a finitely
generated unique factorization domain. 
\end{lemma}

\pagebreak[3]

\section{Cluster algebras
}
\label{sec:cluster-algebras}

This section offers a quick 
introduction to cluster
algebras, limited in scope to our immediate needs.
Further details 
can be found
in~\cite{ca1, ca2, ca4}. 
The only non-standard material in this section is 
Proposition~\ref{prop:cluster-criterion} 
(and its Corollary~\ref{cor:cluster-criterion})
which provides a test for verifying that
a given set of algebraic/combinatorial data inside a commutative ring 
makes the latter into a cluster algebra. 

Cluster algebras are a class of commutative rings endowed with a 
combinatorial structure of a particular kind.
The concept can be defined in varying degrees of \linebreak[3]
generality. In this paper, 
we will only need the notion of a cluster algebra which 
(i)~is of \emph{geometric type}, (ii)~has $\CC$ as its field of scalars,
and (iii)~has \emph{skew-symmetric exchange matrices}. 
To simplify terminology, we call such gadgets ``cluster algebras''
without further qualifications. 

The combinatorial data defining a cluster algebra are encoded in a
\emph{quiver}~$Q$, a finite oriented loopless graph
with no oriented 2-cycles. 
Some vertices of $Q$ are designated as
\emph{mutable}; 
the remaining ones are called \emph{frozen}. 

\begin{definition} 

Let $z$ be a mutable vertex in a quiver~$Q$.
The \emph{quiver mutation} $\mu_z$ transforms $Q$ into the new
quiver~$Q'=\mu_z(Q)$ via a sequence of three steps.
At the first step, for each pair of directed edges $x\to z\to y$
passing through~$z$, introduce a new edge $x\to y$ (unless both
$x$ and~$y$ are frozen, in which case do nothing). 
At the second step, reverse the direction of all edges incident
to~$z$. 
At the third step, repeatedly remove oriented 2-cycles
until unable to do~so. See Figure~\ref{fig:quiver-mutation}. 
\end{definition}

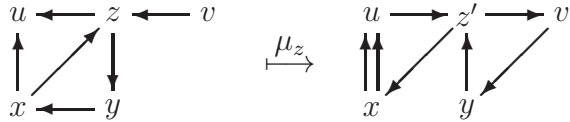
\begin{figure}[ht]
\begin{center}
\setlength{\unitlength}{1.8pt} 
\begin{picture}(40,15)(0,0) 

\put( 0,0){\makebox(0,0){$x$}}
\put(20,0){\makebox(0,0){$y$}}

\put(0,20){\makebox(0,0){$u$}}
\put(20,20){\makebox(0,0){$z$}}
\put(40,20){\makebox(0,0){$v$}}

\thicklines 

\put(16,0){\vector(-1,0){12}}
\put(16,20){\vector(-1,0){12}}
\put(36,20){\vector(-1,0){12}}

\put(20,16){\vector(0,-1){12}}
\put(0,4){\vector(0,1){12}}

\put(3,3){\vector(1,1){14}}

\end{picture}
\begin{picture}(30,15)(0,0) 
\put(15,12){\makebox(0,0){$\stackrel{\displaystyle\mu_z}{\longmapsto}$}}
\end{picture}
\begin{picture}(40,15)(0,0) 

\put( 0,0){\makebox(0,0){$x$}}
\put(20,0){\makebox(0,0){$y$}}

\put(0,20){\makebox(0,0){$u$}}
\put(20,20){\makebox(0,0){$z'$}}
\put(40,20){\makebox(0,0){$v$}}

\thicklines

\put(4,20){\vector(1,0){12}}
\put(24,20){\vector(1,0){12}}

\put(-1.3,4){\vector(0,1){12}}
\put(1.3,4){\vector(0,1){12}}
\put(20,4){\vector(0,1){12}}

\put(17,17){\vector(-1,-1){14}}
\put(37,17){\vector(-1,-1){14}}

\end{picture}
\end{center}
\caption{A quiver mutation. Vertices $u$ and $v$ are frozen.}
\label{fig:quiver-mutation}
\end{figure}


Quiver mutations can be iterated \emph{ad infinitum}, 
using an arbitrary sequence of mutable vertices of an evolving
quiver. 
This combinatorial dynamics drives the algebraic
dynamics of \emph{seed mutations} that we describe next. 

\begin{definition} 
\label{def:seeds}

Let~$\Fcal$ be a field containing~$\CC$. 
A \emph{seed} in~$\Fcal$ is a pair $(Q,\zz)$ 
consisting of a quiver~$Q$ as above together with 
a collection~$\zz$, called an \emph{extended cluster}, 
consisting of algebraically independent (over~$\CC$) 
elements of~$\Fcal$, one for each vertex of~$Q$. 
The elements of~$\zz$ associated with the mutable
vertices are called \emph{cluster variables}; they form a
\emph{cluster}. 
The elements associated with the frozen vertices are called
\emph{frozen variables}, or \emph{coefficient variables}.

A~\emph{seed mutation}~$\mu_z$ at a mutable vertex associated with 
a cluster variable~$z$
transforms $(Q,\zz)$ into the seed $(Q',\zz')=\mu_z(Q,\zz)$ defined as
follows. 
The new quiver is $Q'=\mu_z(Q).$ 
The new extended cluster is 
$\zz'\!=\!\zz\cup\{z'\}\setminus\{z\}$ 
where the new cluster variable~$z'$ replacing $z$ is determined by the 
\emph{exchange relation}
\begin{equation}
\label{eq:exchange-relation}
z\,z'=
\prod_{z\leftarrow y} y
+\prod_{z\rightarrow y
} y\,.
\end{equation}
(The two products are over the edges directed at and from~$z$, respectively.) 
\end{definition}

To illustrate, the exchange relation associated with the quiver mutation
$\mu_z$ shown in Figure~\ref{fig:quiver-mutation} is $zz'=vx+uy$, 
while applying $\mu_x$ to the quiver on the right would invoke the
exchange relation~$xx'=z'+u^2$. 

We note that the mutated seed $(Q',\zz')$ contains the same
coefficient variables as the original seed $(Q,\zz)$. 

It is easy to check that one can recover $(Q,\zz)$
from $(Q',\zz')$ by performing a seed mutation at~$z'$. 

Definition~\ref{def:seeds} deviates from the usual
convention of using the field of rational functions~$\CC(\zz)$
as the ambient field~$\Fcal$.
This distinction is inconsequential since, as we shall see,
all the action is taking place inside~$\CC(\zz)$. 

\begin{definition}[\emph{Cluster algebra}]
\label{def:cluster-algebra}
Two seeds that can be obtained from each other by a sequence of
mutations are called \emph{mutation equivalent}. 
The \emph{cluster algebra $\Acal(Q,\zz)$} associated to a seed
$(Q,\zz)$ is defined as the subring of $\Fcal$
generated by 
all elements of all extended clusters of the seeds mutation equivalent
to~$(Q,\zz)$. 
%
\end{definition}

Thus, to construct a cluster algebra, one begins with an arbitrary 
\emph{initial seed} $(Q,\zz)$ in~$\Fcal$, repeatedly applies
seed mutations in all possible directions, and takes the $\CC$-subalgebra
generated by all elements of~$\Fcal$ appearing in all seeds produced by this
recursive process. 

Note that the above construction of $\Acal(Q,\zz)$ does not depend, up
to a natural isomorphism, on the choice of the extended cluster~$\zz$:
everything, save for the embedding of the resulting cluster
algebra into~$\Fcal$, is determined by the initial quiver~$Q$,
and indeed by its mutation equivalence class. 

The \emph{rank} of a cluster algebra $\Acal(Q,\zz)$ 
is the number of cluster variables in each of its seeds, 
or equivalently the number of mutable vertices in each of its
quivers. 

When a quiver $Q$ undergoes a mutation, 
its mutable part (that is, the induced subquiver on the set of mutable
vertices) does so, too. 
Thus the mutable parts of the quivers at different seeds of a given
cluster algebra~$\Acal$ are all mutation equivalent. 
This mutation equivalence class determines the (cluster) \emph{type}
of~$\Acal$. 

The same ring can be endowed with cluster structures of different
type; see, e.g., \cite[Example~12.10]{ca2}.
Many more examples of this kind appear later in this paper. 

A cluster algebra is said to be of \emph{finite type} if it has
finitely many distinct seeds. 
One of the basic structural results of cluster theory is the 
  \emph{finite type classification}: 

\begin{theorem}
[\textrm{\!\!\cite{ca2}}]
\label{th:finite-type}
A cluster algebra is of finite type if and only if it
has a quiver whose mutable part is an orientation of a disjoint union
of Dynkin diagrams. 
\end{theorem}

Thus the terminology is consistent: 
the property of being of finite type depends only on the
cluster type of a cluster algebra. 

If a cluster algebra~$\Acal$ (say of rank~$n$) 
has a quiver whose mutable part is an orientation of a 
Dynkin diagram~$X_n$, then we say that $\Acal$ is ``of type~$X_n$.'' 
This convention applies as well to the extended Dynkin diagrams of 
affine and elliptic (or extended affine) types; 
see, e.g., \cite[Section~12]{cats1} for further details. 
It is worth noting that all orientations of any such diagram $X_n$ 
are mutation equivalent whenever $X_n$ is a tree (thus in particular
in finite type). 
Furthermore, orientations of non-isomorphic 
diagrams are mutation inequivalent. 

Another basic property of cluster algebras is the 
\emph{Laurent Phenomenon}: 

\begin{theorem}
[\textrm{\!\!\cite{ca1, ca2}}]
\label{th:laurent}
Any element of a cluster algebra $\Acal(Q,\zz)$ is expressed in terms of the
extended cluster~$\zz$ as a Laurent polynomial. 
Furthermore, none of the coefficient variables appears in the
denominator of this Laurent expression (reduced to lowest terms). 
\end{theorem}

Since $\Acal(Q,\zz)$ is generated by cluster variables from the seeds
mutation equivalent to $(Q,\zz)$, 
Theorem~\ref{th:laurent} can be restated as saying that each of those
cluster variables is given by a Laurent expression of the
aforementioned kind. 

The Laurentness part of Theorem~\ref{th:laurent} is a special case
of \cite[Theorem~3.1]{ca1}. 
The last statement of Theorem~\ref{th:laurent} 
was established in \cite[Proposition~11.2]{ca2},
modulo a technical condition \cite[(11.3)]{ca2}
which is in fact satisfied in all applications appearing in this
paper.
This condition can be removed using 
\cite[Theorem~3.7]{ca4} in combination with 
\cite[Conjecture~5.4]{ca4}; the latter was proved (in the generality
considered in the current paper) in \cite[Theorem~1.7]{dwz}. 

Many important rings arising in Lie theory 
possess a natural structure of a cluster algebra. 
The list includes homogeneous coordinate rings of 
partial flag manifolds, Schubert varieties,
and double Bruhat cells in semisimple Lie groups;  
see, e.g., 
\cite{ca3, cdm, gls-survey, keller-bourbaki,leclerc-icm}.

Of particular importance to us is the example of 
a Grassmannian~$\operatorname{Gr}_{k,b}$.
As mentioned earlier, its homogeneous
coordinate ring (with respect to the Pl\"ucker embedding) is canonically
isomorphic to the ring $R_{0,b}(\CC^k)$ 
of $\SL_k$-invariants of configurations of $b$ vectors in a
$k$-dimensional vector space. 
This ring has a natural cluster algebra structure,
first defined in J.~Scott's Ph.D.\ thesis~\cite{scott} 
(cf.\ also~\cite{gsv1, gsv-book}).
It is important to note that this cluster structure 
on $R_{0,b}(V)$ depends on the choice of 
a cyclic ordering of the $b$ vectors;
choosing a different ordering results in a different (albeit
isomorphic) cluster structure. 

\medskip

Any cluster algebra, being a subring of a field, is an integral
domain (and under our conventions, a \hbox{$\CC$-algebra}). 
Conversely, given such a domain~$R$, one may be interested in
identifying $R$ as a cluster algebra. 
As an ambient field~$\Fcal$, 
we can always use the quotient field~$\QF(R)$. 
The challenge is to find a seed $(Q,\zz)$ in $\QF(R)$ such
that $\Acal(Q,\zz)=R$. 
We next present a set of conditions that are sufficient to ensure
that a particular choice of a seed solves this
problem. 


Recall that an integral domain $R$ is called \emph{normal} if it is
integrally closed in~$\QF(R)$. 

\begin{proposition}
\label{prop:cluster-criterion}
Let $R$ be a finitely generated $\CC$-algebra and a normal domain.
Let~$(Q,\zz)$ be a seed in $\QF(R)$ satisfying the following
conditions:
\begin{itemize}
\item
all elements of $\zz$ belong to~$R$; 
\item
the cluster variables in $\zz$ are pairwise coprime (in~$R$); 
\item
for each cluster variable $z\in\zz$, the seed mutation~$\mu_z$
replaces $z$ with an element~$z'$ 
(cf.~\eqref{eq:exchange-relation})
that lies in $R$ and is coprime to~$z$. 
\end{itemize}
(Here we call two elements of $R$ coprime if the locus of their common zeros 
has codimension~$\ge 2$ in $\operatorname{Spec}(R)$.)
Then $R\supset\Acal(Q,\zz)$.

If, in addition, $R$ has a set of generators each of which 
appears in the seeds mutation equivalent to $(Q,\zz)$, then 
$R=\Acal(Q,\zz)$. 
\end{proposition}

Proposition~\ref{prop:cluster-criterion} (and 
its proof given below) extends some of the ideas used in \linebreak[3]
\cite[Proof of Theorem~2.10]{ca3},
and subsequently in \cite[Proof of Proposition~7]{scott}. 

The argument given below actually establishes a stronger statement:
under the conditions of Proposition~\ref{prop:cluster-criterion},
the ring~$R$ contains the \emph{upper cluster algebra}
$\overline\Acal(Q,\zz)$ (see~\cite{ca3}),
or more precisely the subalgebra of  $\QF(R)$ consisting of the elements
which, when expressed in terms of any extended cluster,
are Laurent polynomials in the cluster variables and 
ordinary polynomials in the coefficient variables. 
In particular, the conditions of
Proposition~\ref{prop:cluster-criterion} imply that in this case,
the upper cluster algebra coincides with the ordinary cluster algebra
(localized at coefficient variables). 

\begin{proof}
The last implication is clear since its premise 
ensures that $R\subset\Acal(Q,\zz)$. 

To prove that $R\supset\Acal(Q,\zz)$,
we need to show that each cluster
variable $\bar z$ from any seed $(\bar Q,\bar \zz)$ 
mutation equivalent to $(Q,\zz)$ 
belongs to~$R$. 
This is done using some basic algebraic geometry.
Our assumptions on the ring $R$ mean that 
it can be identified with the coordinate ring of an
(irreducible) normal affine complex 
algebraic variety $X\!=\!\operatorname{Spec}(R)$. 
Then $\QF(R)\!=\!\CC(X)$, the field of rational functions on~$X$. 
Our goal is to show that each ``distant'' cluster variable~$\bar z$ as
above is not just a rational function in $\CC(X)$ but a regular
function in~$\CC[X]\!=\!R$. 
The key property that we need is the algebraic version of
Hartogs' continuation principle for normal varieties 
(see, e.g., \cite[Chapter~2, 7.1]{danilov}) which asserts that a
function on~$X$ that is regular outside a closed algebraic subset
of codimension~$\ge 2$ is in fact regular everywhere on~$X$. 

Let $Y\subset X$ be the locus where at least two of the cluster
variables in $\zz$ vanish, or else one of them, say~$z$, vanishes
together with the element~$z'$ defined by~\eqref{eq:exchange-relation}. 
By the coprimeness conditions imposed on the seed~$(Q,\zz)$, 
the codimension of $Y$ in~$X$ is~$2$. \linebreak[3]
The complement $X\setminus Y$ consists of the points $x\in X$ such
that
\begin{itemize}
\item
at most one of the cluster variables in $\zz$ vanishes at~$x$, and
\item
for each pair $(z,z')$ as above, either $z$ or $z'$ does not vanish
at~$x$. 
\end{itemize}
This means that there is a seed $(Q',\zz')$ 
(either the original seed $(Q,\zz)$ or one of the adjacent seeds 
$\mu_z(Q,\zz)$) none of whose cluster variables vanishes at~$x$;  
moreover $\zz'\subset\CC[X]$. 
Then the Laurent Phenomenon (Theorem~\ref{th:laurent}) implies that 
our distant cluster variable $\bar z$ is regular at~$x$. 
By the algebraic Hartogs' principle mentioned above, 
it follows that $\bar z$ is regular on~$X$, and
we are done. 
\end{proof}

\begin{corollary}
\label{cor:cluster-criterion}
Let $R$ be a finitely generated unique factorization domain over~$\CC$.
Let~$(Q,\zz)$ be a seed in the quotient field of $R$ 
such that 
all elements of $\zz$ and all elements of clusters adjacent to~$\zz$ 
are irreducible elements of~$R$. 
Then $R\supset\Acal(Q,\zz)$. 

If, in addition, the union of extended clusters in $\Acal(Q,\zz)$ 
contains a generating set for~$R$, then $R=\Acal(Q,\zz)$. 
\end{corollary}

\begin{proof}
Let us check the conditions of Proposition~\ref{prop:cluster-criterion}. 
Any UFD is a normal domain. 
To verify the coprimality conditions,
we first note that two elements of~$\zz$
cannot differ by a scalar factor since they are algebraically independent. 
Similarly, the elements of an exchange pair $(z,z')$ cannot differ
by a scalar factor, or else the exchange relation would give 
an algebraic dependence in~$\zz$. 
Now, an irreducible element in a UFD generates a prime ideal,
so the vanishing locus in $\operatorname{Spec}(R)$ for each of our 
irreducible elements is an (irreducible) subvariety. 
These subvarieties are pairwise distinct (cf.\ above),
so their pairwise intersections have codimensions~$\ge 2$. 
\end{proof}

In view of Lemma~\ref{lem:properties-of-RabV}, 
in order to identify a cluster algebra structure in 
the ring of invariants~$R_{a,b}(V)$, all we need to do is 
find a seed $(Q,\zz)$ satisfying the conditions
in Corollary~\ref{cor:cluster-criterion}. 
This is a lot easier said than done.



\section{Tensor diagrams}
\label{sec:tensor-diagrams}

{}From now on, we assume that $k=3$, so that $V\cong\CC^3$. 

Our description of cluster structures in the classical rings of invariants
$R_{a,b}(V)$ will be based on the calculus
of tensor diagrams~\cite{blinn, cvitanovic, 
richter-gebert-lebmeir,
  stedman},
particularly in its version of the ``$A_2$~spider''
developed by G.~Kuperberg~\cite{kuperberg}. 
This calculus, in turn, is based on two simple observations.
The first one was already mentioned in connection with the First Fundamental
Theorem:
\begin{itemize}
\item[(i)]
the identity tensor and 
the volume tensors on $V$ and $V^*$ are $\SL(V)$-invariant.
\end{itemize}
The second observation is that 
\begin{itemize}
\item[(ii)]
the operations of tensor product 
and contraction preserve $\SL(V)$ invariance. 
\end{itemize}
By combining the ingredients derived from (i) using the operations 
in~(ii), 
one can build a surprisingly rich supply of $\SL(V)$-invariant tensors--and 
from them, a large class of polynomial invariants. 
We next explain how, without claiming any originality: 
the machinery described below is a 
variation of diagrammatic calculus whose various versions were (re)discovered
multiple times, under different guises and names, see in particular \cite{blinn,
cvitanovic, kuperberg} and especially \cite[Section~8]{abdesselam-chipalkatti}. 

Tensor diagrams are built using three
types of building blocks shown in Figure~\ref{fig:webs21}. 
These blocks correspond, left to right, 
to the three basic $\SL(V)$-invariant tensors:
the volume tensor, the dual volume tensor, and the identity tensor.
The endpoints of a block correspond to the tensor's arguments:
a sink to a vector, a source to a covector. 
In order for either of the volume tensors to be well defined,
a cyclic ordering of the three arguments must be specified.

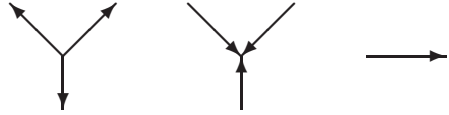
\begin{figure}[ht]
    \begin{center}
\setlength{\unitlength}{2pt}
\begin{picture}(20,20)(0,0)
\thicklines
\put(10,10){\vector(1,1){10}}
\put(10,10){\vector(-1,1){10}}
\put(10,10){\vector(0,-1){10}}
\end{picture}
\qquad
\begin{picture}(20,20)(0,0)
\thicklines
\put(0,20){\vector(1,-1){10}}
\put(20,20){\vector(-1,-1){10}}
\put(10,0){\vector(0,1){10}}
\end{picture}
\qquad
\begin{picture}(15,20)(0,0)
\thicklines
\put(0,10){\vector(1,0){15}}
\end{picture}
\qquad
    \end{center} 
    \caption{Basic building blocks for tensor diagrams.
The arrows are to be thought of as flexible
strings which can be bent and/or stretched.  
}
    \label{fig:webs21}
\end{figure}

We then combine these building blocks by
plugging arrowheads into arrowtails. 
Interpreting this operation as contraction of the corresponding tensors,
we obtain a well defined notion of a tensor associated with the resulting
combinatorial gadget. 

Note that plugging one end of a two-pronged block
(corresponding to the identity tensor) into an end of opposite valence 
in some tensor diagram does not change the tensor that the latter represents. 
Consequently, for any diagram assembled using the above procedure from
the three 
basic types of blocks, the associated tensor does not depend on 
the locations (or number) of plug-in joints along unbranched stretches.  
Erasing those joints produces a directed graph with trivalent and
univalent vertices in which every trivalent vertex is a source or a sink. 
Any such graph, together with the additional data specifying cyclic
ordering of edges incident to trivalent vertices, represents a well defined
tensor.

\pagebreak[3]

We are going to draw all our diagrams in an oriented disk, 
obeying the following conventions. 
We will usually erase the arrows, introducing instead a
bi-coloring of the vertices:
the sinks will be colored black, and the sources white. 
We place the boundary vertices on the boundary of the disk, 
in the order mirroring the one used to define the direct product
of the $a+b$ spaces $V$ and~$V^*$ whose $\SL(V)$ invariants
we study. 
We place the interior vertices inside the disk,
and draw the edges as simple curves 
that are allowed to intersect transversally. 
Last but not least,
we make sure that the cyclic ordering specified at each interior
vertex of the diagram 
matches the clockwise ordering of the three curves 
meeting at the corresponding point in the disk. 

The last step in the construction of a general tensor diagram involves 
gluing together some of the univalent vertices of a directed graph as above. 
We interpret this step as substituting the same (co)vector into
the arguments of the associated tensor which correspond to the glued
endpoints. 
At this point, the $\SL(V)$ invariant  represented by the graph ceases to be
multilinear. 
This last operation is sometimes called (partial) \emph{restitution}, and is
inverse to \emph{polarization}. 

An example illustrating different stages of the construction outlined above is
shown in Figure \ref{fig:webs22}. 

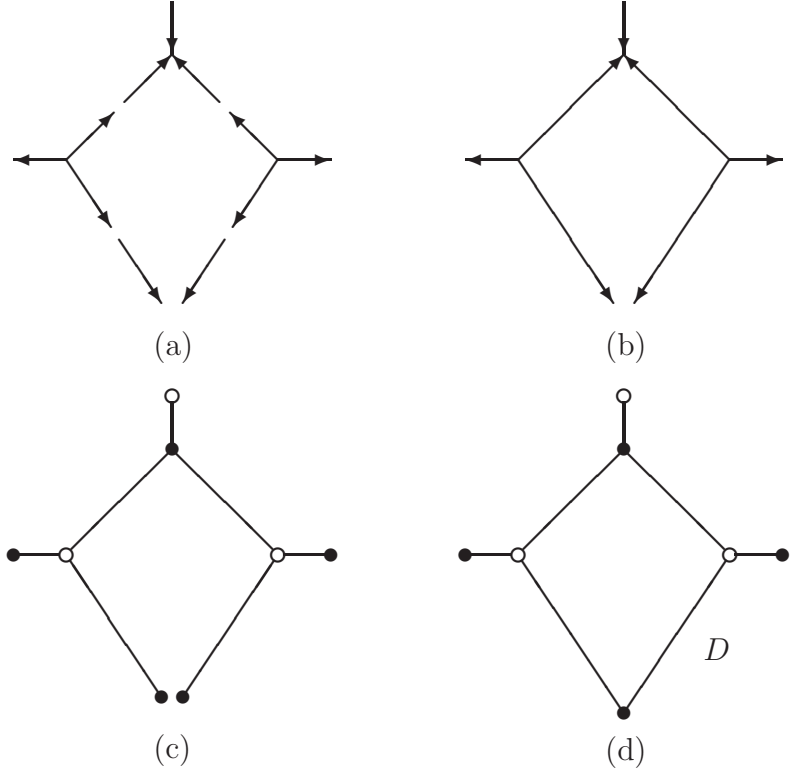
\begin{figure}[ht]
    \begin{center}
\setlength{\unitlength}{2pt}
\begin{picture}(60,67)(0,-7)
\thicklines
\put(10,30){\vector(-1,0){10}}
\put(10,30){\vector(1,1){9}}
\put(10,30){\vector(2,-3){8.5}}
\put(20,15){\vector(2,-3){8}}

\put(50,30){\vector(1,0){10}}
\put(50,30){\vector(-1,1){9}}
\put(50,30){\vector(-2,-3){8.5}}
\put(40,15){\vector(-2,-3){8}}

\put(21,41){\vector(1,1){9}}
\put(39,41){\vector(-1,1){9}}
\put(30,60){\vector(0,-1){10}}

\put(26.5,-7){(a)}

\end{picture}
\qquad\qquad
\begin{picture}(60,67)(0,-7)
\thicklines
\put(10,30){\vector(-1,0){10}}
\put(10,30){\vector(1,1){20}}
\put(10,30){\vector(2,-3){18}}

\put(50,30){\vector(1,0){10}}
\put(50,30){\vector(-1,1){20}}
\put(50,30){\vector(-2,-3){18}}

\put(30,60){\vector(0,-1){10}}

\put(26.5,-7){(b)}

\end{picture}
\\[.2in]
\begin{picture}(60,67)(0,-7)
\thicklines
\put(0,30){\line(1,0){8.5}}
\put(30,50){\line(-1,-1){19}}
\put(30,50){\line(1,-1){19}}
\put(28,3){\line(-2,3){17.3}}

\put(60,30){\line(-1,0){9}}
\put(32,3){\line(2,3){17.3}}

\put(30,59){\line(0,-1){9}}

\put(0,30){\circle*{2.5}}
\put(60,30){\circle*{2.5}}
\put(30,50){\circle*{2.5}}
\put(28,3){\circle*{2.5}}
\put(32,3){\circle*{2.5}}

\put(10,30){\circle{2.5}}
\put(50,30){\circle{2.5}}
\put(30,60){\circle{2.5}}

\put(26.5,-8.5){(c)}

\end{picture}
\qquad\qquad
\begin{picture}(60,67)(0,-7)
\thicklines
\put(0,30){\line(1,0){8.5}}
\put(30,50){\line(-1,-1){19}}
\put(30,50){\line(1,-1){19}}
\put(30,0){\line(-2,3){19.3}}

\put(60,30){\line(-1,0){9}}
\put(30,0){\line(2,3){19.3}}

\put(30,59){\line(0,-1){9}}

\put(0,30){\circle*{2.5}}
\put(60,30){\circle*{2.5}}
\put(30,50){\circle*{2.5}}
\put(30,0){\circle*{2.5}}

\put(10,30){\circle{2.5}}
\put(50,30){\circle{2.5}}
\put(30,60){\circle{2.5}}

\put(45,10){$D$}

\put(26.5,-8.8){(d)}

\end{picture}
    \end{center} 
    \caption{Assembling a tensor diagram. 
    In this example, we use five building blocks. They correspond to two copies of the
    volume tensor, one copy of the dual volume tensor, and two copies
    of the identity tensor. The resulting tensor diagram $D$ of type
    $(1,3)$ represents an
    invariant~of multidegree $(1,2,1,1)$ in~$R_{1,3}(V)$. The disk
    containing $D$ is not~shown.}
    \label{fig:webs22}
\end{figure}

\pagebreak[3]

\begin{definition}
A \emph{tensor diagram}
is a finite bipartite graph $D$ with a fixed {\it {proper}} coloring of its vertices
into two colors, black and white, and with a fixed partition of its vertex set
into the set~$\operatorname{bd}(D)$ of \emph{boundary} vertices and 
the set~$\operatorname{int}(D)$ of \emph{internal} vertices,
satisfying the following conditions: 
\begin{itemize}
\item
each internal vertex is trivalent;
\item
for each internal vertex, a cyclic
order on the edges incident to it is fixed.
\end{itemize}
The last condition ensures that each of the six permutations of the
three edges incident to an internal vertex has a well-defined sign. 

If $\operatorname{bd}(D)$ consists of 
$a$ white vertices and $b$ black ones, 
then we say that $D$ is of \emph{type} $(a,b)$. 
\end{definition}

We regard each edge of a tensor diagram as being oriented
so that it points away from its white endpoint and towards the black
one. 
While drawing tensor diagrams, this orientation is routinely omitted. 

A tensor diagram~$D$ of type $(a,b)$
defines an $\SL(V)$ invariant~$[D]\in R_{a,b}(V)$ \linebreak[3]
obtained by
repeated contraction of elementary $\SL(V)$-invariant tensors 
(followed by bundling up some of the arguments), as explained above. 
To compensate for the informality of that explanation, we next provide a
precise definition of the invariant $[D]$ in the language of coordinates. 

Let us identify $R_{a,b}(V)$ with the ring of 
$\operatorname{SL}_3$ invariants of collections of $a$ covectors 
\[
y(v)=\begin{bmatrix}\,y_1(v)&y_2(v)&y_3(v)\,\end{bmatrix}
\]
and $b$ vectors 
\[
 x(v)=\begin{bmatrix}x_1(v)\\ x_2(v)\\ x_3(v)\end{bmatrix}
\]
labeled by a fixed collection of $a$ white and $b$ black vertices on the
boundary of a disk. 
The invariant $[D]$ associated with a tensor diagram~$D$ with those
boundary vertices is then given by
\begin{equation}
\label{eq:[D]-in-coordinates}
[D]=\sum_\ell 
\biggl(\,\prod_{v\in\operatorname{int}(D)}\operatorname{sign}(\ell(v))\biggr)
\biggl(\,\prod_{\substack{v\in\operatorname{bd}(D)\\
\text{$v$ black}}}x(v)^{\ell(v)}\biggr)
\biggl(\,\prod_{\substack{v\in\operatorname{bd}(D)\\
\text{$v$ white}}}y(v)^{\ell(v)}\biggr),
\end{equation}
where 
\begin{itemize}
\item
$\ell$ runs over all proper labelings of the edges in~$D$ by the
numbers~$1,2,3$; \linebreak[3]
to wit, we require that for each internal vertex~$v$,
the labels associated with the three edges incident to $v$ are all distinct; 
\item
$\operatorname{sign}(\ell(v))$ denotes the sign of the (cyclic) permutation  of
those three labels determined by the cyclic ordering of the edges incident
to~$v$; 
\item
$x(v)^{\ell(v)}$ denotes the monomial $\prod_e x_{\ell(e)}(v)$,
product over all edges $e$ incident to~$v$, and similarly for $y(v)^{\ell(v)}$. 
\end{itemize}


Example~\ref{example:tripod-and-stick} below illustrates
formula~\eqref{eq:[D]-in-coordinates}. 
Many more examples of tensor diagrams are scattered throughout the paper. 

\pagebreak[3]

\begin{example}
\label{example:tripod-and-stick}
Consider the tensor diagram $D$ shown 
in Figure~\ref{fig:webs22}(d). 
We label the black boundary vertices 1, 2, and~3, in clockwise order,
and label the white vertex~4. 
We denote the corresponding three column vectors and one row covector
by 
\[
\begin{bmatrix}x_{11}\\ x_{21}\\ x_{31}\end{bmatrix},{\ \ }
\begin{bmatrix}x_{12}\\ x_{22}\\ x_{32}\end{bmatrix},{\ \ }
\begin{bmatrix}x_{13}\\ x_{23}\\ x_{33}\end{bmatrix},{\ \ }
\text{and}\ \ \begin{bmatrix}\,y_1&y_2&y_3\,\end{bmatrix},
\]
respectively. 
There are 24 proper labelings of the edges of~$D$. 
They can be broken into four six-tuples as shown in
Figure~\ref{fig:edge-labelings}. 
The corresponding monomial weights are
\[
\varepsilon\,x_{\gamma 1}x_{\alpha 2}^2x_{\beta 3} y_\alpha\,,\ \ 
\varepsilon\,x_{\gamma 1}x_{\alpha 2}x_{\beta 2}x_{\beta 3} y_\beta\,,\ \ 
\varepsilon\,x_{\gamma 1}x_{\alpha 2}x_{\gamma 2}x_{\beta 3} y_\gamma\,,\ \ 
-\varepsilon\,x_{\beta 1}x_{\alpha 2}x_{\gamma 2}x_{\beta 3} y_\beta\,,
\]
where $\varepsilon$ denotes the sign of the permutation 
$(\alpha,\beta,\gamma)$ in the symmetric group~$\mathcal{S}_3$. 
The weights of the labelings in the fourth group cancel each other out
under the sign-reversing involution that switches $\alpha$
and~$\gamma$. Consequently, \eqref{eq:[D]-in-coordinates} becomes
\begin{align}
\nonumber
[D]&=\sum_{(\alpha,\beta,\gamma)\in\mathcal{S}_3}
\operatorname{sgn}(\alpha,\beta,\gamma)
x_{\gamma 1}x_{\alpha 2}x_{\beta 3} 
\,(x_{\alpha 2} y_\alpha+x_{\beta 2} y_\beta+x_{\gamma 2} y_\gamma)\\
\label{eq:det*dot-product}
&=\det(x_{ij})_{i,j\in(1,2,3)}  
(x_{12} y_1+x_{22} y_2+x_{32} y_3), 
\end{align}
coinciding with the invariant given by the tensor diagram shown
in Figure~\ref{fig:tripod-and-a-stick}. 
\end{example}

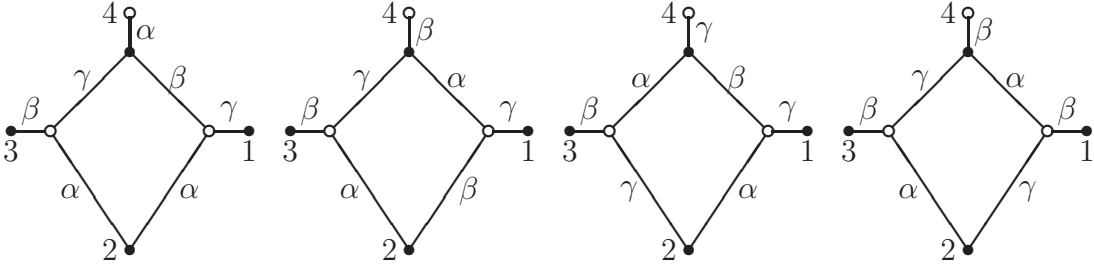
\begin{figure}[ht]
    \begin{center}
\setlength{\unitlength}{1.5pt}

\begin{picture}(60,60)(0,0)
\thicklines
\put(0,30){\line(1,0){8.5}}
\put(30,50){\line(-1,-1){19}}
\put(30,50){\line(1,-1){19}}
\put(30,0){\line(-2,3){19.3}}

\put(60,30){\line(-1,0){9}}
\put(30,0){\line(2,3){19.3}}

\put(30,59){\line(0,-1){9}}

\put(0,30){\circle*{2.5}}
\put(60,30){\circle*{2.5}}
\put(30,50){\circle*{2.5}}
\put(30,0){\circle*{2.5}}

\put(10,30){\circle{2.5}}
\put(50,30){\circle{2.5}}
\put(30,60){\circle{2.5}}

\put(0,25){\makebox(0,0){$3$}}
\put(60,25){\makebox(0,0){$1$}}
\put(25,0){\makebox(0,0){$2$}}
\put(25,60){\makebox(0,0){$4$}}

\put(15,15){\makebox(0,0){$\alpha$}}
\put(45,15){\makebox(0,0){$\alpha$}}
\put(5,35){\makebox(0,0){$\beta$}}
\put(55,35){\makebox(0,0){$\gamma$}}
\put(18,43){\makebox(0,0){$\gamma$}}
\put(42,43){\makebox(0,0){$\beta$}}
\put(34,55){\makebox(0,0){$\alpha$}}

\end{picture}
\quad
\begin{picture}(60,60)(0,0)
\thicklines
\put(0,30){\line(1,0){8.5}}
\put(30,50){\line(-1,-1){19}}
\put(30,50){\line(1,-1){19}}
\put(30,0){\line(-2,3){19.3}}

\put(60,30){\line(-1,0){9}}
\put(30,0){\line(2,3){19.3}}

\put(30,59){\line(0,-1){9}}

\put(0,30){\circle*{2.5}}
\put(60,30){\circle*{2.5}}
\put(30,50){\circle*{2.5}}
\put(30,0){\circle*{2.5}}

\put(10,30){\circle{2.5}}
\put(50,30){\circle{2.5}}
\put(30,60){\circle{2.5}}

\put(0,25){\makebox(0,0){$3$}}
\put(60,25){\makebox(0,0){$1$}}
\put(25,0){\makebox(0,0){$2$}}
\put(25,60){\makebox(0,0){$4$}}

\put(15,15){\makebox(0,0){$\alpha$}}
\put(45,15){\makebox(0,0){$\beta$}}
\put(5,35){\makebox(0,0){$\beta$}}
\put(55,35){\makebox(0,0){$\gamma$}}
\put(18,43){\makebox(0,0){$\gamma$}}
\put(42,43){\makebox(0,0){$\alpha$}}
\put(34,55){\makebox(0,0){$\beta$}}

\end{picture}
\quad
\begin{picture}(60,60)(0,0)
\thicklines
\put(0,30){\line(1,0){8.5}}
\put(30,50){\line(-1,-1){19}}
\put(30,50){\line(1,-1){19}}
\put(30,0){\line(-2,3){19.3}}

\put(60,30){\line(-1,0){9}}
\put(30,0){\line(2,3){19.3}}

\put(30,59){\line(0,-1){9}}

\put(0,30){\circle*{2.5}}
\put(60,30){\circle*{2.5}}
\put(30,50){\circle*{2.5}}
\put(30,0){\circle*{2.5}}

\put(10,30){\circle{2.5}}
\put(50,30){\circle{2.5}}
\put(30,60){\circle{2.5}}

\put(0,25){\makebox(0,0){$3$}}
\put(60,25){\makebox(0,0){$1$}}
\put(25,0){\makebox(0,0){$2$}}
\put(25,60){\makebox(0,0){$4$}}

\put(15,15){\makebox(0,0){$\gamma$}}
\put(45,15){\makebox(0,0){$\alpha$}}
\put(5,35){\makebox(0,0){$\beta$}}
\put(55,35){\makebox(0,0){$\gamma$}}
\put(18,43){\makebox(0,0){$\alpha$}}
\put(42,43){\makebox(0,0){$\beta$}}
\put(34,55){\makebox(0,0){$\gamma$}}

\end{picture}
\quad
\begin{picture}(60,60)(0,0)
\thicklines
\put(0,30){\line(1,0){8.5}}
\put(30,50){\line(-1,-1){19}}
\put(30,50){\line(1,-1){19}}
\put(30,0){\line(-2,3){19.3}}

\put(60,30){\line(-1,0){9}}
\put(30,0){\line(2,3){19.3}}

\put(30,59){\line(0,-1){9}}

\put(0,30){\circle*{2.5}}
\put(60,30){\circle*{2.5}}
\put(30,50){\circle*{2.5}}
\put(30,0){\circle*{2.5}}

\put(10,30){\circle{2.5}}
\put(50,30){\circle{2.5}}
\put(30,60){\circle{2.5}}

\put(0,25){\makebox(0,0){$3$}}
\put(60,25){\makebox(0,0){$1$}}
\put(25,0){\makebox(0,0){$2$}}
\put(25,60){\makebox(0,0){$4$}}

\put(15,15){\makebox(0,0){$\alpha$}}
\put(45,15){\makebox(0,0){$\gamma$}}
\put(5,35){\makebox(0,0){$\beta$}}
\put(55,35){\makebox(0,0){$\beta$}}
\put(18,43){\makebox(0,0){$\gamma$}}
\put(42,43){\makebox(0,0){$\alpha$}}
\put(34,55){\makebox(0,0){$\beta$}}

\end{picture}
    \end{center} 
    \caption{The proper labelings of the edges of a tensor diagram. 
Here $(\alpha,\beta,\gamma)$ runs over all permutations of $1,2,3$.}
    \label{fig:edge-labelings}
\end{figure}

\vspace{-.1in}

\begin{figure}[ht]
    \begin{center}
\setlength{\unitlength}{1.5pt}
\begin{picture}(60,60)(0,0)
\thicklines
\put(0,30){\line(1,0){18.5}}
\put(30,0){\line(-1,3){9.5}}

\put(60,30){\line(-1,0){39}}

\put(30,59){\line(0,-1){59}}

\put(0,30){\circle*{2.5}}
\put(60,30){\circle*{2.5}}
\put(30,0){\circle*{2.5}}

\put(20,30){\circle{2.5}}
\put(30,60){\circle{2.5}}

\put(0,25){\makebox(0,0){$3$}}
\put(60,25){\makebox(0,0){$1$}}
\put(25,0){\makebox(0,0){$2$}}
\put(25,58){\makebox(0,0){$4$}}

\end{picture}
    \end{center} 
    \caption{A tensor diagram defining the same invariant as the one in
    Figure~\ref{fig:edge-labelings}.}
    \label{fig:tripod-and-a-stick}
\end{figure}
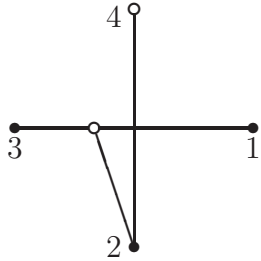

The $\SL(V)$ invariant $[D]$ defined by a tensor
diagram~$D$ is homo\-geneous 
in each of its arguments~$x(v)$, for $v$ a boundary vertex.
The degree of $[D]$ with respect to~$x(v)$ is equal to the
degree~$d_v$ of~$v$ (in the graph theory sense). 
The \emph{multidegree} of~$D$ is, by definition, 
the multidegree of~$[D]$ (cf.\ \eqref{eq:multideg}):
\begin{equation}
\label{eq:td-multideg}
\operatorname{multideg}(D)=
\operatorname{multideg}([D])=
(d_v)_{v\in\operatorname{bd}(D)}. 
\end{equation}

\pagebreak[3]

\begin{remark}
An $\SL_3$ invariant $[D]$ associated with a tensor diagram~$D$
without multiple edges can be
alternatively defined by a determinantal formula, as explained below.
Although we do not rely on such formulas in this paper,
they may be useful in some calculations. 

As before, suppose that $D$ has type $(a,b)$.
Assume that no two boundary vertices are directly connected by an
edge. 
(Otherwise, factor out the corresponding scalar product, and
remove the edge.)  
Assume that $D$ has $c+d$ interior vertices, 
$c$ of them white and $d$ black. 
Clearly $a+3c=b+3d$, the number of edges in~$D$.
We now build a square $(a+3c)\times(b+3d)$ matrix $M(D)$ as follows:
\begin{itemize}
\item
the rows of $M(D)$ come from the white vertices
in~$\operatorname{int}(D)$ (each such vertex contributes three rows) 
as well as from the edges incident to the white vertices
in~$\operatorname{bd}(D)$ (each such edge contributes one row);
\item
similarly, the columns of $M(D)$ come from the black interior vertices (each
contributes three) and from the edges incident to the black boundary
vertices (each contributes one). 
\end{itemize}
The matrix $M(D)$ is then built out of rectangular blocks $M(i,j)$
of sizes $3\times 3$, $3\times 1$, $1\times 3$, and $1\times 1$, 
defined as follows:
\begin{itemize}
\item
if $i$ is a white interior vertex and $j$ is a black interior vertex incident to $i$,
then $M(i,j)$ is a $3\times 3$ identity matrix;
\item
if $i$ is a white interior vertex and $j$ is an edge incident to $i$ and a
black boundary vertex~$v$, 
then $M(i,j)$ is the column vector~$x(v)$; 
\item
if $j$ is a black interior vertex and $i$ is an edge incident to $j$ and a
white boundary vertex~$v$, 
then $M(i,j)$ is the row (co)vector~$y(v)$; 
\item
otherwise, 
$M(i,j)$ consists entirely of zeroes. 
\end{itemize}
It is straightforward to check that $\det(M(D))=\pm [D]$. 
(The sign depends on the ordering of the rows and columns in~$M(D)$.) 
It is also easy to see directly that $\det(M(D))$ is an $SL_3$
invariant. 

We illustrate this construction for the tensor diagram shown in
Figure~\ref{fig:webs22}(d),
using the notation introduced in
Example~\ref{example:tripod-and-stick}. 
The recipe described above gives
\[
M(D)=
\left[
\begin{array}{ccc|c|c|c|c}
1 & 0 & 0 & x_{11} & x_{12} & 0 & 0 \\
0 & 1 & 0 & x_{21} & x_{22} & 0 & 0 \\
0 & 0 & 1 & x_{31} & x_{32} & 0 & 0 \\
\hline
1 & 0 & 0 &   0    &   0    & x_{12} & x_{13}\\
0 & 1 & 0 &   0    &   0    & x_{22} & x_{23}\\
0 & 0 & 1 &   0    &   0    & x_{32} & x_{33}\\
\hline
y_1 & y_2 & y_3 & 0 & 0 & 0 & 0
\end{array}
\right].
\]
It is then easy to see that 
\[
-\det(M(D))=
\det
\begin{bmatrix}
x_{11} & x_{12} & x_{13}\\
x_{21} & x_{22} & x_{23}\\
x_{31} & x_{32} & x_{33}
\end{bmatrix}
(x_{12} y_1+x_{22} y_2+x_{32} y_3), 
\]
matching \eqref{eq:det*dot-product}.
\end{remark}

\pagebreak[3]

The calculus of tensor diagrams includes natural counterparts for both the
additive and the multi\-plicative structures in the ring of invariants. 
For the former, allow formal linear combinations of
tensor diagrams, and extend the definition of $[D]$ by linearity.  
For the latter, use 
superposition of 
diagrams: if 
$D$ 
is a union of subdiagrams $D_1,D_2,\dots$ 
connected only at boundary vertices, then
$[D]=[D_1][D_2]\cdots$; cf.\ Figure~\ref{fig:tripod-and-a-stick}.

The Weyl generators of the ring $R_{a,b}(V)$ are encoded by the simplest
possible tensor diagrams: two types of \emph{tripods} (corresponding
to two kinds of Pl\"ucker coordinates) and simple edges.
To rephrase, one takes each of the three types of building
blocks shown in Figure~\ref{fig:webs21} and attaches it directly to the
boundary vertices. 
The First Fundamental Theorem implies that any $\SL(V)$ invariant can
be represented (non-uniquely) as a linear combination of invariants
associated with tensor diagrams obtained by superposition of tripods
and edges. 

An important class of relations among invariants associated with
tensor diagrams on an oriented plane is obtained from \emph{local
  transformation rules}.  
Such a rule trans\-forms a small fragment $F$ of
a tensor diagram $D$ into a linear combination $\sum_i c_i F_i$ 
of other such pieces
(corresponding to tensors of the same type) while keeping the rest of
the diagram intact. 
The intrinsic definition of the invariants~$[D]$ implies 
that if $[F]\!=\!\sum_i c_i [F_i]$  (i.e., the tensor does not change
locally), then $[D]=\sum_i c_i [D_i]$, where $D_i$ denotes the tensor
diagram obtained from $D$ by replacing $F$ by~$F_i$. 

Figure~\ref{fig:skein} shows several fundamental relations of this
kind. 
Checking their validity is straightforward. 

Several basic relations involving boundary vertices are shown in
Figure~\ref{fig:local-boundary}. 

\begin{figure}[ht]
    \begin{tabular}{cl}
\setlength{\unitlength}{2pt}
\begin{picture}(16,20)(0,0)
\put(8,10){\makebox(0,0){(a)}}
\end{picture}
&
\setlength{\unitlength}{2pt}
\begin{picture}(30,20)(0,0)
\thicklines
\put(0,20){\line(3,-2){30}}
\put(0,20){\vector(3,-2){25}}

\put(0,0){\line(3,2){30}}
\put(0,0){\vector(3,2){25}}
\end{picture}
\begin{picture}(16,20)(0,0)
\put(8,10){\makebox(0,0){$=$}}
\end{picture}
\begin{picture}(30,20)(0,0)
\thicklines
\put(0,0){\line(1,1){9.2}}
\put(0,20){\line(1,-1){9.2}}
\put(10,10){\circle*{2.5}}

\put(11,10){\line(1,0){8}}
\put(20,10){\circle{2.5}}

\put(30,20){\line(-1,-1){9.2}}
\put(30,0){\line(-1,1){9.2}}
\end{picture}
\begin{picture}(16,20)(0,0)
\put(8,10){\makebox(0,0){$+$}}
\end{picture}
\begin{picture}(30,20)(0,0)
\thicklines
\put(0,0){\line(1,0){30}}
\put(0,0){\vector(1,0){20}}
\put(0,20){\line(1,0){30}}
\put(0,20){\vector(1,0){20}}
\end{picture}
\\[.3in]
\setlength{\unitlength}{2pt}
\begin{picture}(16,30)(0,0)
\put(8,15){\makebox(0,0){(b)}}
\end{picture}
&
\setlength{\unitlength}{2pt}
\begin{picture}(30,30)(0,0)
\thicklines
\put(0,0){\line(1,1){9}}
\put(30,30){\line(-1,-1){9}}
\put(0,30){\line(1,-1){9}}
\put(30,0){\line(-1,1){9}}
\put(10,10){\circle{2.5}}
\put(20,20){\circle{2.5}}
\put(10,20){\circle*{2.5}}
\put(20,10){\circle*{2.5}}
\put(11,10){\line(1,0){8}}
\put(11,20){\line(1,0){7.8}}
\put(10,11.2){\line(0,1){8}}
\put(20,11){\line(0,1){8}}
\end{picture}
\begin{picture}(16,30)(0,0)
\put(8,15){\makebox(0,0){$=$}}
\end{picture}
\begin{picture}(30,30)(0,0)
\thicklines
\put(0,0){\line(1,0){30}}
\put(15,0){\vector(-1,0){5}}
\put(0,30){\line(1,0){30}}
\put(15,30){\vector(1,0){5}}
\end{picture}
\begin{picture}(16,30)(0,0)
\put(8,15){\makebox(0,0){$+$}}
\end{picture}
\begin{picture}(30,30)(0,0)
\thicklines
\put(0,0){\line(0,1){30}}
\put(0,15){\vector(0,-1){5}}
\put(30,0){\line(0,1){30}}
\put(30,15){\vector(0,1){5}}
\end{picture}
\\[.3in]
\setlength{\unitlength}{2pt}
\begin{picture}(16,10)(0,0)
\put(8,5){\makebox(0,0){(c)}}
\end{picture}
&
\setlength{\unitlength}{2pt}
\begin{picture}(30,10)(0,0)
\thicklines
\put(0,5){\line(1,0){9}}
\put(30,5){\line(-1,0){9}}
\qbezier(11,6)(15,10)(19,6)
\qbezier(11,4)(15,0)(19,4)
\put(10,5){\circle*{2.5}}
\put(20,5){\circle{2.5}}
\end{picture}
\begin{picture}(16,10)(0,0)
\put(8,5){\makebox(0,0){$=$}}
\end{picture}
\begin{picture}(14,10)(0,0)
\put(6,5){\makebox(0,0){$(-2)$}}
\end{picture}
\begin{picture}(30,10)(0,0)
\thicklines
\put(0,5){\line(1,0){30}}
\put(15,5){\vector(1,0){5}}
\end{picture}
\\[.2in]
\setlength{\unitlength}{2pt}
\begin{picture}(16,20)(0,0)
\put(8,10){\makebox(0,0){(d)}}
\end{picture}
&
\setlength{\unitlength}{2pt}
\begin{picture}(30,20)(0,0)
\thicklines
\put(15,10){\circle{20}}
\end{picture}
\begin{picture}(16,20)(0,0)
\put(8,10){\makebox(0,0){$=$}}
\end{picture}
\begin{picture}(4,20)(0,0)
\put(1.5,10.5){\makebox(0,0){$3$}}
\end{picture}    
\end{tabular} 
    \caption{Skein relations for tensor diagrams. All
    edges are oriented towards their black endpoints. The cycle in
    relation~(d) can be oriented either way.
Kuperberg~\cite{kuperberg} gives $q$-versions of relations (a), (c) and~(d).}
    \label{fig:skein}
\end{figure}
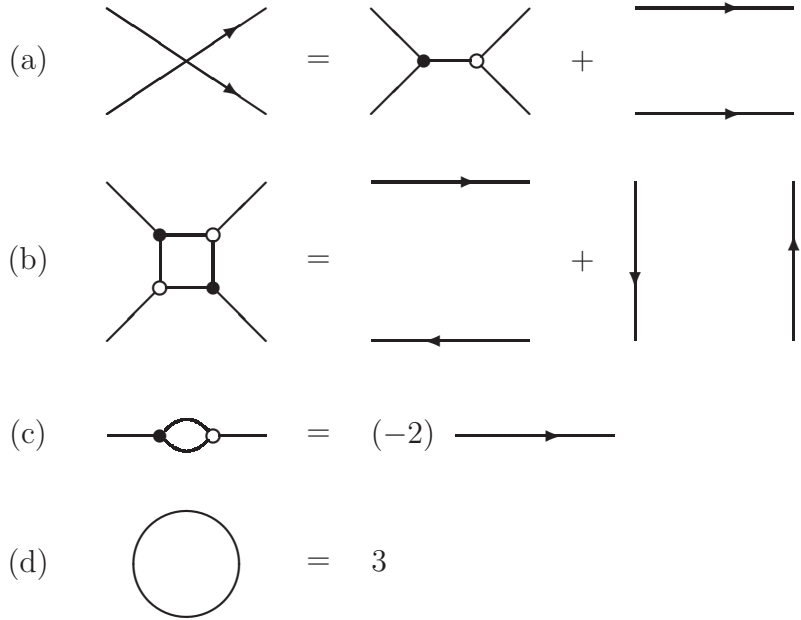

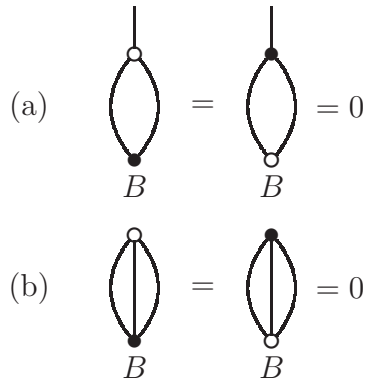
\begin{figure}[ht]
    \begin{tabular}{cl}
\setlength{\unitlength}{2pt}
\begin{picture}(16,25)(0,-5)
\put(8,10){\makebox(0,0){(a)}}
\end{picture}
&
\setlength{\unitlength}{2pt}
\begin{picture}(10,35)(0,-5)
\thicklines
\put(5,-5){\makebox(0,0){$B$}}
\put(5,0){\circle*{2.5}}
\put(5,20){\circle{2.5}}
\put(5,21){\line(0,1){8}}
\qbezier(4,1)(-3,10)(4,19)
\qbezier(6,1)(13,10)(6,19)
\end{picture}
\begin{picture}(12,25)(0,-5)
\put(6,10){\makebox(0,0){$=$}}
\end{picture}
\begin{picture}(10,25)(0,-5)
\thicklines
\put(5,-5){\makebox(0,0){$B$}}
\put(5,0){\circle{2.5}}
\put(5,20){\circle*{2.5}}
\put(5,21){\line(0,1){8}}
\qbezier(4,1)(-3,10)(4,19)
\qbezier(6,1)(13,10)(6,19)
\end{picture}
\begin{picture}(16,25)(0,-5)
\put(6,10){\makebox(0,0){$=0$}}
\end{picture}
\\[.2in]
\setlength{\unitlength}{2pt}
\begin{picture}(16,25)(0,-5)
\put(8,10){\makebox(0,0){(b)}}
\end{picture}
&
\setlength{\unitlength}{2pt}
\begin{picture}(10,25)(0,-5)
\thicklines
\put(5,-5){\makebox(0,0){$B$}}
\put(5,0){\circle*{2.5}}
\put(5,20){\circle{2.5}}
\put(5,1){\line(0,1){18}}
\qbezier(4,1)(-3,10)(4,19)
\qbezier(6,1)(13,10)(6,19)
\end{picture}
\begin{picture}(12,25)(0,-5)
\put(6,10){\makebox(0,0){$=$}}
\end{picture}
\begin{picture}(10,25)(0,-5)
\thicklines
\put(5,-5){\makebox(0,0){$B$}}
\put(5,0){\circle{2.5}}
\put(5,20){\circle*{2.5}}
\put(5,1){\line(0,1){18}}
\qbezier(4,1)(-3,10)(4,19)
\qbezier(6,1)(13,10)(6,19)
\end{picture}
\begin{picture}(16,25)(0,-5)
\put(6,10){\makebox(0,0){$=0$}}
\end{picture}
\end{tabular}
    \caption{Local relations involving a vertex~$B$ lying on the boundary.}
    \label{fig:local-boundary}
\end{figure}

Figure~\ref{fig:yang-baxter} shows a few additional relations which
are immediate from the definition of the invariants
associated with tensor diagrams. 
These relations can also be deduced from those in Figure~\ref{fig:skein}.

\begin{figure}[ht]
    \begin{tabular}{cl}
\setlength{\unitlength}{2pt}
\begin{picture}(16,20)(0,0)
\put(8,10){\makebox(0,0){(a)}}
\end{picture}
&
\setlength{\unitlength}{2pt}
\begin{picture}(30,20)(0,0)
\thicklines
\put(0,20){\line(3,-2){30}}
\put(0,0){\line(3,2){30}}
\qbezier(0,10)(15,0)(30,10)
\end{picture}
\begin{picture}(16,20)(0,0)
\put(8,10){\makebox(0,0){$=$}}
\end{picture}
\begin{picture}(30,20)(0,0)
\thicklines
\put(0,20){\line(3,-2){30}}
\put(0,0){\line(3,2){30}}
\qbezier(0,10)(15,20)(30,10)
\end{picture}
\\[.2in]
\setlength{\unitlength}{2pt}
\begin{picture}(16,16)(0,0)
\put(8,8){\makebox(0,0){(b)}}
\end{picture}
&
\setlength{\unitlength}{2pt}
\begin{picture}(30,16)(0,0)
\thicklines
\qbezier(0,0)(20,5)(20,10)
\qbezier(30,0)(10,5)(10,10)
\qbezier(10,10)(10,15)(15,15)
\qbezier(20,10)(20,15)(15,15)
\end{picture}
\begin{picture}(16,16)(0,0)
\put(8,8){\makebox(0,0){$=$}}
\end{picture}
\begin{picture}(30,16)(0,0)
\thicklines
\put(0,0){\line(1,0){30}}
\end{picture}
\\[.3in]
\setlength{\unitlength}{2pt}
\begin{picture}(16,10)(0,0)
\put(8,5){\makebox(0,0){(c)}}
\end{picture}
&
\setlength{\unitlength}{2pt}
\begin{picture}(30,10)(0,0)
\thicklines
\qbezier(0,0)(15,15)(30,0)
\qbezier(0,10)(15,-5)(30,10)
\end{picture}
\begin{picture}(16,10)(0,0)
\put(8,5){\makebox(0,0){$=$}}
\end{picture}
\begin{picture}(30,10)(0,0)
\thicklines
\put(0,0){\line(1,0){30}}
\put(0,10){\line(1,0){30}}
\end{picture}
\\[.2in]
\setlength{\unitlength}{2pt}
\begin{picture}(16,20)(0,0)
\put(8,10){\makebox(0,0){(d)}}
\end{picture}
&
\setlength{\unitlength}{2pt}
\begin{picture}(30,20)(0,0)
\thicklines
\put(0,10){\line(1,0){8.8}}
\put(10,10){\circle{2.5}}
\qbezier(11,11)(20,18)(30,0)
\qbezier(11,9)(20,2)(30,20)
\end{picture}
\begin{picture}(16,20)(0,0)
\put(8,10){\makebox(0,0){$=$}}
\end{picture}
\begin{picture}(12,20)(0,0)
\put(5,10){\makebox(0,0){$(-1)$}}
\end{picture}
\begin{picture}(30,20)(0,0)
\thicklines
\put(0,10){\line(1,0){8.8}}
\put(10,10){\circle{2.5}}
\put(30,20){\line(-2,-1){18.7}}
\put(30,0){\line(-2,1){18.7}}
\end{picture}
\\[.2in]
\setlength{\unitlength}{2pt}
\begin{picture}(16,20)(0,0)
\put(8,10){\makebox(0,0){(e)}}
\end{picture}
&
\setlength{\unitlength}{2pt}
\begin{picture}(30,20)(0,0)
\thicklines
\put(0,10){\line(1,0){8.8}}
\put(10,10){\circle*{2.5}}
\qbezier(11,11)(20,18)(30,0)
\qbezier(11,9)(20,2)(30,20)
\end{picture}
\begin{picture}(16,20)(0,0)
\put(8,10){\makebox(0,0){$=$}}
\end{picture}
\begin{picture}(12,20)(0,0)
\put(5,10){\makebox(0,0){$(-1)$}}
\end{picture}
\begin{picture}(30,20)(0,0)
\thicklines
\put(0,10){\line(1,0){8.8}}
\put(10,10){\circle*{2.5}}
\put(30,20){\line(-2,-1){18.7}}
\put(30,0){\line(-2,1){18.7}}
\end{picture}

\end{tabular} 

    \caption{Yang-Baxter-type relations for tensor diagrams.
In relations (a)--(c), edge
    orientations on the left-hand sides can be arbitrary;
the orientations on each right-hand side should match those on the~left.}
    \label{fig:yang-baxter}
\end{figure}
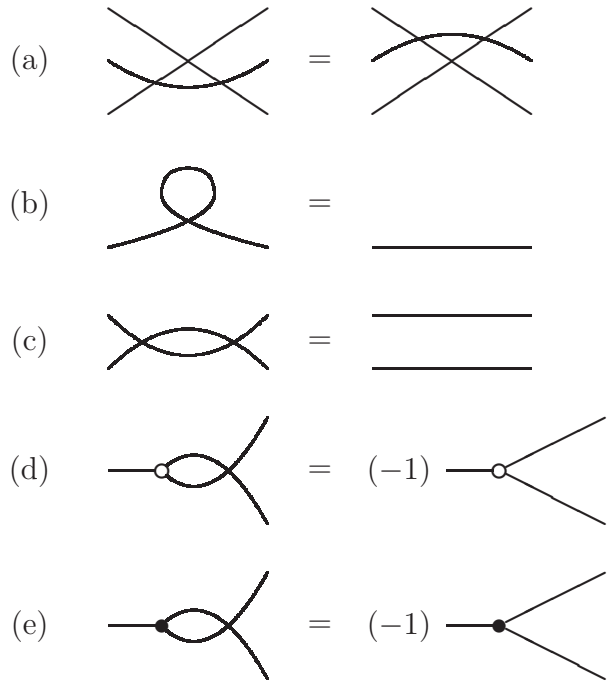

It is possible to develop a diagrammatic calculus based on the local
relations listed above without making any mention of tensors or the
special linear group~$\SL_3$; see in particular \cite{kuperberg} and
references therein. 
Such diagrammatic calculus, and its $q$-analogues, can in particular
be used to construct invariants of knots and links. 
We do not discuss these connections in this paper. 

\newpage

\section{Webs}
\label{sec:webs}

Informally speaking, webs are \emph{planar} tensor diagrams. 
The systematic  study of webs was pioneered by
G.~Kuperberg~\cite{kuperberg} whose 
foundational results are reviewed below,
in  slightly different terminology. 

\begin{definition}[\emph{Webs}]
A (planar) \emph{web} is a tensor diagram $D$ embedded in an oriented
disk as described above in Section~\ref{sec:tensor-diagrams},  
so that its edges do not cross or touch each other, except at
endpoints. 
Each web is considered up to an isotopy of the disk that fixes its
boundary---so it is in essence a combinatorial object.  

Recall that the boundary (respectively internal)
vertices of $D$ must lie on the boundary (respectively in the
interior) of the disk, and the three edges meeting at each internal
vertex are viewed as cyclically ordered clockwise. 

A web is called \emph{non-elliptic} if
it has no multiple edges, and no 4-cycles 
whose all four vertices are internal.

An invariant $[D]$ associated with a non-elliptic web~$D$ is called a
\emph{web invariant}. 
\end{definition}

A web $D$ of type $(a,b)$ has $a$ white boundary vertices and $b$
black ones. 
The cyclic pattern of colors of the boundary vertices 
is encoded by the (cyclic) \emph{signature} of~$D$
(cf.~a similar notion introduced in
Section~\ref{sec:rings-of-invariants}), 
a cyclically ordered binary string, 
or more precisely a word in the alphabet $\{\bullet,\circ\}$
considered up to a cyclic permutation. 
For example, the webs in
Figures~\ref{fig:edge-labelings}--\ref{fig:tripod-and-a-stick} have
signature 
\[
[\bullet\bullet\bullet\,\circ]
=[\bullet\bullet\circ\,\bullet]
=[\bullet\circ\bullet\,\bullet]
=[\circ\bullet\bullet\,\bullet]
\]

The following highly nontrivial result is the cornerstone of the
theory of webs. 

\begin{theorem}[{\rm G.~Kuperberg \cite{kuperberg}}]
\label{th:web-basis} 
Web invariants with a fixed signature $\sigma$ of type $(a,b)$ 
form a linear basis in the ring of invariants $R_\sigma(V)\cong R_{a,b}(V)$. 
\end{theorem}


Being a special kind of tensor diagrams, 
webs have \emph{multidegrees}, 
cf.~\eqref{eq:td-multideg};
the invariants they define are multi-homogeneous. 
Kuperberg's theorem can be restated as saying that 
web invariants with a fixed signature~$\sigma$ 
and a fixed multidegree make up a linear basis in the corresponding
multi-homogeneous component of~$R_\sigma(V)$. 

\begin{example}
\label{example:R24}
For the ring of invariants $R_{4,2}(V)$,  
choose the signature 
\[
\sigma=[\bullet\circ\bullet\circ\circ\,\circ]
\]
and  the multidegree $(1,1,1,1,2,1)$. 
There are $5$ non-elliptic webs of such signature and multidegree,
shown in Figure~\ref{fig:webs23} at the top.
Thus the corresponding multi-homogeneous component 
of $R_\sigma(V)\cong R_{4,2}(V)$ is $5$-dimensional. 
\end{example}

\begin{figure}[ht]
\scalebox{0.95}{\input{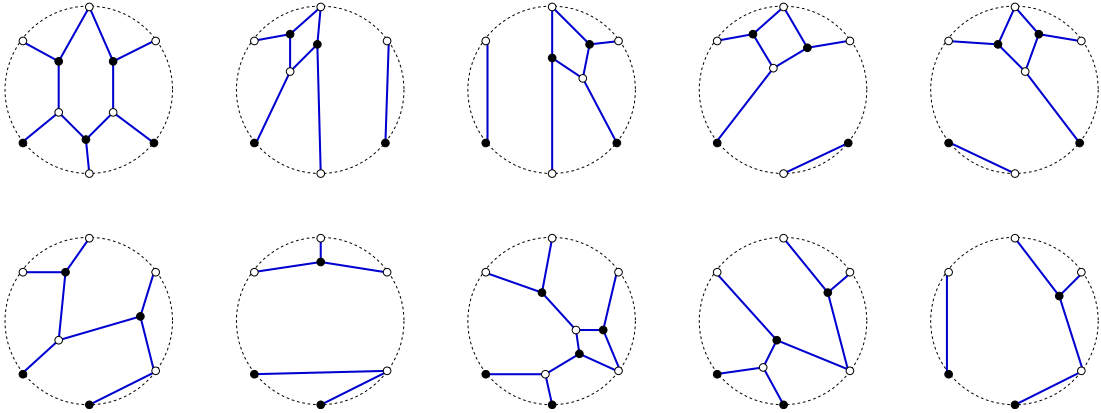}}
    \caption{Two web bases in a multi-homogeneous component
    of $R_{4,2}(V)$.}
    \label{fig:webs23}
\end{figure}

\begin{remark}
Theorem~\ref{th:web-basis} implies that the number of
non-elliptic webs of given multidegree 
does not depend on the
choice of a signature of given type~$(a,b)$. 
To illustrate, continue with Example~\ref{example:R24}. 
Taking instead the signature
$[\bullet\bullet\circ\circ\circ\,\circ]$
of the same type $(4,2)$, and choosing the multidegree
$(1,1,1,1,1,2)$---so that, as before, the invariants in 
question are multilinear in all arguments except for a single
covector---we get the $5$ webs (same number!) 
shown in Figure~\ref{fig:webs23} at the bottom.
\end{remark}

\begin{remark}
\label{rem:confluence}
Any linear combination of tensor diagrams 
can be transformed 
 into a linear combination of non-elliptic webs by repeated
 application of local relations shown in
 Figures~\hbox{\ref{fig:skein}--\ref{fig:local-boundary}}. 
(Just apply them left-to-right.) 
It follows that $R_\sigma(V)$ is spanned by the web invariants. 
The content of Theorem~\ref{th:web-basis} is that web invariants
are linearly independent, so the reduction process described above is
 \emph{confluent}.

Theorem~\ref{th:web-basis} implies that
two linear combinations of tensor
diagrams define the same invariant if and only if they can be
transformed into each other using the relations
in Figures~\ref{fig:skein}--\ref{fig:local-boundary}. 
\end{remark}

\begin{remark}
\label{rem:khovanov-kuperberg}
It is tempting to hypothesize, as M.~Khovanov and
G.~Kuperberg originally did~\cite{khovanov-kuperberg}, 
that the web basis in a space of multilinear invariants (of some fixed
signature) 
coincides with the corresponding instance of G.~Lusztig's \emph{dual canonical
  basis} (see \cite{khovanov-kuperberg} for definitions and
references).
Their investigation however established that web bases for $\SL_3$ 
are generally \emph{not} dual canonical. 
The first discrepancy occurs in degree~12, for the signature
$\bullet\bullet\circ\circ\bullet\bullet\circ\circ\bullet\bullet\circ\,\circ$. 
See Section~\ref{sec:thickening-non-cluster} 
for a cluster-theoretic interpretation of this counterexample.  
\end{remark}

A \emph{closed} tensor diagram (without boundary vertices)
represents a tensor of type $(0,0)$, \emph{i.e.}\  a scalar. 
This scalar has the following direct description. 

\begin{proposition}
\label{prop:closed-webs}
Let $D$ be a closed web with $m$ white and $m$ black vertices. 
Then $[D]$ is equal to $(-1)^m$ times the number of proper colorings
of the edges of~$D$ into three colors. (That is, we require the colors
of the three edges incident to each vertex of~$D$ to be distinct.) 
\end{proposition}

The simplest illustration of Proposition~\ref{prop:closed-webs} 
is provided by formula~(d) in Figure~\ref{fig:skein} (with $m=0$).
A more interesting example: let $D$ be the $1$-skeleton of the
(three-dimensional) cube; then $[D]=24$. 

The proof of Proposition~\ref{prop:closed-webs} is omitted; we do not
rely on it elsewhere in the paper. 

\pagebreak[3]

\section{Special invariants}
\label{sec:special-invariants}

Fix a (cyclic) signature $\sigma$ of type $(a,b)$ with $a+b\ge 5$. 
We assume that $\sigma$ is \emph{non-alternating}, i.e.,
it has two adjacent vertices of the same color. 
In this section, we construct a family of ``special'' elements in the
ring of invariants~$R_\sigma(V)\cong R_{a,b}(V)$.
These special invariants will play a key role in our main construction. 

The proofs of the propositions stated in
Sections~\ref{sec:special-invariants}--\ref{sec:special-seeds} are given
in Section~\ref{sec:special-proofs}. 

\medskip

As before, we work in a disk with $a+b$ marked points (vertices) 
on the boundary, $a$~of them white and $b$~of them black,
arranged in accordance with the signature~$\sigma$. 
We label these vertices $1,\dots,a+b$, going clockwise,
and work with them modulo~$a+b$. 

For each boundary vertex~$p$,  
we next define two trees $\Lambda_p$ and $\Lambda^p$ embedded into our
disk. These trees will then serve as building blocks for 
certain tensor diagrams. 

If the vertex $p$ is black, then $\Lambda_p$ has one vertex,
namely~$p$, and no edges.
If $p$ is white, then place a new black vertex 
(which we call the \emph{proxy} of~$p$) inside the disk
and connect it to~$p$. 
We then examine $p\!+\!1$. If $p\!+\!1$ is white,
then we connect it to the proxy vertex, and stop; 
see Figure~\ref{fig:caterpillars}(a). 
If $p$ is white and $p+1$ is black, then we look at $p\!+\!2$. 
If $p\!+\!2$ is black, then put a white vertex inside the disk,
connect it to both $p\!+\!1$ and $p\!+\!2$, and to the proxy vertex;
see Figure~\ref{fig:caterpillars}(b). 
If $p$ is white, $p\!+\!1$ is black, and $p\!+\!2$ is white, then
examine~$p\!+\!3$, and so on, cf.\ Figure~\ref{fig:caterpillars}(c).
In general, we proceed clockwise from~$p$ until we find two
consecutive vertices of the same color (here we need the condition
that $\sigma$ is non-alternating), then build a caterpillar-like
bi-colored tree whose all interior vertices have degree~$3$ except for
the proxy vertex which has degree~$2$. 


\begin{figure}[ht]
\scalebox{1}{\input{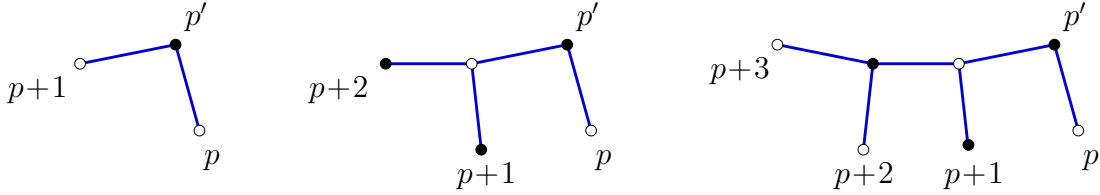}}
    \caption{Trees $\Lambda_p$. The proxy vertex is denoted by~$p'$.}
    \label{fig:caterpillars}
\end{figure}

The graph $\Lambda^p$ is defined in the same way, with the colors
swapped. The proxy vertex is defined analogously (if $p$ is white).

At the next stage, we stitch together several 
caterpillar trees 
to build tensor diagrams. 
In anticipation of this stage, we draw each of these trees without
self-intersections, and sufficiently close to the boundary of the disk,
so as to avoid ambiguities later on 
related to the choice of cyclic ordering at
each vertex of the tensor diagram. 


\begin{definition}[\emph{Special invariants}]
\label{def:special-inv}
Let $p$ and $q$ be boundary vertices, $p\ne q$.
The special invariant $J_p^q$ is defined by the tensor
diagram obtained by connecting the trees $\Lambda_p$ and~$\Lambda^q$
by a single edge.
One of its endpoints is $p$ if the latter is black, or else
take the proxy of~$p$.
The other endpoint is $q$ if the latter is white, or else take the
proxy of~$q$. 
Make sure the connector edge approaches every proxy vertex from
the same side where the disk's center lies. 
See Figure~\ref{fig:webs61} for a couple of examples.

Now, let $p,q,r$ be distinct boundary vertices, ordered clockwise.
The special invariants $J_{pqr}$ and $J^{pqr}$ are defined by tensor
diagrams constructed in a similar
  fashion to those used for~$J_p^q$. 
For $J_{pqr}$, place a white vertex in the middle, and
draw edges from it to each of $\Lambda_p, \Lambda_q, \Lambda_r$.
As the other endpoints of these three edges,
use the vertices $p,q,r$ whenever they are black, or else take
their respective proxies. 
For $J^{pqr}$, reverse the roles of the colors. 
See Figure~\ref{fig:webs61}. 

Finally, let $p,q,r,s$ be four boundary vertices, ordered clockwise.
The invariant $J_{pq}^{rs}$ is defined by the tensor diagram
obtained as follows. 
Place a white vertex~$W$ and a black vertex~$B$ near the center of the
disk. Connect them by an edge. 
Connect $W$ to $\Lambda_p$ and~$\Lambda_q$ using $p$ and $q$ if these two are
black, or else using their proxies as needed. 
Similarly, connect $B$ to the appropriate white vertices in $\Lambda^r$ and~$\Lambda^s$. 
Make sure that the five edges incident to $B$ and $W$ do not cross each other. 
\end{definition}

\ \vspace{-.3in}

\begin{figure}[ht]
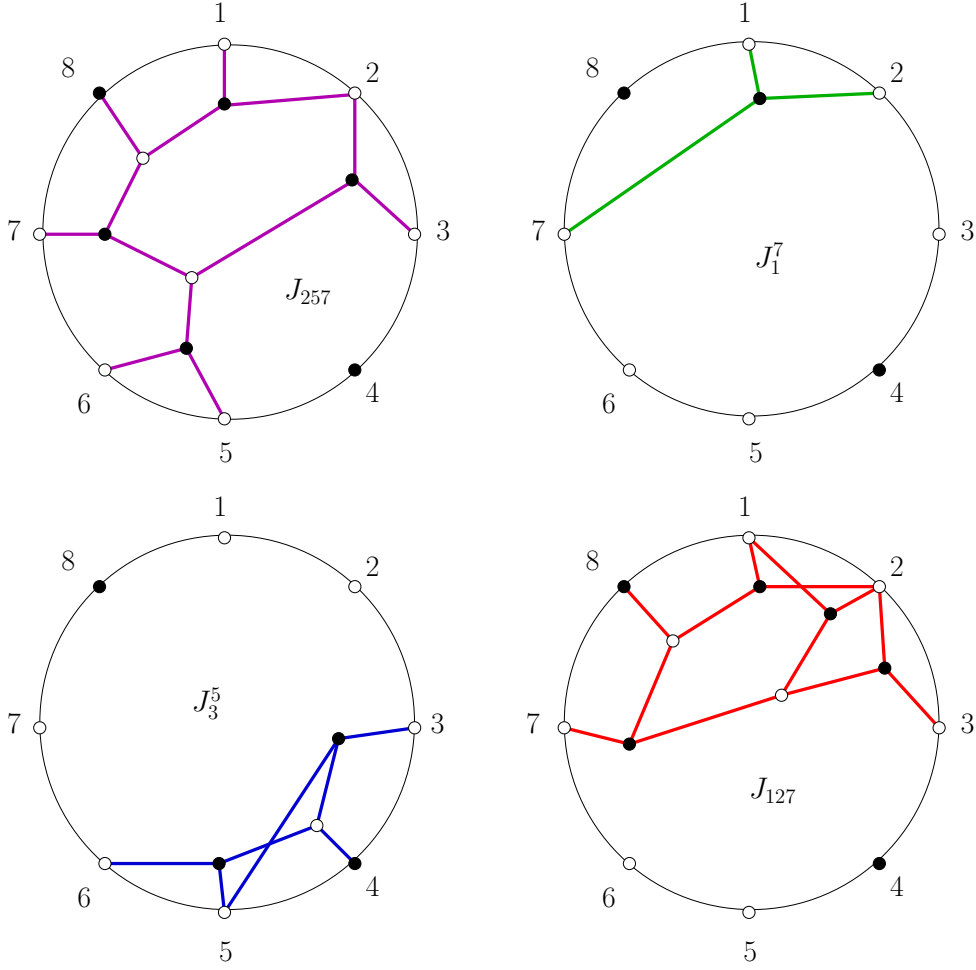

\begin{center}
\scalebox{0.65}{\input{webs61a.pstex_t}}\qquad\quad
\scalebox{0.65}{\input{webs61b.pstex_t}}
\\[.2in]
\scalebox{0.65}{\input{webs61c.pstex_t}}\qquad\quad
\scalebox{0.65}{\input{webs61d.pstex_t}}
\end{center}
    \caption{Examples of special invariants for 
$\sigma=[\circ \circ \circ \bullet \circ
    \circ \circ \, \bullet]$.\quad
Note~that $J_{257}=J_{25}^{78}$, $J_1^7=J^{127}$, $J_3^5=J^{345}$.}
    \label{fig:webs61}
\end{figure}

Some special invariants are identically equal to zero. 
The vanishing invariants of the form $J^q_p$ are identified 
in the following proposition, whose
straightforward proof we omit.  

\begin{proposition}
\label{prop:special=0}
If the vertices $p$ and~$q$ are not adjacent, 
then $J_p^q$ is a nontrivial invariant. 
If $p$ and~$q$ are adjacent, 
then exactly one of the two invariants
$J_p^q$ and $J_q^p$ vanishes.
Specifically, if $p$ is white, then $J_p^{p+1}=0$;
if $p$ is black, then $J_{p+1}^p=0$. 
\end{proposition}

For example, for the signature in Figure~\ref{fig:webs61} 
we have $J_1^2=J_2^3=J_3^4=J_5^4=J_5^6=J_6^7=J_7^8=J_1^8=0$. 

Some nonzero special invariants are equal to each other.
We do not catalogue all such instances (although we could). 
Cf.\ the caption to Figure~\ref{fig:webs61}. 

\begin{proposition}
\label{prop:specweb}
 Each nonzero special invariant is a web invariant.
\end{proposition}

To illustrate, refer to Figure~\ref{fig:webs61}. 
While the two tensor diagrams in the top row 
 are non-elliptic webs,
the ones in the bottom row are not. 
They can however be transformed using skein relations into the
non-elliptic webs shown in Figure~\ref{fig:webs62}. 


\begin{figure}[ht]
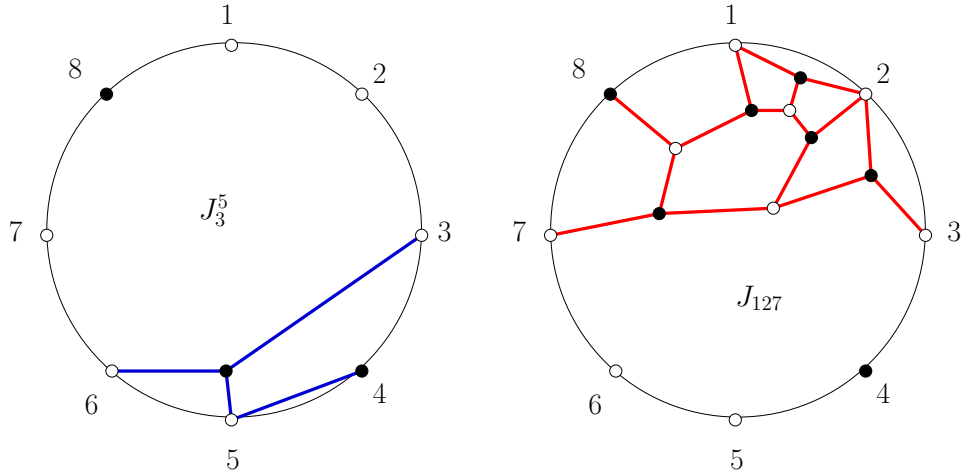

\begin{center}
\scalebox{0.65}{\input{webs62a.pstex_t}}
\qquad
\scalebox{0.65}{\input{webs62b.pstex_t}}
\end{center}
    \caption{
$J_3^5$ and $J_{127}$ are web invariants.}
    \label{fig:webs62}
\end{figure}

The following proposition is straightforwardly verified. 

\begin{proposition}
\label{prop:special-reducible}
Special invariants 
satisfy the following factorization identities: 
\begin{enumerate}
\item
\label{enum:fact-step-1a}
if $p$ is white and $p\!+\!1$ is black, then
$J_p^{p+2}\! =\! J_{p+2}^{p} J_{p+1}^{p+2}$;
\item
\label{enum:fact-step-2}
if $p$ is white, $p+1$ is black, and $p+2$ is white, then
$J_{p}^{p+3} = J_{p+1}^{p+3} J_{p+2}^{p}$;
\item
\label{enum:fact-step-2x}
if $p$ is white, then
$J_{p,p+1,q} \!=\! J_{p+1}^{p} J^{p+1}_{q}$;
\item
\label{enum:fact-step-1b}
if $p$ is white and $p\!+\!1$ is black, then
$J_{p,p+2,q} \!=\! J_{p+2}^{p} J^{p+1}_{q}$;
\item
\label{enum:fact-step-3x}
if $p$ is white, then
$J_{p,p+1}^{q,r} \!=\! J_{p+1}^{p} J^{p+1,q,r}$;
\item
\label{enum:fact-step-3}
if $p$ is white, then
$J_{rp}^{p+1,q} = J_{r}^{p+1} J_{p}^{q}$;
\item
\label{enum:fact-step-4}
if $p$ is white and $p+1$ is black, then
$J_{p,p+2}^{qr} = J_{p+2}^{p} J^{p+1,q,r}$;
\item
\label{enum:fact-step-5}
if the colors are reversed (i.e., $p$ is black, etc.), 
then switch the colors and
swap subscripts with superscripts in rules
\eqref{enum:fact-step-1a}--\eqref{enum:fact-step-4} above.
\end{enumerate}
\end{proposition}

We next show that ``almost all'' special invariants are
\emph{irreducible} elements of~$R_\sigma(V)$. 
Incidentally, \cite[Theorem~3.17]{popov-vinberg} 
asserts that $f\in R_\sigma(V)$ is irreducible 
(as an element of $R_\sigma(V)$)
if and only if $f$ is an irreducible polynomial. 
In other words, if an $\SL(V)$-invariant polynomial $f$ 
factors nontrivially in the polynomial ring 
$\CC[(V^*)^a\times V^b]$, then the factors must be $\SL(V)$-invariant. 
We will not rely on this result. 

\begin{proposition}
\label{prop:special-are-irreducible}
A nonzero special invariant is irreducible unless it fits one of the
descriptions listed in Proposition~\ref{prop:special-reducible}. 
\end{proposition}

In our running example of signature 
\hbox{$\sigma=[\circ \circ \circ \bullet \circ \circ \circ \, \bullet]$}, 
the invariants $J_{257}$ and $J_1^7$ are irreducible whereas 
$J_3^5$ and $J_{127}$ are not: 
$J_3^5=J_5^3 J_4^5$ 
and $J_{127}=J_8^2 J_2^1 J_1^7$.

\begin{figure}[ht]
    \begin{center}
\scalebox{0.7}{\input{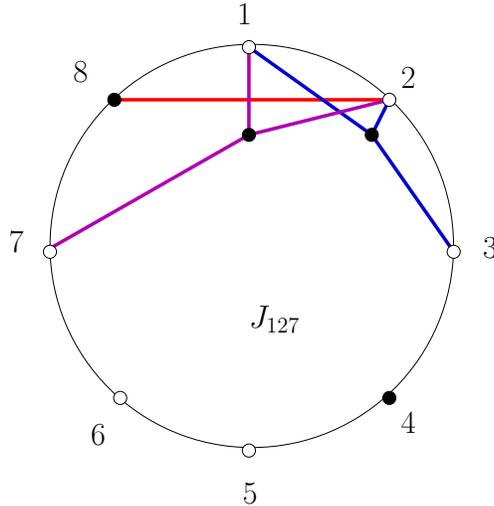}}
    \end{center}
\vspace{-.2in}
    \caption{Factorization of $J_{127}$ into irreducible special invariants}
    \label{fig:webs63}
\end{figure}


\begin{prop}
\label{prop:special-inv-factoring}
 Each nonzero special invariant is represented uniquely as a product of
 irreducible special invariants. 
\end{prop}

Such a factorization can be found by iterating the formulas
\eqref{enum:fact-step-1a}--\eqref{enum:fact-step-5}; 
see Section~\ref{sec:special-proofs}. 

Irreducible special invariants will eventually play the role of
(a subset of) generators for the cluster structure in~$R_\sigma(V)$.

We call two special invariants \emph{compatible} if their product is a
single web invariant. This terminology will prove consistent with the
notion of compatibility of cluster variables;
cf.\ Conjecture~\ref{conj:cluster-compatibility}. 
The reader is welcome to check (this may take some time) that the
four special invariants in Figure~\ref{fig:webs61} are pairwise
compatible. 

\begin{prop}
\label{prop:coeff-special}
Let $\sigma$ be a signature of type $(a,b)$, with $a+b\ge 5$, $a\le b$, 
and $\sigma\neq[\bullet\circ\bullet\circ\bullet]$. 
Then there are $a+b$ special
 invariants com\-patible with all special~invariants. 
These $a+b$ invariants are 
the (irreducible) nonzero invariants of the form $J_p^{p \pm 1}$.
Moreover the product of any of them and any web invariant is
 a web invariant. 
\end{prop}

Anticipating their future role, we call these special invariants
\emph{coefficient invariants}.
We denote the $(a\!+\!b)$-element set of these invariants by~$\cc_\sigma$. 
See Figure~\ref{fig:webs64}. 

\pagebreak[3]

\begin{remark}
\label{rem:xoxox}
For $\sigma=[\bullet\circ\bullet\circ\bullet]$, the special invariant
$J_5^1$ factors: $J_5^1=J_3^5 J_4^2$. \linebreak[3]
As a result, there are $6$ coefficient invariants in this case, 
namely $J_1^2, J_3^2, J_3^4, J_5^4, J_3^5, J_4^2$. 
\end{remark}

\pagebreak[3]

\begin{figure}[ht]
\begin{center}
\vskip-.1in 
\scalebox{0.65}{\input{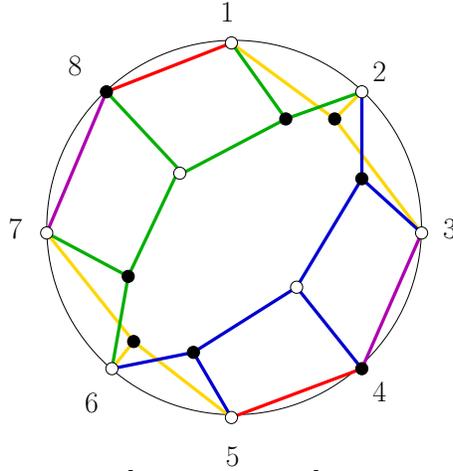}}
\end{center}
\vspace{-.2in}

    \caption{For the signature 
$\sigma=[\circ \circ \circ \bullet \circ \circ \circ \,\bullet]$, 
the set of coefficient invariants is
$\cc_\sigma=\{J_2^1, J_3^2, J_4^3, J_4^5, J_6^5, J_7^6, J_8^7,
J_8^1\}$.
}
    \label{fig:webs64}
\end{figure} 

In order to write exchange relations for the cluster
algebra~$R_\sigma(V)$, we will need certain 
\emph{$3$-term skein relations} for special invariants.


\begin{proposition}
\label{prop:3-term-skein-special}
Let $p,q,r$ (respectively $p,q,r,s$) be distinct boundary vertices,
listed in clockwise order.
Then 
\begin{align}
\label{eq:3-term-1}
J_{pqr} J^{pqr} &= J_r^p J_q^r J_p^q + J_r^q J_p^r J_q^p\,; \\
\label{eq:3-term-2}
J_p^r J_{qrs} &= J_q^r J_{prs} + J_s^r J_{pqr}\,;\\
\label{eq:3-term-3}
J_p^r J_s^q &= J_s^r J_p^q + J_{sp}^{qr}\,;\\
\label{eq:3-term-4}
J_p^r J_{rs}^{pq} &=J_p^q J_r^p J_s^r + J_{prs} J^{pqr}\,;\\
\label{eq:3-term-5}
J_r^p J_{sp}^{qr} &= J_r^q J_s^p J_p^r + J_{prs} J^{pqr}\,.
\end{align} 
\end{proposition}

It is important to note that identities 
\eqref{eq:3-term-1}--\eqref{eq:3-term-5}
hold irrespective of the
choice~of the signature~$\sigma$. 
For some of those choices, these identities are illustrated in
Figure~\ref{fig:webs56}. 

\begin{figure}[ht]
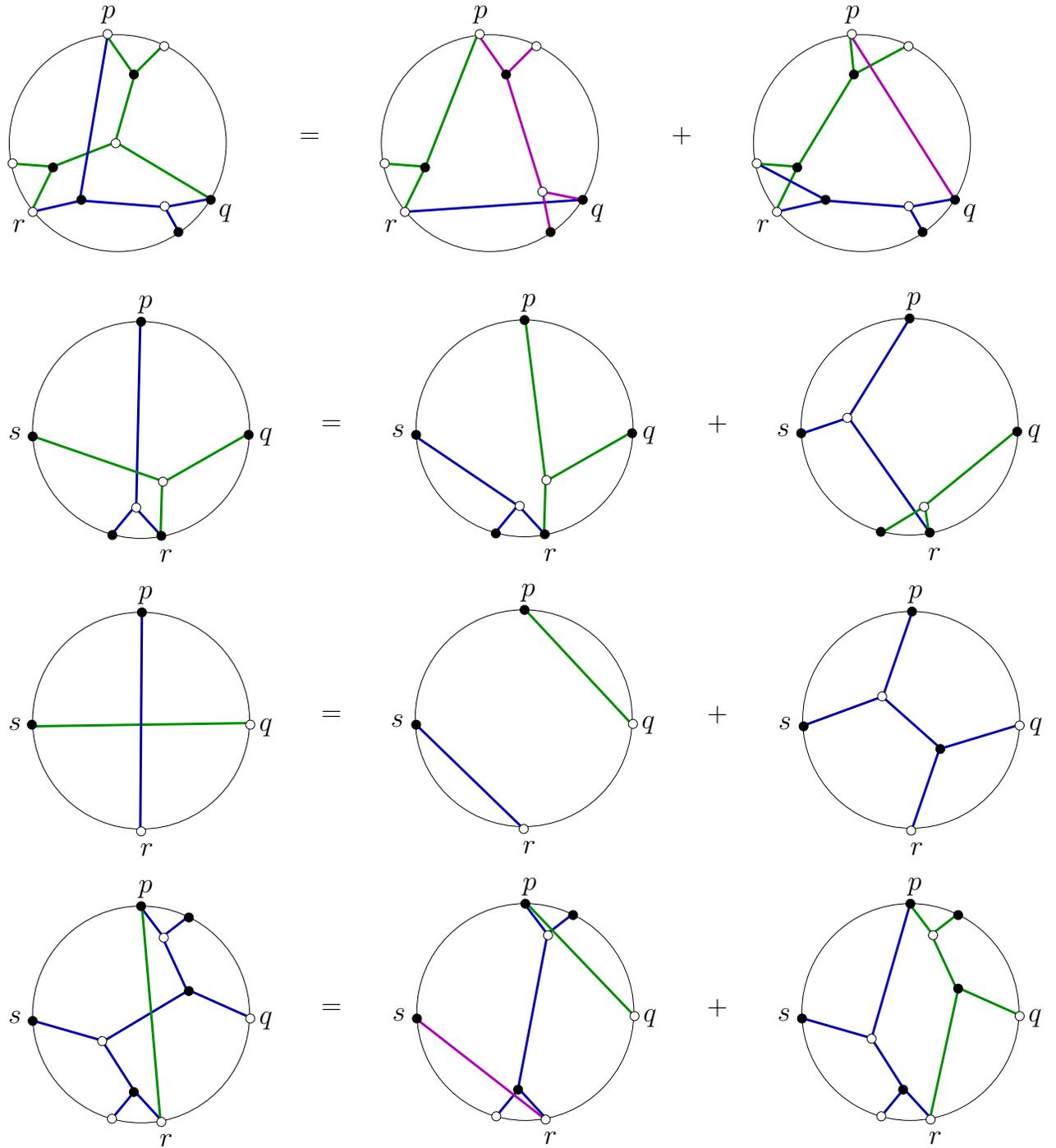

\smallskip

\scalebox{0.75}{\input{webs67a.pstex_t}}

\bigskip\bigskip

\scalebox{0.75}{\input{webs67b.pstex_t}}

\bigskip

\scalebox{0.75}{\input{webs67d.pstex_t}}

\bigskip

\scalebox{0.75}{\input{webs67c.pstex_t}}
\caption{$3$-term skein relations \eqref{eq:3-term-1}--\eqref{eq:3-term-4}.}
    \label{fig:webs56}
\end{figure}

\begin{remark}
\label{rem:distill}
In each of identities \eqref{eq:3-term-1}--\eqref{eq:3-term-5}, some special
invariants might be equal to~$0$, while others might factor further.
After everything is expressed in terms of irreducible special invariants, 
we either get a tautological formula $A=A$, 
or else a genuine \hbox{$3$-term} relation. 
We call the latter relation the \emph{distilled form} of the original
one. 
While the distilled relation is always a skein relation of some sort,
it does not have to be an instance of 
\eqref{eq:3-term-1}--\eqref{eq:3-term-5}; 
an example is given in~\eqref{eq:j17*j258}. 
\end{remark}

\begin{example}
\label{example:distill}
Let $\sigma =[\circ \circ \circ \bullet \circ
 \circ \circ \, \bullet]$ as before.
One instance of equation \eqref{eq:3-term-1} is
$J_{127}J^{127}=J_7^1J_2^7J_1^2+J_7^2J_1^7J_2^1$. 
This does not yield a nontrivial identity: 
after substituting $J_1^2=0$, 
$J_{127}=J_8^2 J_2^1 J_1^7$, $J^{127}=J_1^7$, and
 $J_7^2=J_8^2J_1^7$, everything cancels~out. 

A more interesting example is 
$J_7^2 J_{258} = J_8^2 J_{257} + J_5^2 J_{278}$
(cf.~\eqref{eq:3-term-2}). 
Substituting the factorizations $J_7^2=J_8^2J_1^7$ 
and $J_{278}=J_8^7 J_8^2 J_2^1$ and simplifying, we get  
\begin{equation}
\label{eq:j17*j258}
J_1^7 J_{258} =
 J_8^7J_5^2J_2^1+J_{257}\,. 
\end{equation}
\end{example}

\pagebreak[3]

\section{Seeds associated to triangulations}
\label{sec:special-seeds}

In this section, we construct a family of distinguished seeds in
(the quotient field~of) the ring~$R_\sigma(V)\cong
R_{a,b}(V)$. 
These seeds will be used in Section~\ref{sec:main-theorem}
to define a cluster algebra structure in~$R_\sigma(V)$. 


\begin{definition}
\label{def:z(T)}
Let $T$ be a triangulation of our $(a+b)$-gon by its diagonals. 
Let $K(T)$ be the collection of special invariants
built as follows: 
\begin{itemize}
 \item for each diagonal or side $pq$ in $T$ , include $J_p^q$ and $J_q^p$;
 \item for each triangle $pqr$ in $T$ (here $p,q,r$ are ordered clockwise), 
include $J_{pqr}$.
\end{itemize}
The \emph{extended cluster} $\zz(T)$ 
consists of all irreducible special invariants which appear in
factorizations of nonzero elements of $K(T)$ into irreducibles. \linebreak[3]
The \emph{cluster} $\xx(T)=\zz(T)\setminus\cc_\sigma$ 
consists of all non-coefficient invariants in~$\zz(T)$. 
\end{definition}

\begin{theorem}
\label{th:z(T)}
For any triangulation~$T$ as above, the following holds: 
\begin{itemize}
\item
the extended cluster $\zz(T)$ contains the entire set $\cc_\sigma$ of
coefficient invariants; 
\item
$\zz(T)$ consists of $3(a+b)-8$ invariants;
\item
all special invariants in $\zz(T)$ are pairwise compatible. 
\end{itemize}
\end{theorem}

\begin{example}
\label{example:K(T)}
Let $\sigma=[\circ \circ \circ \bullet \circ \circ \circ \,\bullet]$.
Then
\[
\cc_\sigma=\{J_2^1, J_3^2, J_4^3, J_4^5, J_6^5, J_7^6, J_8^7,
J_8^1\}
\]
(see Figure~\ref{fig:webs64}). 
For~the triangulation~$T$ shown in Figure~\ref{fig:webs60} on the
left, we have
\[
K(T)= \cc_\sigma \cup 
\{J_7^1, J_1^7, J_2^7, J_7^2, J_2^5, J_5^2, J_7^5, J_5^7, J_5^3,
J_3^5, J_{127}, J_{178}, J_{257}, J_{235}, J_{567}\}.
\]
Now, $J_{178}=J_{345}=0$ and $J_5^7=J_6^5$, 
while several other elements of $K(T)$
factor:
\[
J_7^1=J_8^1 J_1^7, {\ \ }
J_7^2=J_8^2J_1^7, {\ \ }
J_3^5=J_4^5J_5^3, {\ \ }
J_{127}=J_8^2 J_2^1 J_1^7, {\ \ }
J_{235}=J_5^3J_3^2, {\ \ }
J_{567}=J_6^5J_7^6. 
\]
We conclude that
$\zz(T)= \cc_\sigma\cup \xx(T)$ where  
\begin{equation}
\label{eq:x(T)-example}
\xx(T)=\{J_8^2, J_2^7, J_1^7, J_2^5, J_5^2, J_7^5, J_5^3, J_{257}\}.
\end{equation}

\vspace{-.05in}

\begin{figure}[ht]
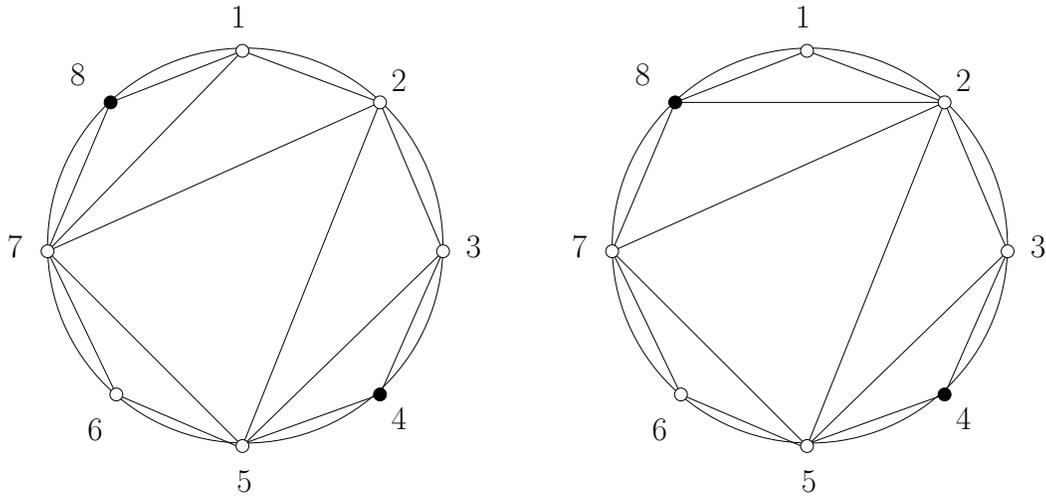

\begin{center}
\scalebox{0.6}{\input{webs60.pstex_t}}
\qquad\quad
\scalebox{0.6}{\input{webs60a.pstex_t}}
\end{center}
    \caption{Triangulations $T$ and $T'$ of an octagon,
$\sigma=[\circ \circ \circ \bullet \circ
    \circ \circ \, \bullet]$. 
}
    \label{fig:webs60}
\end{figure} 

Figure~\ref{fig:webs66} shows the webs corresponding to the elements
of~$\xx(T)$. 
\end{example}

\begin{figure}[ht]
\scalebox{0.65}{\input{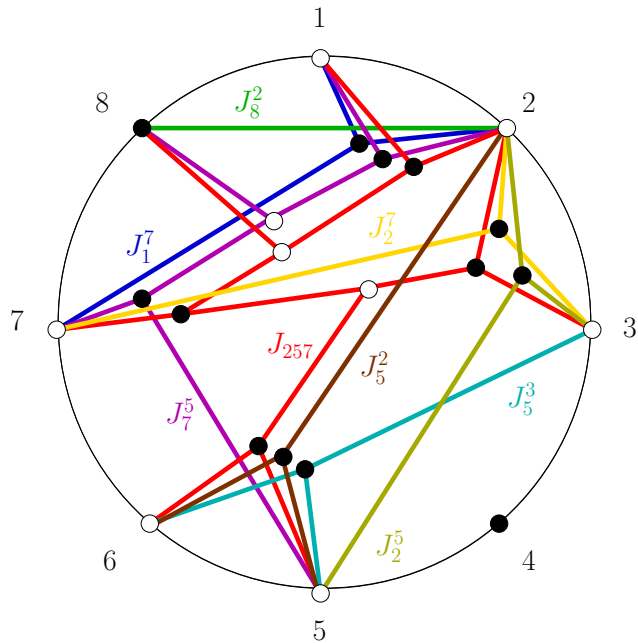}}
    \caption{Cluster associated with the triangulations shown
    in Figure~\ref{fig:webs60}.}
    \label{fig:webs66}
\end{figure}

Surprisingly, different triangulations may define the same 
cluster. 

\begin{proposition}
\label{pr:x(T)=x(T')}
Let $T$ and $T'$ be triangulations which contain the diagonal \hbox{$q(q+3)$}
and coincide outside the quadrilateral $q(q+1)(q+2)(q+3)$. 
If the colors of the vertices $(q,q+1,q+2)$ alternate, 
then $\xx(T)=\xx(T')$.
\end{proposition}

\begin{example}
\label{example:x(T)=x(T')}
The two triangulations in Figure~\ref{fig:webs60} yield
the same cluster. (Apply Proposition~\ref{pr:x(T)=x(T')} with $q\!=\!7$.)
To check this directly, refer to Example~\ref{example:K(T)}, 
write
\[
K(T')=\{J_8^2,J_2^8,J_{128},J_{278}\}\cup
K(T)\setminus\{J_7^1,J_1^7,J_{127},J_{178}\},
\]
and use the factorizations $J_2^8=J_8^2J_2^1$, $J_{128}=J_8^2 J_2^1$, 
and $J_{278}=J_8^7 J_8^2 J_2^1$. 
\end{example}

To complete our construction, 
we need to describe the quivers~$Q(T)$ which, 
together with the extended clusters~$\zz(T)$,
will form the seeds defining 
a cluster algebra structure in the ring~$R_\sigma$.
To put it differently, we need to write the corresponding 
exchange relations~\eqref{eq:exchange-relation}. 
These are obtained from the
$3$-term skein relations \eqref{eq:3-term-1}--\eqref{eq:3-term-5}
using a couple of additional observations. 

Let $T$ be a triangulation as in Definition~\ref{def:z(T)}. 
If $pqr$ and $prs$ are triangles in~$T$, then by construction 
the first factor on the left-hand side of each relation
\hbox{\eqref{eq:3-term-1}--\eqref{eq:3-term-5}}
is an element of~$K(T)$. 
In order for such a relation (or rather its distilled form, see
Remark~\ref{rem:distill}) to be a good candidate for an exchange
relation out of the extended cluster~$\zz(T)$,
we would like the terms on the right-hand side to factor into the
elements of~$\zz(T)$. 
For \eqref{eq:3-term-1}--\eqref{eq:3-term-2}, 
this is true by design;
for \eqref{eq:3-term-3}--\eqref{eq:3-term-5},
the terms $J_{sp}^{qr}$ and $J^{pqr}$
will only have the requisite property in some special cases. 

\pagebreak[3]

We say that a side $pq$ of a triangle $pqr$ 
is \emph{exposed} if it lies on the
boundary of the $(a+b)$-gon; 
in other words, $pq$ is exposed if $q=p+1$.

\begin{proposition}
\label{pr:thin-triangles}
Suppose a triangle~$pqr$ has an exposed side.
Then 
\[
J^{pqr}\in \{J_p^r, J_q^p, J_r^q, 
J_p^q J_q^r, J_q^r J_r^p, J_r^p J_p^q \}
\]
(which element of this set $J^{pqr}$ is equal to depends on the signature).
\end{proposition}

\begin{proposition}
\label{pr:thin-quads}
Let $(p,p+1,p+2,s)$ be distinct vertices of the $(a+b)$-gon. 
\begin{align}
\label{eq:3-term-3wb}
&\text{If $p$ is white and $p+1$ is black, 
then $J_{p+2}^p J_{p+1}^s = J_{p+1}^p J_{p+2}^s + J_p^s$.}\\
\label{eq:3-term-3bw}
&\text{If $p$ is black and $p+1$ is white, 
then $J_p^{p+2} J_s^{p+1} = J_p^{p+1} J_s^{p+2} + J_s^p$.}
\end{align}
\end{proposition}

\begin{definition}[\emph{Exchange relations for a cluster associated
      with a triangulation}]
\label{def:exch-rel-T}
Continuing in the setting of Definition~\ref{def:z(T)},
write the following identities:
\begin{itemize}
\item
for each triangle $pqr$ of the triangulation~$T$, 
write formula~\eqref{eq:3-term-1};
\item
for each diagonal $pr$ in~$T$ separating triangles $pqr$ and $prs$:
\begin{itemize}
\item 
write formula~\eqref{eq:3-term-2};
\item 
if one of the sides of $pqr$ is exposed, 
write~\eqref{eq:3-term-4}--\eqref{eq:3-term-5};
\item 
if two sides of $pqr$ are exposed, 
write the appropriate instance of~\eqref{eq:3-term-3wb}
or~\eqref{eq:3-term-3bw}, if applicable. 
\end{itemize}
\end{itemize}
It is easy to check, with the help of
Proposition~\ref{pr:thin-triangles}, that 
each of the two monomials on the right-hand side of each of resulting
relations will either vanish (in which case the relation is of no use
to us) or else factor into nontrivial irreducibles. 
In the latter case, we distill the relation
(see Remark~\ref{rem:distill}) to obtain a $3$-term relation in
irreducible special invariants all of which, with the exception of
the second factor on the left, will belong~to~$\zz(T)$. 

To obtain the final list of exchange relations,
we should also inspect all instances where we can apply
Proposition~\ref{pr:x(T)=x(T')} to get another triangulation~$T'$ with
the same cluster~$\xx(T')=\xx(T)$,
then check whether applying the above recipe to~$T'$ yields any
additional $3$-term relations. 
\end{definition}

A few simple observations help reduce the amount of work
required to write the $3$-term relations of
Definition~\ref{def:exch-rel-T} for a particular choice of~$T$
and~$\sigma$.  
For example, if a triangle $pqr$ has an exposed side, 
then one of the terms on the right-hand
side of~\eqref{eq:3-term-1} vanishes, 
by virtue of Proposition~\ref{prop:special=0}.
Thus, we only need to write relations~\eqref{eq:3-term-1}
for the triangles whose all sides are diagonals. 

\begin{proposition}
\label{prop:special-seeds-exchanges}
For any triangulation $T$ of the $(a+b)$-gon, the procedure described
in Definition~\ref{def:exch-rel-T} yields 
as many relations as there are elements in the cluster~$\xx(T)$. 
More precisely, it yields one relation of the form 
\[
x x'=M_1+M_2
\]
for each $x\in\xx(T)$; here $M_1,M_2$ are 
monomials in the elements of~$\zz(T)$. 
\end{proposition}

\pagebreak[3]

We are now ready to define the quiver associated
  with a given triangulation of the $(a+b)$-gon. 

\begin{definition}
\label{def:Q(T)}
The \emph{quiver $Q(T)$ associated
  with a triangulation~$T$}  is constructed so that the exchange relations
\eqref{eq:exchange-relation} match the $3$-term relations obtained via
  the process described in Definition~\ref{def:exch-rel-T}. 
This requirement defines $Q(T)$ up to simultaneous reversal of
  direction of all edges incident to any subset of connected components
  of the mutable part of the quiver. 
Typically, there will be a~single connected component,
so that $Q(T)$ is defined up to a global change of orientation. 
It is not hard to see that the choice of a particular incarnation of
$Q(T)$ is inconsequential, as the resulting cluster structure will 
not depend on this choice. 
To simplify presentation, we henceforth use the notation $Q(T)$
  without mentioning its ambiguity. 
\end{definition}

\begin{example}
We continue with the triangulation~$T$ 
discussed in Example~\ref{example:K(T)};
cf.\ Figure~\ref{fig:webs60} on the left. 
The only relevant instance of formula~\eqref{eq:3-term-1} is 
\[
J_{257} J^{257}
=J_7^2 J_5^7 J_2^5 + J_7^5 J_2^7 J_5^2 
\qquad \text{(for $p=2,\ q=5,\ r=7$).} 
\]
After factoring $J_7^2$ into irreducibles 
and substituting $J_5^7=J_6^5$ 
(so as to use consistent notation for the coefficient invariants), 
we obtain
\begin{equation}
\label{eq:j257*u257}
J_{257} J^{257}
= J_8^2 J_1^7 J_6^5 J_2^5 + J_7^5 J_2^7 J_5^2\,.
\end{equation}
Among 10 instances of~\eqref{eq:3-term-2}, 
the nontrivial identities are (in distilled form):
\begin{align}
\label{eq:j27*j58}
J_2^7 J_5^8 &= J_6^5 J_8^2 J_2^1+J_{257} 
\qquad \text{(for $p=2,\ q=5,\ r=7,\ s=1$),}\\[.1in] 
J_2^5 J_7^4 &= J_{257} J_4^5 + J_3^2 J_7^5 
\qquad \text{(for $p=2,\ q=3,\ r=5,\ s=7$),}\\[.1in] 
\label{eq:j52*j73}
J_5^2 J_7^3 &= J_8^2 J_1^7 J_5^3 + J_{257} 
\qquad \text{(for $p=5,\ q=7,\ r=2,\ s=3$),}\\[.1in] 
\label{eq:j75*j26}
J_7^5 J_2^6 &= J_2^5 J_7^6 + J_{257}\quad
\qquad \text{(for $p=7,\ q=2,\ r=5,\ s=6$).}
\end{align} 
For each of \eqref{eq:3-term-4} and \eqref{eq:3-term-5},
there are 7~instances to check. 
Equation \eqref{eq:3-term-4} produces one nontrivial relation: 
\begin{align}
J_8^2 J_{25}^{71} &= J_8^1 J_2^7 J_5^2 + J_{257}
\qquad \text{(for $p=7,\ q=1,\ r=2,\ s=5$),}
\end{align} 
Equation \eqref{eq:3-term-5} yields three relations---but all of them
replicate those obtained earlier: 
\begin{itemize}
\item
for $p=7,\ q=1,\ r=2,\ s=5$, we get~\eqref{eq:j27*j58};
\item
for $p=2,\ q=3,\ r=5,\ s=7$, we get~\eqref{eq:j52*j73};
\item
for $p=5,\ q=6,\ r=7,\ s=2$, we get~\eqref{eq:j75*j26}. 
\end{itemize} 
The rule \eqref{eq:3-term-3wb} supplies two equations
one of which results in a new nontrivial identity: 
\begin{equation}
\label{eq:j53*j42}
J_5^3 J_4^2 = J_4^3 J_5^2 + J_3^2 
\qquad \text{(for $p=3,\ s=2$);} \hspace{.75in}
\end{equation} 
and there are no instances of~\eqref{eq:3-term-3bw}.

We have obtained 7~exchange relations
\eqref{eq:j257*u257}--\eqref{eq:j53*j42},
one for each element of $\xx(T)$ (cf.~\eqref{eq:x(T)-example}) except
for~$J_1^7$. 
To get the missing relation, we need to replace $T$ by the
triangulation~$T'$ shown in Figure~\ref{fig:webs60} on the right. 
(Cf.\ Example~\ref{example:x(T)=x(T')}). 
The only new relation is obtained by applying~\eqref{eq:3-term-2}
with $p=7, q=8, r=2, s=5$.
As explained in Example~\ref{example:distill}, this results in 
the equation~\eqref{eq:j17*j258}: 
\begin{equation}
\label{eq:j17*j258-again}
J_1^7 J_{258} =
 J_8^7J_5^2J_2^1+J_{257}\,. 
\end{equation}

The corresponding quiver~$Q(T)$ has 16 vertices:
8~frozen, labeled by the coefficient
invariants (see Figure~\ref{fig:webs64}), 
and 8~mutable, labeled by the elements of the cluster~$\xx(T)$
(see~\eqref{eq:x(T)-example}).
One by one, we draw the edges incident to each mutable vertex
using the corresponding exchange relation 
\eqref{eq:j257*u257}--\eqref{eq:j17*j258-again}, 
so as to match the formula~\eqref{eq:exchange-relation}.
This occasionally requires swapping the two terms on the right-hand
side. 

\pagebreak[3]

The resulting quiver $Q(T)$ is shown in
Figure~\ref{fig:webs65}. 
It has cluster type~$E_8$. 
\end{example}

\begin{figure}[ht]
\scalebox{0.65}{\input{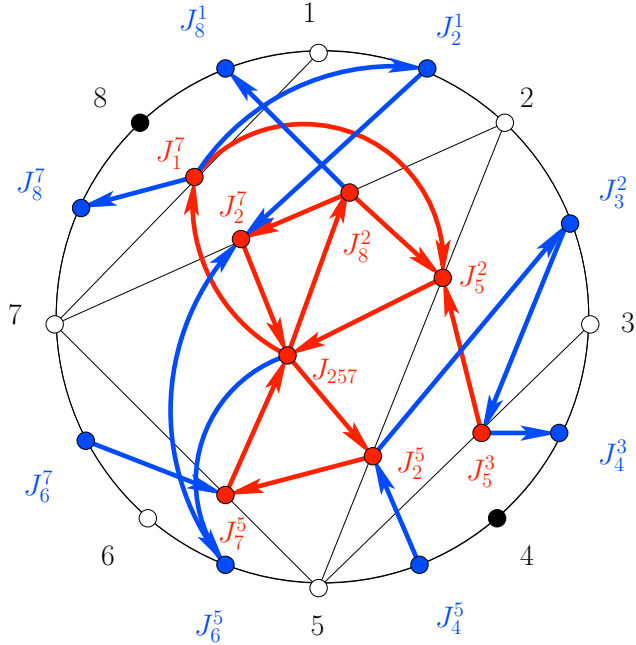}}
    \caption{The quiver associated with the triangulations shown
    in Figure~\ref{fig:webs60}.}
    \label{fig:webs65}
\end{figure}

A general recipe for building the quiver $Q(T)$ associated with an
arbitrary triangulation~$T$ of the polygon~$P_\sigma$ is described in
Section~\ref{sec:building-a-quiver}. 

Additional examples of seeds $(Q(T),\zz(T))$
associated with particular triangulations~$T$ and signatures~$\sigma$
can be found in Sections~\ref{sec:T-fan} and~\ref{sec:examples-of-seeds}.



\usection{Results and conjectures}

\section{Main results}
\label{sec:main-theorem}

After all the preparation of the preceding sections, the statement of
our main theorem comes as no surprise:
for any choice of a non-alternating cyclic signature,
the construction described in Section~\ref{sec:special-seeds}
makes the ring of
invariants~$R_\sigma(V)$ into a cluster algebra. 
Here is a precise statement: 


\begin{theorem}
\label{th:main}
Let $\sigma$ be a non-alternating cyclic signature of type $(a,b)$, 
let $P_\sigma$ be a convex $(a+b)$-gon with vertices colored
according to~$\sigma$, and 
let $T$ be a triangulation of~$P_\sigma$ by its diagonals.
Then the extended cluster~$\zz(T)$ and the quiver~$Q(T)$ 
described in Definitions~\ref{def:z(T)}
and~\ref{def:Q(T)} 
form a seed in the quotient field of
the ring of invariants~$R_\sigma(V)$. 
The associated cluster algebra is~$R_\sigma(V)$: 
\[
R_\sigma(V)=\Acal(Q(T),\zz(T)).
\]
This cluster structure on $R_\sigma(V)$
does not depend on the choice of triangulation~$T$.
\end{theorem} 

\begin{corollary}
Any element of $R_\sigma(V)$, and in particular any web invariant, can
be expressed as a Laurent polynomial in the elements of any extended
cluster~$\zz(T)$. 
\end{corollary}

For the record, we state below some basic properties of the cluster
algebra $R_\sigma(V)$ that are immediate from its construction 
(as justified by Theorem~\ref{th:main}). 

\begin{proposition}
The cluster algebra $R_\sigma(V)$ has the following properties:
\begin{enumerate}
\item
The set of 
coefficient variables and cluster variables 
includes all irreducible special invariants, and
in particular all Weyl generators. 
\item
The coefficient variables are the
coefficient invariants (cf.\ Proposition~\ref{prop:coeff-special}). 
\item
Each extended cluster consists of $3(a+b)-8$ elements; 
\item
The rank (the cardinality of each cluster)
is $2(a+b)-8$, 
except for the  cases $\sigma=[\bullet\circ\bullet\circ\bullet]$
and $\sigma=[\circ\bullet\circ\bullet\circ]$
when it is equal to~1
(cf.\ Remark~\ref{rem:xoxox}). 
\end{enumerate}
\end{proposition}

\begin{example}
Let $a=2$, $b=3$, $\sigma=[\bullet \bullet \bullet \circ \circ]$. 
In this case, the cluster algebra $R_\sigma(V)$ has type~$A_2$. 
It has $5$~coefficient variables, namely: 
$J_{123},\ J_{23}^{45},\ J_3^4,\ J_{12}^{45},\ J_1^5$. 
(Some of these have alternative presentations, for example $J_{23}^{45}=J_{234}=J^{245}=J_2^3$.) 
The $5$~cluster variables in $R_\sigma(V)$ are 
$J_1^4, \ J_2^4, \ J_2^5, \ J_3^5, \ J_{13}^{45}=J_{134}=J_1^3$. 
They form $5$~clusters:
\[
\{J_1^4, J_2^4\}, \{J_2^4, J_2^5\}, \{J_2^5, J_3^5\}, \{J_3^5, J_{13}^{45}\}, 
\{J_{13}^{45}, J_1^4\}. 
\]
The exchange relations are:
\begin{align*}
J_1^4 J_3^5 &= J_1^5 J_3^4 + J_{13}^{45}, \\
J_2^4 J_{13}^{45} &= J_{12}^{45} J_3^4 + J_{23}^{45} J_1^4, \\
J_2^5 J_1^4 &= J_{12}^{45} + J_1^5 J_2^4, \\
J_3^5 J_2^4 &= J_{23}^{45} + J_3^4 J_2^5, \\
J_{13}^{45} J_2^5 &= J_{23}^{45} J_1^5 + J_{12}^{45} J_3^5.
\end{align*}
Notice that the coefficient variable $J_{123}$ does not appear in these relations.
Accordingly, the frozen vertex corresponding to $J_{123}$ is isolated in every
quiver of~$R_\sigma(V)$. 
One example of such a quiver is shown in Figure~\ref{fig:xxxoo}. 
\end{example}

\begin{figure}[ht]
\begin{center}
\setlength{\unitlength}{2pt} 
\begin{picture}(80,22)(0,0) 

\put( 0,20){\makebox(0,0){$J_1^5$}}
\put(20,0){\makebox(0,0){$J_{12}^{45}$}}
\put(20,20){\makebox(0,0){$J_2^5$}}
\put(40,0){\makebox(0,0){$J_{23}^{45}$}}
\put(40,20){\makebox(0,0){$J_2^4$}}
\put(60,20){\makebox(0,0){$J_3^4$}}
\put(80,20){\makebox(0,0){$J_{123}$}}

\thicklines 

\put(16,20){\vector(-1,0){12}}
\put(56,20){\vector(-1,0){12}}
\put(24,20){\vector(1,0){12}}

\put(20,4){\vector(0,1){12}}
\put(40,16){\vector(0,-1){12}}

\end{picture}
\end{center}
\caption{The quiver associated 
with the cluster $\{J_2^4, J_2^5\}$ 
in the cluster algebra $R_\sigma(V)$ 
with $\sigma=[\bullet \bullet \bullet \circ \circ]$.}
\label{fig:xxxoo}
\end{figure}
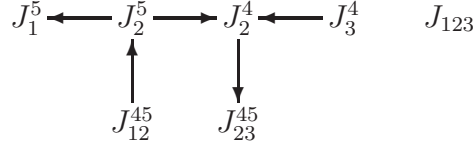

\pagebreak[3]

While the cluster structure on $R_\sigma(V)\cong R_{a,b}(V)$ defined
in Theorem~\ref{th:main} is
independent of the choice of triangulation, 
it does critically depend on the signature~$\sigma$. 
In fact, fixing the parameters $a$ and~$b$ does not determine 
the cluster type of~$R_\sigma(V)$, nor
whether $R_\sigma(V)$ is of finite or infinite type. 
See Figure~\ref{fig:boundary-signatures-5678}. 

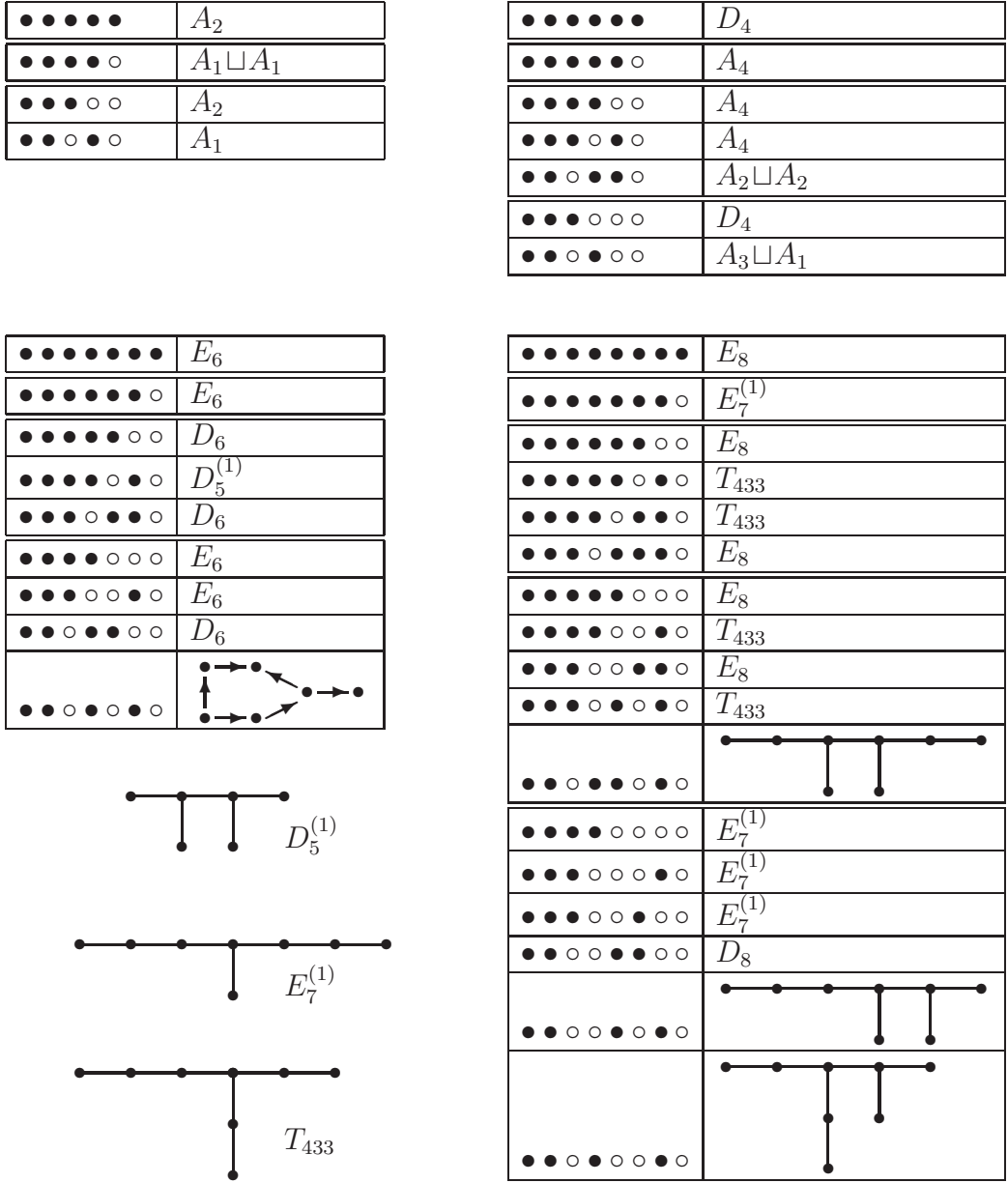
\begin{figure}[ht]
\vspace{-.05in}

\begin{center}
\begin{tabular}[t]{lll}
\begin{tabular}[t]{|p{2cm}|p{2.5cm}|}
\hline
$\bullet\bullet\bullet\bullet\bullet$ & $A_2$ \\
\hline
\hline
$\bullet\bullet\bullet\bullet\circ$ & $A_1\!\sqcup\!A_1$ \\
\hline
\hline
$\bullet\bullet\bullet\circ\circ$ & $A_2$ \\
\hline
$\bullet\bullet\circ\bullet\circ$ & $A_1$ \\
\hline
\end{tabular}
&&
\begin{tabular}[t]{|p{2.33cm}|p{3.8cm}|}
\hline
$\bullet\bullet\bullet\bullet\bullet\,\bullet$ & $D_4$ \\
\hline
\hline
$\bullet\bullet\bullet\bullet\bullet\,\circ$ & $A_4$\\
\hline
\hline
$\bullet\bullet\bullet\bullet\circ\,\circ$ & $A_4$ \\ 
\hline
$\bullet\bullet\bullet\circ\bullet\,\circ$& $A_4$ \\
\hline
$\bullet\bullet\circ\bullet\bullet\,\circ$& $A_2\!\sqcup\!A_2$\\
\hline
\hline
$\bullet\bullet\bullet\circ\circ\,\circ$ & $D_4$ \\ 
\hline
$\bullet\bullet\circ\bullet\circ\,\circ$& $A_3 \!\sqcup\! A_1$\\
\hline
\end{tabular}
\\
\\
\hskip-.07in\begin{tabular}[t]{l}
\begin{tabular}[t]{|p{2cm}|p{2.5cm}|}
\hline
$\bullet\bullet\bullet\bullet\bullet\bullet\bullet$ & $E_6$ \\
\hline
\hline
$\bullet\bullet\bullet\bullet\bullet\bullet\circ$ & $E_6$ \\ 
\hline
\hline
$\bullet\bullet\bullet\bullet\bullet\circ\circ$ & $D_6$ \\
\hline
$\bullet\bullet\bullet\bullet\circ\bullet\circ$ & $D_5^{(1)}$ \\
\hline
$\bullet\bullet\bullet\circ\bullet\bullet\circ$ & $D_6$ \\
\hline
\hline
$\bullet\bullet\bullet\bullet\circ\circ\circ$ & $E_6$ \\ 
\hline
$\bullet\bullet\bullet\circ\circ\bullet\circ$ & $E_6$ \\
\hline
$\bullet\bullet\circ\bullet\bullet\circ\circ$ & $D_6$ \\ 
\hline
$\bullet\bullet\circ\bullet\circ\bullet\circ$& 
\setlength{\unitlength}{2pt} 
\begin{picture}(33,13)(-3,0) 
\thicklines 
\put(0,0){\circle*{2}}
\put(10,0){\circle*{2}}
\put(0,10){\circle*{2}}
\put(10,10){\circle*{2}}
\put(20,5){\circle*{2}}
\put(30,5){\circle*{2}}
\put(2,0){\vector(1,0){6}}
\put(2,10){\vector(1,0){6}}
\put(0,2){\vector(0,1){6}}
\put(12,0){\vector(2,1){6}}
\put(18,6){\vector(-2,1){6}}
\put(22,5){\vector(1,0){6}}
\end{picture}
\\
\hline
\end{tabular}
\\
\begin{tabular}{c}
\setlength{\unitlength}{2pt} 
\begin{picture}(53,23)(-10,0) 
\thicklines 
\multiput(10,10)(10,0){4}{\circle*{2}}
\put(20,0){\circle*{2}}
\put(30,0){\circle*{2}}
\put(10,10){\line(1,0){30}}
\put(20,0){\line(0,1){10}}
\put(30,0){\line(0,1){10}}
\put(40,0){$D_5^{(1)}$}
\end{picture}
\\
\setlength{\unitlength}{2pt} 
\begin{picture}(53,27)(-10,0) 
\thicklines 
\multiput(0,10)(10,0){7}{\circle*{2}}
\put(30,0){\circle*{2}}
\put(0,10){\line(1,0){60}}
\put(30,0){\line(0,1){10}}
\put(40,0){$E_7^{(1)}$}
\end{picture}
\\
\setlength{\unitlength}{2pt} 
\begin{picture}(53,33)(-10,0) 
\thicklines 
\multiput(0,20)(10,0){6}{\circle*{2}}
\put(30,0){\circle*{2}}
\put(30,10){\circle*{2}}
\put(30,20){\circle*{2}}
\put(0,20){\line(1,0){50}}
\put(30,0){\line(0,1){20}}
\put(40,5){$T_{433}$}
\end{picture}
\end{tabular}
\end{tabular}
&\qquad\quad&
\begin{tabular}[t]{|p{2.33cm}|p{3.8cm}|}
\hline
$\bullet\bullet\bullet\bullet\bullet\bullet\bullet\,\bullet$ & $E_8$ \\
\hline
\hline
$\bullet\bullet\bullet\bullet\bullet\bullet\bullet\,\circ$ & $E_7^{(1)}$ \\
\hline
\hline
$\bullet\bullet\bullet\bullet\bullet\bullet\circ\,\circ$ & $E_8$ \\ 
\hline
$\bullet\bullet\bullet\bullet\bullet\circ\bullet\,\circ$ & $T_{433}$\\
\hline
$\bullet\bullet\bullet\bullet\circ\bullet\bullet\,\circ$ & $T_{433}$ \\
\hline
$\bullet\bullet\bullet\circ\bullet\bullet\bullet\,\circ$ & $E_8$ \\ 
\hline
\hline
$\bullet\bullet\bullet\bullet\bullet\circ\circ\,\circ$ & $E_8$ \\ 
\hline
$\bullet\bullet\bullet\bullet\circ\circ\bullet\,\circ$ & $T_{433}$ \\
\hline
$\bullet\bullet\bullet\circ\circ\bullet\bullet\,\circ$ & $E_8$\\
\hline
$\bullet\bullet\bullet\circ\bullet\circ\bullet\,\circ$ & $T_{433}$ \\
\hline
$\bullet\bullet\circ\bullet\bullet\circ\bullet\,\circ$ & 
\setlength{\unitlength}{2pt} 
\begin{picture}(53,13)(-2,0) 
\thicklines 
\multiput(0,10)(10,0){6}{\circle*{2}}
\put(20,0){\circle*{2}}
\put(30,0){\circle*{2}}
\put(0,10){\line(1,0){50}}
\put(20,0){\line(0,1){10}}
\put(30,0){\line(0,1){10}}
\end{picture} 
\\ 
\hline
\hline
$\bullet\bullet\bullet\bullet\circ\circ\circ\,\circ$ & $E_7^{(1)}$\\
\hline
$\bullet\bullet\bullet\circ\circ\circ\bullet\,\circ$ & $E_7^{(1)}$\\ 
\hline
$\bullet\bullet\bullet\circ\circ\bullet\circ\,\circ$ & $E_7^{(1)}$\\ 
\hline
$\bullet\bullet\circ\circ\bullet\bullet\circ\,\circ$ & $D_8$\\
\hline
$\bullet\bullet\circ\circ\bullet\circ\bullet\,\circ$ & 
\setlength{\unitlength}{2pt} 
\begin{picture}(53,13)(-2,0) 
\thicklines 
\multiput(0,10)(10,0){6}{\circle*{2}}
\put(40,0){\circle*{2}}
\put(30,0){\circle*{2}}
\put(0,10){\line(1,0){50}}
\put(40,0){\line(0,1){10}}
\put(30,0){\line(0,1){10}}
\end{picture} 
\\
\hline
$\bullet\bullet\circ\bullet\circ\circ\bullet\,\circ$ & 
\setlength{\unitlength}{2pt} 
\begin{picture}(53,23)(-2,0) 
\thicklines 
\multiput(0,20)(10,0){5}{\circle*{2}}
\put(20,0){\circle*{2}}
\put(20,10){\circle*{2}}
\put(30,10){\circle*{2}}
\put(0,20){\line(1,0){40}}
\put(20,0){\line(0,1){20}}
\put(30,10){\line(0,1){10}}
\end{picture}
\\ 
\hline
\end{tabular}
\end{tabular}
\end{center}
    \caption{Cluster types of $R_\sigma(V)$ 
for non-alternating signatures $\sigma$ of type $(a,b)$ with 
$a\!+\!b\in\{5,6,7,8\}$.
These cluster types are
invariant under dihedral symmetries and/or global
change of coloring; 
we include one representative from each equivalence class.
Take an arbitrary orientation of each tree to get the corresponding quiver.}
    \label{fig:boundary-signatures-5678}
\end{figure}

For each signature~$\sigma$,
there are choices of a triangulation~$T$ for which the construction of the
quiver $Q(T)$ simplifies considerably.  
One such choice is presented in Section~\ref{sec:T-fan}, 
providing an explicit rule for determining
the (extended) cluster type of~$R_\sigma(V)$ for any signature~$\sigma$.




We next discuss the important special case $a=0$ (ring of invariants of $b$
vectors in \hbox{$3$-space}),
or equivalently the case of monochromatic signature
\[
\sigma=[\bullet\bullet\bullet\cdots\bullet\bullet\,\bullet]
\]
of type~$(0,b)$.
As mentioned in Section~\ref{sec:cluster-algebras}, 
cluster structures in the rings~$R_{0,b}(V)$,
and indeed in homogeneous coordinate rings of arbitrary Grassmannians
${\mathrm{Gr}}_{k,b}$, 
were described by J.~Scott~\cite{scott} and extensively studied thereafter
(cf., e.g.,~\cite{gsv-book}). 

\pagebreak[3]

\begin{theorem}
\label{th:we=scott}
The cluster algebra structure in the ring of invariants $R_{0,b}(V)$
(or equivalently in the homogeneous coordinate ring of the
Grassmannian ${\mathrm{Gr}}_{3,b}$)
described in Theorem~\ref{th:main} 
coincides with the one given by J.~Scott~{\rm\cite{scott}}. 
\end{theorem}

In the Grassmannian case $a=0$, our construction of seeds
$(Q(T),\zz(T))$
simplifies considerably,
see Figure~\ref{fig:webs47}. 
This construction 
generalizes the one given by Scott~\cite{scott} for a
particular kind of triangulation~$T$. 

\pagebreak[3]

\begin{figure}[ht]
    \begin{center}
\scalebox{1.1}{ \input{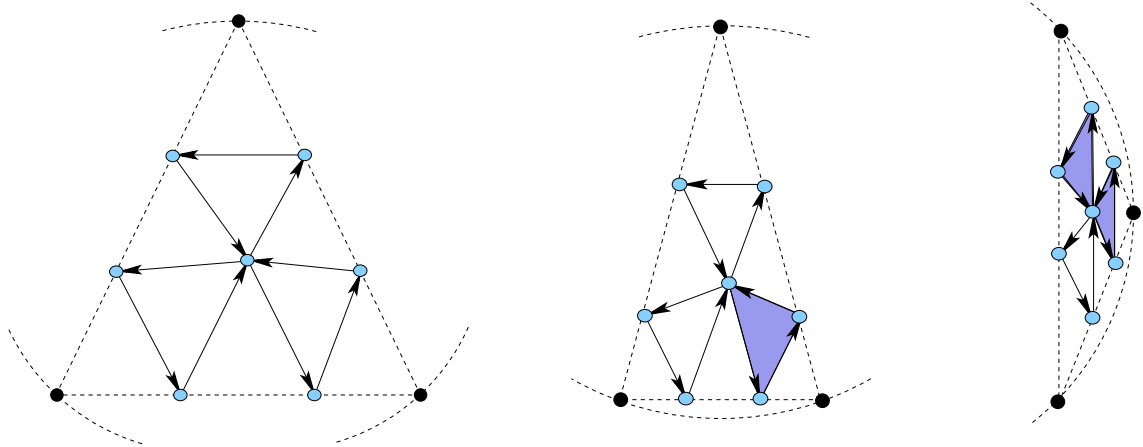}}
    \end{center} 
\caption{Given a triangulation~$T$ of  a convex $b$-gon, 
the seed $(Q(T),\zz(T))$ in the ring of invariants $R_{0,b}(V)$
is constructed as follows.
Place one vertex of~$Q(T)$ inside each triangle $pqr$ of~$T$; 
it corresponds to the Pl\"ucker invariant~$P_{pqr}$. 
Place two vertices of~$Q(T)$ on each diagonal $pq$ of~$T$ 
and on each side $pq$ of the polygon;
the vertex closer to~$p$ (resp., to~$q$)
corresponds to $P_{p,p+1,q}$ 
(resp., to~$P_{p,q,q+1}$). 
For every triangle in~$T$ with no exposed sides, 
connect the seven associated vertices as shown on the left. 
If one or two sides are exposed (second and third pictures), 
identify the vertices of each solid triangular region,
and remove the arrows bounding it. 
Finally, freeze the variables~$P_{p,p+1,p+2}$.}
    \label{fig:webs47}
\end{figure}

\begin{example}
\label{example:G3-678}
For $b\le 8$, the cluster algebras
$R_{0,b}(V)\cong\CC[{\mathrm{Gr}}_{3,b}]$ are of finite type,
so we can list all their generators explicitly.
Unsurprisingly, we end up reproducing Scott's original description
of these generators~\cite{scott} in the language of webs. 
In fact, Scott gave a geometric interpretation of these generators
(for $b\le 8$) which can be seen to match our construction. 



The cluster types of the rings $\CC[{\mathrm{Gr}}_{3,b}]$, $b\le 8$,
are well known to be as follows: 
$A_2$ for $b=5$, 
$D_4$ for $b=6$,
$E_6$~for $b=7$, and
$E_8$ for $b=8$ (see, e.g., \cite{cdm};
cf.\ Figure~\ref{fig:boundary-signatures-5678}). 
The set of generators of each cluster algebra $R_{0,b}(V)$ 
includes the $\binom{b}{3}$ Pl\"ucker invariants $P_{pqr}=J_{pqr}$ given by
tripod~webs. 
Among them there are $b$ coefficient variables of the form $P_{p,p+1,p+2}$. 

For $b\ge 6$, there are also non-Pl\"ucker cluster variables. 
The simplest among them are the ``hexapod'' invariants
given by webs of the kind shown in Figure~\ref{fig:webs50} on the~left. 
There are $2\binom{b}{6}$ of these hexapods in $R_{0,b}(V)$. 
For $b=6$ and $b=7$, the tripods and hexapods comprise the entire list
of generators. 
For $b=8$, there are 24 additional cluster variables,
namely 8 (resp.,~16) web invariants of the kind shown in the second (resp.,
third) picture in Figure~\ref{fig:webs50}. 

\end{example}

\begin{figure}[h!]
    \begin{center}
\scalebox{1.25}{ \input{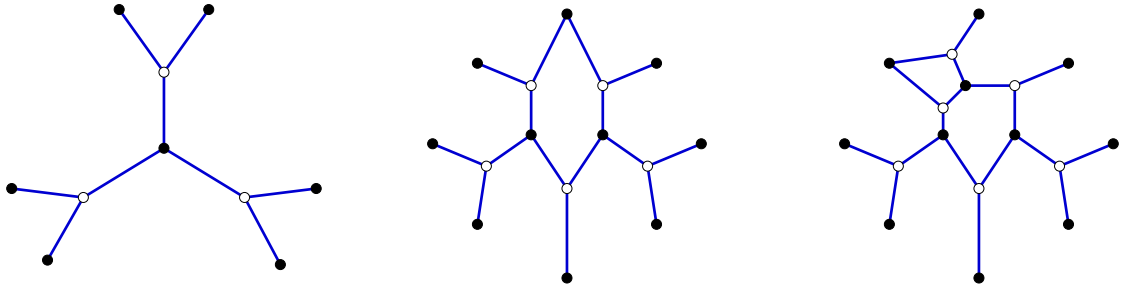}}
    \end{center} 
    \caption{Non-Pl\"ucker cluster variables in
$R_{0,b}(V)$, for $b\in\{6,7,8\}$. 
To get all of them, include rotations and, in the case of 
the~rightmost web, the mirror images.}
    \label{fig:webs50}
\end{figure}

\medskip

Cluster algebra structures in rings of invariants $R_\sigma(V)$ behave
nicely under two types of embeddings of such rings into one another. 
To state these results, we will need the following natural
definition. 

\begin{definition} 
Let $\Acal$ be a cluster algebra.
Take a seed $(Q,\zz)$ of~$\Acal$. 
Freeze some subset of cluster variables in~$\zz$ 
(thus, some subset of mutable vertices in~$Q$). 
If the resulting quiver has vertices that are not connected by an
edge to a mutable vertex, remove some of them (and remove the
corresponding elements from~$\zz$) to get a new seed $(Q',\zz')$. 
The cluster algebra $\Acal'=\Acal(Q',\zz')$ obtained in this way is called a
\emph{cluster subalgebra} of~$\Acal$. 
\end{definition}

\begin{theorem}
\label{th:drop-vertex}
Let $\sigma$ and $\sigma'$ be two non-alternating signatures 
such that $\sigma'$ is obtained from $\sigma$ by removing a single
symbol. 
Then the image of $R_{\sigma'}(V)$ under the obvious 
tautological embedding $R_{\sigma'}(V)\to R_\sigma(V)$
is a cluster subalgebra of~$R_\sigma(V)$. 
\end{theorem}

Theorem~\ref{th:drop-vertex} seems to be new even in the case of
Grassmannians (cf.\ Theorem~\ref{th:we=scott}). 
For example, it implies that any web invariant of the form shown in 
Figure~\ref{fig:webs50} is a cluster variable in $R_{0,b}(V)$ for any
$b\ge 8$, or
indeed in any cluster algebra $R_{a,b}(V)$ with $b\ge 6$,
respectively $b\ge 8$. 

\medskip

Recall that $V$ is a three-dimensional vector space endowed with a
volume form. 
The \emph{cross product} of two vectors $u,v\in V$ is the covector
$u\times v$ given by $(u\times v)(w)=\vol(u,v,w)$, for $w\in V$. 
One similarly defines a cross product of two covectors in~$V^*$,
which is a vector. 

\pagebreak[3]

\begin{theorem}
\label{th:fork}
Let $\sigma$ and $\sigma'$ be two non-alternating signatures 
such that $\sigma'$ is obtained from $\sigma$ by 
replacing two consecutive entries of the same color by a single entry
of the opposite color. 
Interpreting this operation algebraically as a cross product, 
consider the corresponding embedding
$R_{\sigma'}(V)\to R_\sigma(V)$. 
Then the image of $R_{\sigma'}(V)$ under this embedding 
is a cluster subalgebra of~$R_\sigma(V)$. 
\end{theorem}

Theorems~\ref{th:drop-vertex} and~\ref{th:fork} can be restated as
saying that if $\sigma'$ is obtained from~$\sigma$ 
by removing an entry, or by replacing $\bullet\,\bullet\!\to\!\circ$
or $\circ\,\circ\!\to\!\bullet$, or more generally by a sequence of such
steps, then $R_{\sigma'}(V)$ can be viewed as a cluster subalgebra
of~$R_\sigma(V)$.
In particular, in each of those cases the cluster type of~$R_{\sigma'}(V)$
can be obtained by taking an induced subgraph inside 
some quiver for~$R_{\sigma}(V)$. 

One consequence of this hierarchy is that each cluster algebra
$R_{\sigma}(V)$ can be interpreted as a cluster subalgebra of some
Grassmann cluster algebra $R_{0,b}(V)$. 

At the level of tensor diagrams, the embeddings in
Theorems~\ref{th:drop-vertex}--\ref{th:fork} correspond to 
simple transformations near the boundary of our disk. 
In the case of Theorem~\ref{th:drop-vertex}, 
passing from $\sigma'$ to~$\sigma$ corresponds to adding an isolated
vertex.
For the cross product embedding of Theorem~\ref{th:fork},
we replace one boundary vertex by a pair of vertices 
of the opposite color, simultaneously 
adding \emph{forks} to all edges of a tensor diagram which use this
vertex, see 
Figure \ref{fig:webs17}. 

\begin{figure}[ht]
    \begin{center}
    \input{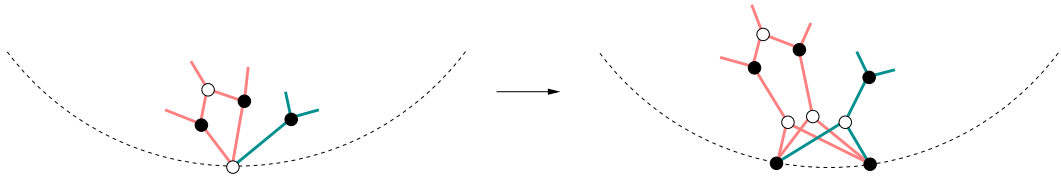} 
    \end{center} 
\vspace{-.1in}
    \caption{Adding a fork. 
}
    \label{fig:webs17}
\end{figure}


Theorems~\ref{th:drop-vertex} and~\ref{th:fork} can be 
used to obtain the following corollary. 

\begin{corollary}
\label{cor:planar-tree}
If a tensor diagram $D$ is a planar tree, 
then the corresponding web invariant $[D]$ is a cluster or coefficient
variable in~$R_{\sigma}(V)$.  
\end{corollary}

An observant reader may have noticed that our definition of the cluster structure in 
the ring $R_\sigma(V)$ does not treat the black and white colors in exactly the same way. 
The only place where the symmetry is broken is Definition~\ref{def:z(T)}:
the set $K(T)$ used to define the extended cluster $\zz(T)$ includes 
the invariants $J_{pqr}$---but not~$J^{pqr}$. 

\begin{theorem}
\label{th:swap-colors}
Interchanging the two colors in the definition of the extended clusters $\zz(T)$ 
does not change the resulting cluster structure in the ring~$R_\sigma(V)$.
\end{theorem}

Besides swapping the colors, one can also consider the ``mirror image'' 
of our main construction.
We expect this version to yield the same result. 

\begin{conjecture}
Reversal of direction (more specifically, replacing the clockwise direction 
in the definition of special invariants by the counterclockwise one)
does not affect the cluster structure in~$R_\sigma(V)$.
\end{conjecture}

We do not know of a natural 
construction of a cluster algebra structure in the ring~$R_\sigma(V)$ 
for the \emph{alternating} signature~$\sigma=[\bullet\circ\bullet\circ\cdots]$; 
this case seems to be genuinely exceptional. 
If $\sigma$ is alternating, 
then the structure one gets by applying the approach of
Sections~\ref{sec:special-invariants}--\ref{sec:special-seeds}
is one of a ``generalized cluster
algebra'' in which some of the exchange relations have more than two
terms on the right-hand side. 


\section{Main conjectures}
\label{sec:main-conjectures}

In this section, we discuss conjectural links between 
\begin{itemize}
\item
the cluster structure in a ring of invariants $R_\sigma(V)$ 
(cf.\ Section~\ref{sec:main-theorem}) and 
\item
Kuperberg's basis of this ring 
formed by the web invariants (cf.\ Section~\ref{sec:webs}). 
\end{itemize}
The conjectures in this section
are intentionally formulated in general terms, 
so as to suggest their possible extensions to other contexts, 
including other classical rings of invariants and 
cluster algebras associated with marked surfaces. 

\medskip

We call a web invariant~$z$ \emph{indecomposable} if it cannot be expressed 
as a product of two web invariants. 
(This is likely equivalent to irreducibility of~$z$.)
For example, any coefficient variable is an indecomposable web invariant.  

\begin{conjecture}
\label{conj:cluster-vars-are-in-web-basis}
All cluster variables 
are indecomposable web invariants. 
\end{conjecture}

Two cluster variables are called \emph{compatible} 
if they belong to the same cluster. 
Coefficient variables are compatible with each other, and with all
cluster variables. 

\begin{conjecture}
\label{conj:cluster-compatibility}
Two cluster (or coefficient) 
variables are compatible if and
only if their product is a web invariant. 
\end{conjecture} 

Note that by Proposition~\ref{prop:coeff-special}, 
the product of a coefficient 
variable and any web invariant
(in particular, by
Conjecture~\ref{conj:cluster-vars-are-in-web-basis}, any cluster
variable) is a web invariant. 

The simplest illustration of 
Conjecture~\ref{conj:cluster-compatibility} can be found in
Figure~\ref{fig:tripod-and-a-stick}: 
the invariants $J_2^4$ and $J_{123}$ shown there
 are compatible with each other,
and their product is a web invariant shown in
Figure~\ref{fig:edge-labelings}. 
To see more interesting examples, 
take any two (compatible) webs in Figure~\ref{fig:webs66}
and verify that their superposition can be converted into a single web 
by iterated skein transformations.

We conjecture that much more is true.
Define a \emph{cluster monomial} (cf.~\cite{cdm}) 
as a monomial in the elements of any given extended cluster. 
For cluster algebras defined by quivers, 
cluster monomials are known to be linearly independent~\cite{cklp}.

\begin{conjecture}
\label{conj:cluster-monomials-are-in-web-basis}
All cluster monomials are web invariants. 
\end{conjecture}

Conjecture~\ref{conj:cluster-monomials-are-in-web-basis} suggests the
following property of the web basis.

\begin{conjecture}
\label{conj:compat-web-invariants}
Given a finite collection of distinct web invariants,
if the product of any two of them is a web invariant, 
then so is the product of all of them. 
\end{conjecture} 

\pagebreak[3]

In Conjecture~\ref{conj:compat-web-invariants},
we cannot replace ``the product of all of them'' by ``any monomial in them:'' 
it is possible to construct a web invariant 
(see Section~\ref{sec:imaginary-elts})
whose square is not a web invariant.
This invariant is \emph{not} a cluster monomial, 
or else its square would provide a counterexample to 
Conjecture~\ref{conj:cluster-monomials-are-in-web-basis}.

We note that it is enough to prove Conjecture~\ref{conj:compat-web-invariants} 
for collections of three invariants. 
The general case would then follow by induction. 

Conjecture~\ref{conj:compat-web-invariants} is inspired by the 
``flag property'' of cluster complexes (conjectural, but proved in many instances, see
\cite[Conjecture 5.5, Theorem 5.6]{cats1}),
and by the analogous property of the dual canonical basis 
(cf.~\cite{hernandez}, \cite[Corollary 1.4]{leclerc-hernandez}). 

We anticipate that 
all of the above conjectures generalize broadly. 
Our most optimistic hopes are expressed in
Conjecture~\ref{conj:A-and-B} whose statements
\eqref{enum:cluster-mon-in-B}--\eqref{enum:clique}
are patterned after
Conjectures~\ref{conj:cluster-compatibility}--\ref{conj:compat-web-invariants}
and Proposition~\ref{prop:coeff-special}. 

\pagebreak[3]

\begin{conjecture}
\label{conj:A-and-B}
In any cluster algebra~$\Acal$ of geometric type,
there exists an additive basis~$\Bcal$ with the following properties:
\begin{enumerate}
\item
\label{enum:cluster-mon-in-B}
All cluster monomials in $\Acal$ lie in~$\Bcal$. 
\item
Two cluster variables are compatible if and
only if their product lies in~$\Bcal$. 
\item
The product of a coefficient variable and an element of~$\Bcal$ lies
in~$\Bcal$.
\item
\label{enum:clique}
If $B$ is a finite subset of $\Bcal$ such that the product of any two
distinct elements of~$B$ lies in~$\Bcal$,
then the product of all elements of~$B$ lies in~$\Bcal$. 
\end{enumerate}
\end{conjecture}

It is already nontrivial to show the existence of a basis~$\Bcal$
satisfying a small subset of the conditions
in Conjecture~\ref{conj:A-and-B}.
For example, the existence of a basis $\Bcal$
satisfying~\eqref{enum:cluster-mon-in-B} is equivalent 
to linear independence of cluster monomials, cf.~\cite{cklp}. 


Perhaps even more important is the 
problem of determining a cluster structure in a given ring
which is in some sense ``compatible'' with a
particular distinguished basis~$\Bcal$. 
In the case of the web basis, we used the properties 
\eqref{enum:cluster-mon-in-B}--\eqref{enum:clique}
in Conjecture~\ref{conj:A-and-B} as the guiding principles in
designing the cluster structure in the rings of
invariants~$R_\sigma(V)$. 
One cannot help but wonder whether this kind of 
``reverse engineering'' process can be made algorithmic, or at least
axiomatic. 
A~couple of approaches to this problem are outlined below. 

At the heart of the matter lies this question: 
What distinguishes cluster variables among other indecomposable
elements of the basis?\footnote{Unfortunately, 
in the case of the web basis, one cannot use B.~Leclerc's ``reality check''
criterion \cite{hernandez-leclerc, leclerc} 
according to which the cluster monomials are identified as the \emph{real} 
basis elements, i.e., those whose square lies in the basis. 
There exist web invariants which are not cluster monomials,
yet their squares are web invariants; 
see Section~\ref{sec:thickening-non-cluster}.}
One conjectural answer involves an extension of the notion of 
compatibility of cluster variables. 
Let us call a set of distinct indecomposable web invariants 
a \emph{clique} if their product 
(variant: the product of any two of them, cf.\ Conjecture~\ref{conj:compat-web-invariants}) 
is again a web invariant. 
By Conjecture~\ref{conj:cluster-monomials-are-in-web-basis},
any extended cluster is a clique. 

\begin{conjecture}
\label{conj:largest-clique}
Among all cliques, extended clusters are precisely
the ones of the largest cardinality. 
Thus, an indecomposable web invariant~$z$ is a cluster variable if 
and only if the largest clique containing~$z$ 
has the cardinality of an extended cluster. 
\end{conjecture}

While Conjecture~\ref{conj:largest-clique} may be aesthetically pleasing,
it is rather impractical as a tool for
determining whether a particular basis element is a cluster
variable. \linebreak[3] 
It~also is of little help in finding the exchange relations, or
equivalently the quivers accompanying extended clusters. 
For that, we need to move beyond purely multiplicative properties 
of web invariants 
(to put it differently, beyond \hbox{2-term} relations they satisfy)
into the realm of \emph{3-term relations}. 

\begin{conjecture}
\label{conj:3-term-cluster}
Suppose that indecomposable web invariants $z$ and $z'$ satisfy the 3-term relation 
\begin{equation}
\label{eq:3-term-general}
zz'=M_1+M_2
\end{equation} 
such that $M_1$, $M_2$, $zM_1$, $zM_2$, $z'M_1$, $z'M_2$, and $M_1M_2$ are web invariants.  
Then $z$ and $z'$ are cluster variables, and \eqref{eq:3-term-general} is an exchange relation. 
\end{conjecture}


Note that by Conjectures~\ref{conj:cluster-vars-are-in-web-basis}--\ref{conj:cluster-monomials-are-in-web-basis},
we expect each exchange relation in our cluster algebra to satisfy
the conditions in Conjecture~\ref{conj:3-term-cluster}. 

There is also a test that can (conjecturally) disqualify a web invariant
from being a cluster variable. 

\begin{conjecture}
\label{conj:3-term-non-cluster}
Suppose that an indecomposable web invariant $z$ and 
a cluster variable $z'$ satisfy 
a 3-term relation~{\rm\eqref{eq:3-term-general}}
such that $M_1$, $M_2$, $z'M_1$, and $z'M_2$ are web invariants 
whereas $zM_1$ and $zM_2$ are not.
Then $z$ is not a cluster variable. 
\end{conjecture}

In Conjecture~\ref{conj:3-term-non-cluster}, 
the requirement that $z'M_1$, and $z'M_2$ be web invariants 
cannot be dropped---see Section~\ref{sec:fake-exchange} for a relevant
counterexample. 

\begin{remark}
We conjecture that \emph{any} cluster structure can be uniquely recovered 
from a suitably chosen additive basis~$\Bcal$ 
using appropriate versions of the criteria in
Conjectures~\ref{conj:largest-clique}--\ref{conj:3-term-non-cluster}.
\end{remark}


{}From its very inception, cluster theory was motivated by the desire to
better understand the
(dual) ``canonical'' bases in the corresponding rings. 
Recall that for most of the rings~$R_\sigma(V)$ 
the dual canonical basis is \emph{different} from the web basis---although
the two bases share many important features,
see Remark~\ref{rem:khovanov-kuperberg}. 
The canonical  basis is expected to 
have strong positivity properties, such as those spelled out 
in the conjecture below, which has long been part of the cluster algebras folklore. 

\begin{conjecture}[\emph{Strong Positivity Conjecture}]
\label{conj:strong-positivity}
In any cluster algebra $\Acal$ of geometric type,
there is an additive basis $\Ccal$
which includes the cluster monomials
and has nonnegative structure constants. 
\end{conjecture}

That is, any product of elements of the basis~$\Ccal$
should have nonnegative coefficients when expanded in the same basis. 
This condition suggests the existence of a monoidal categorification, 
wherein structure constants become tensor product multiplicities. 
See \cite{hernandez-leclerc, kimura-qin, nakajima}. 

\pagebreak[3]

The web basis does not always satisfy the conditions of
Conjecture~\ref{conj:strong-positivity}: 

\begin{proposition}
\label{prop:negative-structure-const}
For some choices of signature, 
some structure constants of the web basis are negative. 
\end{proposition}

An example justifying Proposition~\ref{prop:negative-structure-const}
is given in Section~\ref{sec:negative-structure-const}. 

The \emph{Laurent positivity conjecture}~\cite{ca1} predicts
that the Laurent polynomial expressing any cluster variable in terms
of any given seed has positive coefficients.
This conjecture has been proved in many 
special cases, see in particular \cite{caldero-reineke,
  musiker-schiffler-williams, nakajima}.



It is well known, and easy to see, that
Conjecture~\ref{conj:strong-positivity}  
is \emph{stronger} than Laurent positivity.
To deduce the latter from the former, multiply a cluster
variable by the denominator of its Laurent expansion. 
The result is a linear combination of cluster monomials.
By Conjecture~\ref{conj:strong-positivity},
they all belong to the basis~$\Ccal$; 
moreover the coefficients in this linear combination 
must be positive, and we are~done.

\pagebreak[3]

\section{Arborization}
\label{sec:arborization}

By definition, a web invariant can be represented by a single planar
tensor diagram (a~web). 
We conjecture that the cluster variables in~$R_\sigma(V)$
are distinguished from other web invariants by the property of
possessing a particular kind of alternative presentation;
informally speaking, they can be defined by tensor diagrams which
are \emph{trees}. 
Let us explain.

The \emph{unclasping} of a tensor diagram~$D$ is the graph obtained
from~$D$ by replacing each boundary vertex~$p$, say of
degree~$k$, by $k$ distinct vertices serving as endpoints of the edges
formerly incident to~$p$. 
By a harmless abuse of terminology,
we shall call a tensor diagram~$D$
whose unclasping has no cycles a \emph{forest diagram}; 
if moreover the unclasping is connected, we call $D$ a \emph{tree diagram}. 
We emphasize that such a diagram~$D$ does \emph{not} have to be planar. 

\begin{conjecture}
\label{conj:cluster-variables-are-trees}
A web invariant $z$ is a cluster or coefficient variable
if and only if $z$ is indecomposable and 
$z=[D]$ for some tree diagram~$D$. 

\end{conjecture}

To illustrate, consider the webs for non-Pl\"ucker cluster variables
in $R_{0,8}(V)$ shown in Figure~\ref{fig:webs50}.
The first one is a tree. 
The second one has a $6$-cycle but it goes through a boundary vertex.
The third web does have an internal cycle---but  
the alternative presentation of this invariant shown in 
Figure~\ref{fig:G38-nonplanar-tree} unclasps to a (non-planar) tree. 

\begin{figure}[ht]
\begin{center}
\setlength{\unitlength}{1.25pt}
\begin{picture}(70,75)(0,0)
\thicklines
\put(35,5){\circle*{4}}
\put(35,5){\line(0,1){18}}

\put(0,40){\circle*{4}}
\put(0,40){\line(3,-2){13}}

\put(10,60){\circle*{4}}
\put(10,60){\line(3,-2){13}}
\put(10,60){\line(3,1){13}}

\put(10,15){\circle*{4}}
\put(10,15){\line(1,3){4.2}}

\put(25,40){\circle*{4}}
\put(25,40){\line(2,-3){8.7}}
\put(25,40){\line(-1,-1){8.5}}
\put(25,40){\line(0,1){8}}

\put(35,55){\circle*{4}}
\put(35,55){\line(1,0){8}}
\put(35,55){\line(-2,-1){8.2}}
\put(35,55){\line(-1,1){8.4}}

\put(35,75){\circle*{4}}
\put(35,75){\line(-1,-1){8.4}}

\put(45,40){\circle*{4}}
\put(45,40){\line(-2,-3){8.7}}
\put(45,40){\line(1,-1){8.5}}
\put(45,40){\line(0,1){13}}

\put(60,15){\circle*{4}}
\put(60,15){\line(-1,3){4.2}}

\put(60,60){\circle*{4}}
\put(60,60){\line(-3,-1){13}}

\put(70,40){\circle*{4}}
\put(70,40){\line(-3,-2){13}}

\put(35,25){\circle{4}}
\put(15,30){\circle{4}}
\put(55,30){\circle{4}}
\put(45,55){\circle{4}}
\put(25,50){\circle{4}}
\put(25,65){\circle{4}}
\end{picture}
\qquad\qquad
\begin{picture}(70,75)(0,0)
\thicklines
\put(35,5){\circle*{4}}
\put(35,5){\line(0,1){18}}

\put(0,40){\circle*{4}}
\put(0,40){\line(3,-2){13}}

\put(10,60){\circle*{4}}
\put(10,60){\line(6,-1){32.8}}
\put(10,60){\line(4,1){12.7}}

\put(10,15){\circle*{4}}
\put(10,15){\line(1,3){4.2}}

\put(25,40){\circle*{4}}
\put(25,40){\line(2,-3){8.7}}
\put(25,40){\line(-1,-1){8.5}}
\put(25,40){\line(0,1){21.5}}


\put(35,75){\circle*{4}}
\put(35,75){\line(-5,-6){8.4}}

\put(45,40){\circle*{4}}
\put(45,40){\line(-2,-3){8.7}}
\put(45,40){\line(1,-1){8.5}}
\put(45,40){\line(0,1){12}}

\put(60,15){\circle*{4}}
\put(60,15){\line(-1,3){4.2}}

\put(60,60){\circle*{4}}
\put(60,60){\line(-5,-2){13}}

\put(70,40){\circle*{4}}
\put(70,40){\line(-3,-2){13}}

\put(35,25){\circle{4}}
\put(15,30){\circle{4}}
\put(55,30){\circle{4}}
\put(45,54.2){\circle{4}}
\put(25,63.7){\circle{4}}

\end{picture}
\end{center}
\vspace{-.1in}
\caption{Two representations of a non-Pl\"ucker cluster variable in $R_{0,8}(V)$.}
\label{fig:G38-nonplanar-tree}
\end{figure}
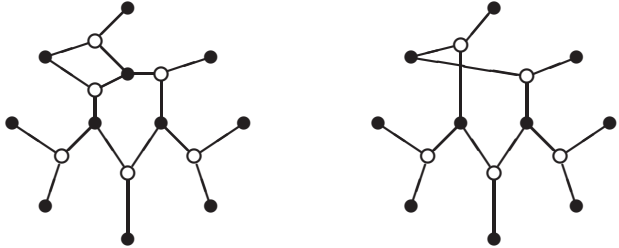

\pagebreak[3]

A much more complicated example is shown in Figure~\ref{fig:webs71}. 

\begin{figure}[h]
    \begin{center}
    \input{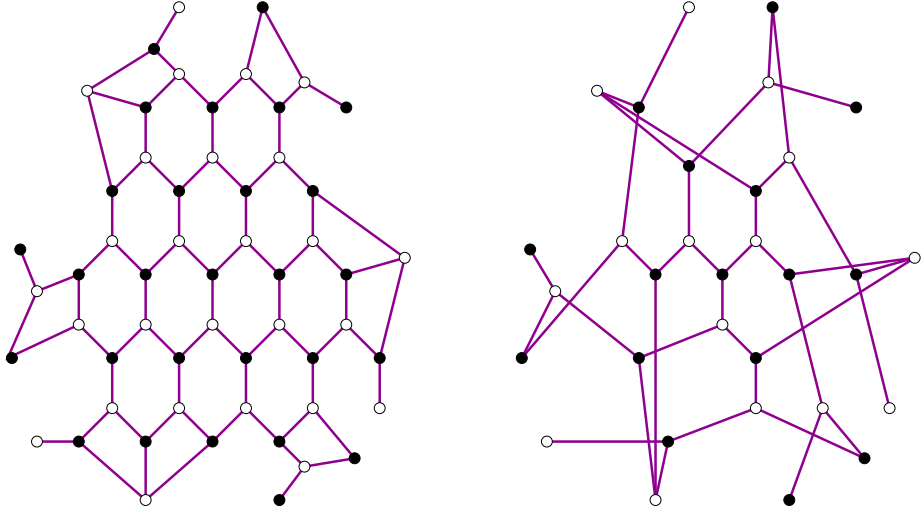}
    \end{center}
    \caption{A cluster variable represented by a web and by a tree diagram.}
    \label{fig:webs71}
\end{figure}

We note that Corollary~\ref{cor:planar-tree} establishes a very
special case of Conjecture~\ref{conj:cluster-variables-are-trees},
namely the case when the same tensor diagram is both a web and a
tree. 


An example given in Figure \ref{fig:websnew1} shows that the condition of indecomposability
in Conjecture~\ref{conj:cluster-variables-are-trees} cannot be
dropped: 
a decomposable web invariant representable by a tree diagram does not 
have to be a cluster monomial, let alone a cluster variable. 


\begin{figure}[ht]
    \begin{center}
\scalebox{0.75}{ \input{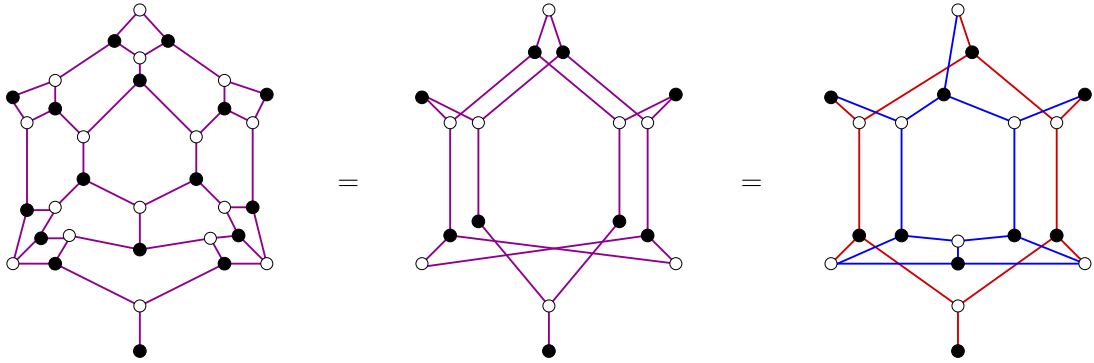}}
    \end{center}
\caption{Three representations of the same invariant:
(1) as a web invariant; (2) by a tree diagram; (3) as a product of two web invariants.
It can be shown that this invariant is not a cluster monomial.}
    \label{fig:websnew1}
\end{figure}

We next describe an algorithm that 
takes a web invariant~$z$ as input 
and conjecturally constructs a forest diagram defining~$z$,
or else concludes that none exists. 
This will require some preparation. 

\pagebreak[3]

\begin{definition}
Let $D$ be a tensor diagram, $s_1$ and~$s_2$ 
two of its internal vertices, and $e_1$ and $e_2$ two edges 
incident to $s_1$ and~$s_2$, respectively. 
We call vertices $s_1$ and~$s_2$ \emph{siblings} of each other 
(more precisely, ``siblings away from $e_1$ and~$e_2$'')
if the following happens. 
For $i\in\{1,2\}$, let $B_i$ denote the subgraph of~$D$ whose edge set
consists of those edges which can be reached from~$s_i$ without going
along~$e_i$ or connecting through a boundary vertex. 
(In particular, the edge $e_i$ is not in~$B_i$.)
We then want $B_1$ and~$B_2$ to be isomorphic binary trees
having the same multisets of leaves on the boundary of the disk. 
Thus, $s_1$ and~$s_2$ are siblings if they are
obtained from the same multiset of boundary vertices by 
the same sequence of taking pairwise joins. 
\end{definition}

\pagebreak[3]

\begin{definition}
\label{def:arborizing-step}
Suppose a tensor diagram~$D$ contains a fragment which is: 
\begin{itemize}
 \item a quadrilateral with one vertex on the boundary, or 
 \item a four-edge path whose endpoints are siblings of each other, 
looking away from the edges of the path.
\end{itemize}
An \emph{arborizing step} is the transformation of such a diagram~$D$ shown
in Figure~\ref{fig:webs27}. 
\end{definition}

\begin{figure}[ht]
    \begin{center}
\vspace{-.2in}
\input{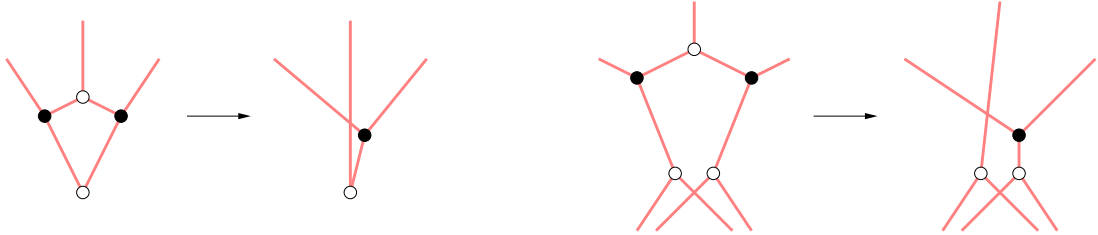}
\vspace{-.1in}
    \end{center}
    \caption{Arborizing steps. 
At left, the bottom vertex is on the boundary. 
At right, the boundary is further down (not shown).}
    \label{fig:webs27}
\end{figure}

\begin{lemma}
\label{lem:arborizing-step}
An arborizing step does not change the value of the
invariant defined by a tensor diagram.
\end{lemma}

The \emph{arborization algorithm} takes a tensor diagram (not
necessarily planar) as input, and applies arborizing steps
until unable to do so. 
See Figure~\ref{fig:webs26}. 

\begin{figure}[ht]
    \begin{center}
\scalebox{0.9}{
\input{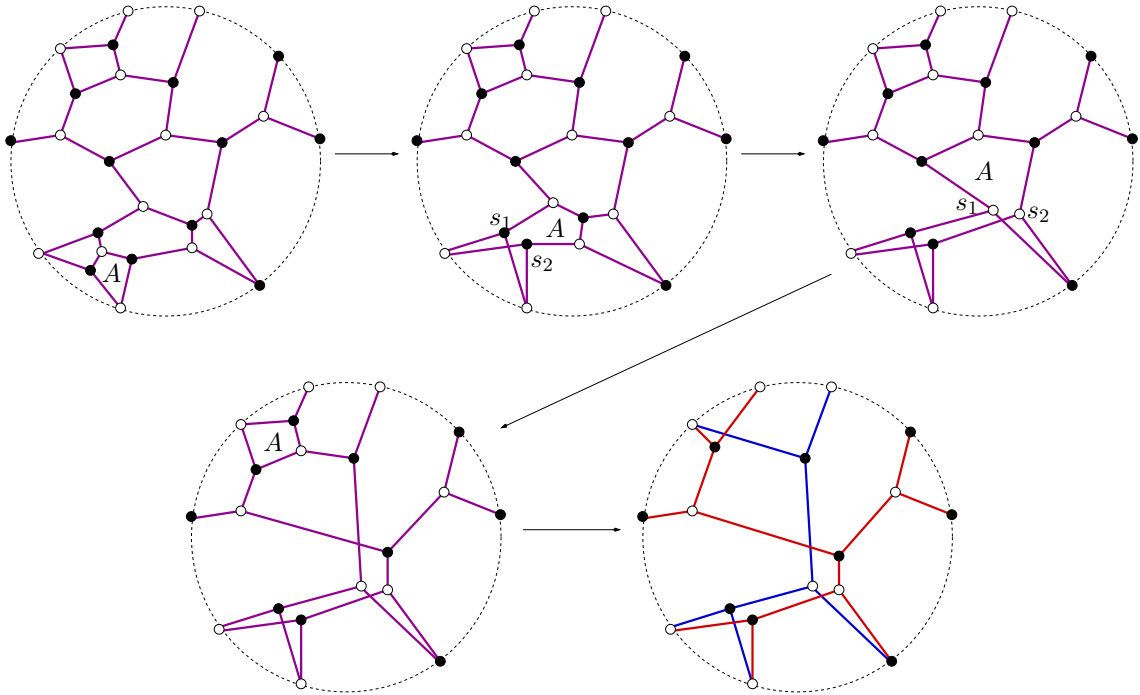} 
}
    \end{center} 
\caption{Arborization algorithm.
The output is a forest diagram defining a
cluster monomial, a product of two compatible special invariants.   
We indicate the location of each arborizing step 
and the two sibling vertices, if applicable.
}
    \label{fig:webs26}
\end{figure}


\begin{theorem}
\label{th:arborization-confluent}
The arborization algorithm is confluent.
That is, its output does not depend on the
choice of an arborizing step made at each stage. 
\end{theorem}

We are tempted to suggest the following enhancement of 
Conjecture~\ref{conj:cluster-variables-are-trees}. 

\begin{conjecture} 
\label{conj:lift2}
An indecomposable web invariant $z$ is a cluster or coefficient variable 
if and only if the arborization algorithm applied to the web
defining~$z$ outputs a tree diagram. 
\end{conjecture}


If a web invariant arborizes to a tree diagram, then we expect it 
to be a cluster variable 
(cf.\ Conjecture~\ref{conj:lift2}),
so its powers---which are cluster monomials---should be web invariants
as well (cf.\ Conjecture~\ref{conj:cluster-monomials-are-in-web-basis}). 
Theorem~\ref{th:powers-arborizable} below confirms this expectation,
thereby providing indirect support for the aforementioned
conjectures. 

\begin{theorem}
\label{th:powers-arborizable}
Let $z$ be a web invariant defined by a web which arborizes to a tree
diagram via the arborization algorithm. 
Then any power of $z$ is a web invariant. 
\end{theorem}

We actually do a bit more:
under the assumptions of Theorem~\ref{th:powers-arborizable},
we explicitly describe the web that
defines the power~$z^k$ of an arborizable web invariant~$z$.

\begin{definition}[\emph{Thickening of a web}]
\label{def:thickening}

Let $k$ be a positive integer, and $W$ a web.
The \emph{$k$-thickening} of~$W$ 
is obtained as follows: 
\begin{itemize}
\item
replace each internal vertex of~$W$ by a ``honeycomb'' fragment $H_k$
of the kind shown in Figure~\ref{fig:webs29} (boundary vertices stay put);
\item
replace each edge of~$W$ by a $k$-tuple of edges connecting the
corresponding honeycombs and/or boundary vertices. 
\end{itemize}
An example is shown in Figure~\ref{fig:webs30}.
\end{definition}


\begin{figure}[ht]
    \begin{center}
\vspace{-.1in}
    \scalebox{0.9}{\input{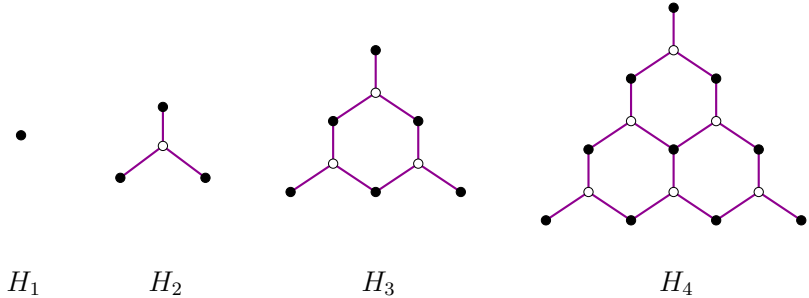}}
    \end{center} 
    \caption{Honeycomb web fragments~$H_k$
used to replace a black vertex. 
When replacing a white vertex, reverse the colors.}
    \label{fig:webs29}
\end{figure}

\begin{figure}[ht]
    \begin{center}
\input{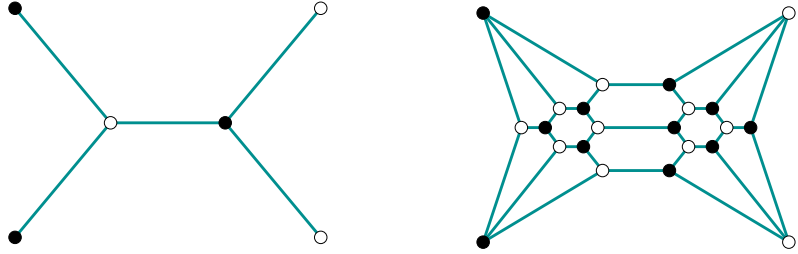}
    \end{center} 
    \caption{A quadripod web and its 
$3$-thickening.}
    \label{fig:webs30}
\end{figure}

\begin{theorem}
\label{th:thickening}
Let $z$ be a web invariant defined by a web~$W$ that arborizes to a tree
diagram via the arborization algorithm. 
Then each power $z^k$ is a web invariant defined by the $k$-thickening
of~$W$. 
\end{theorem}

\pagebreak[3]

\section{Web gallery}
\label{sec:zoo}

In this section, we present examples of web invariants 
possessing various notable properties. 

\subsection{Non-arborizable indecomposable webs}

By Conjecture~\ref{conj:cluster-vars-are-in-web-basis}, 
the set of indecomposable web invariants includes all cluster
variables. 
Which other web invariants does it include? 
Conjecture~\ref{conj:lift2} suggests an approach to constructing such
invariants:  if a web does not
arborize to a tree diagram, then the corresponding invariant is not a
cluster variable.
A couple of examples are shown in Figure~\ref{fig:webs34}. 


\begin{figure}[ht]
    \begin{center}
    \input{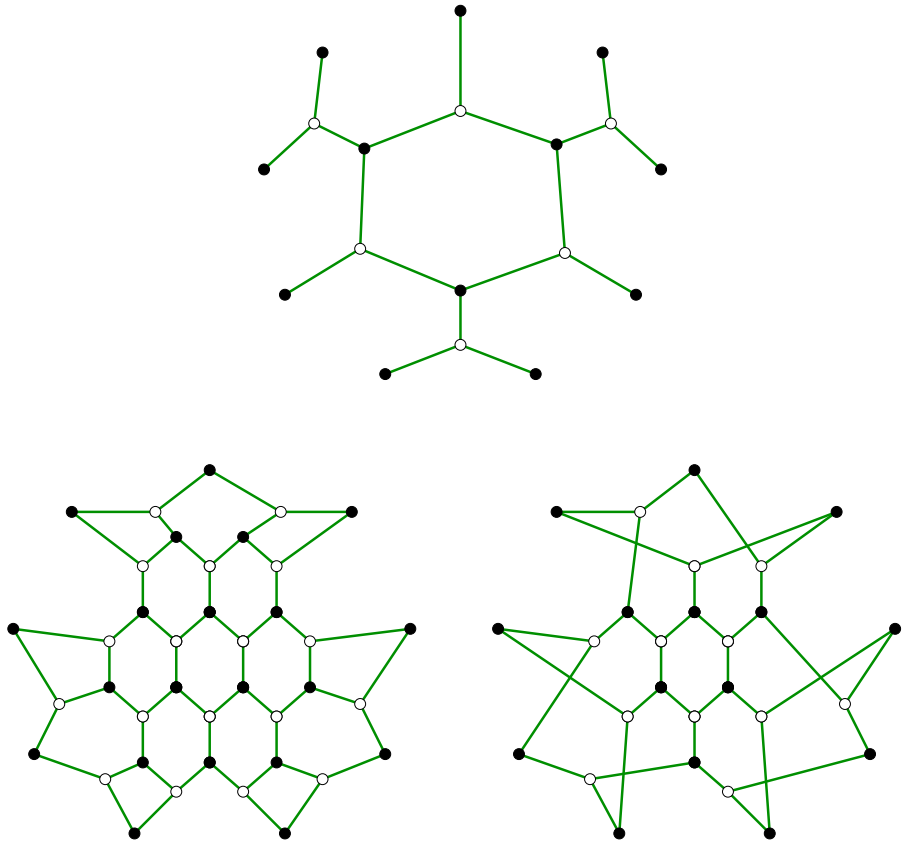} 
    \end{center} 
    \caption{Indecomposable web invariants in $R_{0,9}(V)$ which
do not arborize to trees. 
(We conjecture that there are infinitely many such invariants in~$R_{0,9}(V)$.)
The web at the top is its own arborization.
At the bottom, we show a web together with its arborized form.
}
    \label{fig:webs34}
\end{figure}

\subsection{Non-arborizable web whose powers are obtained by
  thickening} 
\label{sec:thickening-non-cluster}

The converse to Theorem~\ref{th:thickening} (or
Theorem~\ref{th:powers-arborizable}) is false: 
there are lots of web invariants which cannot be represented by a tree
diagram, yet all their powers are web invariants themselves;
moreover they can be obtained by the thickening procedure. 
The simplest example is shown in Figure~\ref{fig:webs32}. 

\begin{figure}[ht]
    \begin{center}
\input{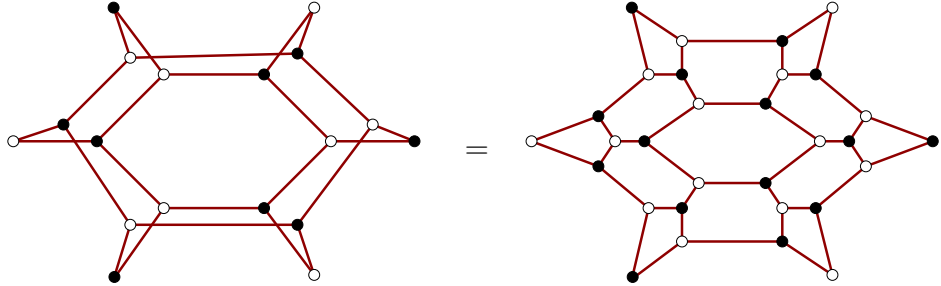} 
    \end{center} 
    \caption{Any power of a single-cycle web is equal to its
thickening. Add~a~dummy boundary vertex to make the signature non-alternating.}
    \label{fig:webs32}
\end{figure}

The unclasping of the second web in Figure~\ref{fig:webs32} yields the
minimal counterexample of M.~Khovanov and G.~Kuperberg
\cite[Theorem~4]{khovanov-kuperberg}. 

\subsection{Imaginary elements}
\label{sec:imaginary-elts}

The square of the web invariant does not have to be  a web invariant, 
see Figure~\ref{fig:webs33}. 
This phenomenon parallels the existence of ``imaginary'' elements in
dual canonical bases, first discovered by B.~Leclerc~\cite{leclerc}.



\begin{figure}[ht]
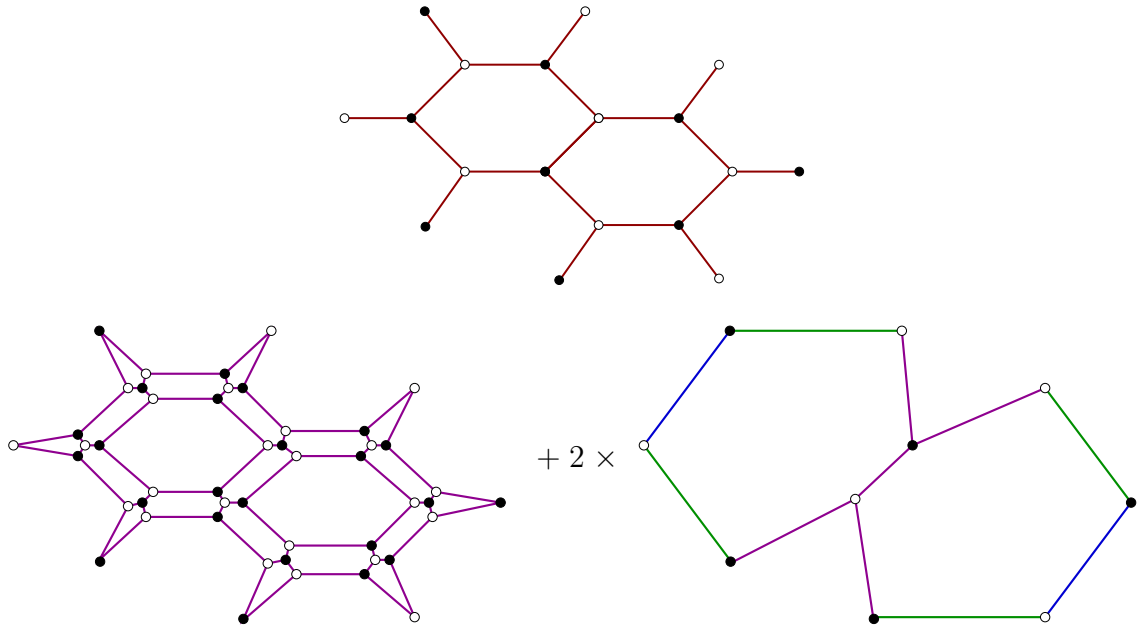

    \begin{center}
    \scalebox{0.8}{\input{webs33.pstex_t}}
\\[.2in]
\input{webs70.pstex_t}
    \end{center}
    \caption{The square of the web invariant shown above is a linear
      combination of web invariants shown below.
}
\label{fig:webs33}
\end{figure}

\pagebreak[3]

\subsection{Fake exchange relations}
\label{sec:fake-exchange}

The identity presented in Figure~\ref{fig:webs68} 
shows that in Conjecture~\ref{conj:3-term-non-cluster}, 
one cannot drop the requirement for $z'M_1$, and $z'M_2$ 
to be web invariants.

\begin{figure}[ht]
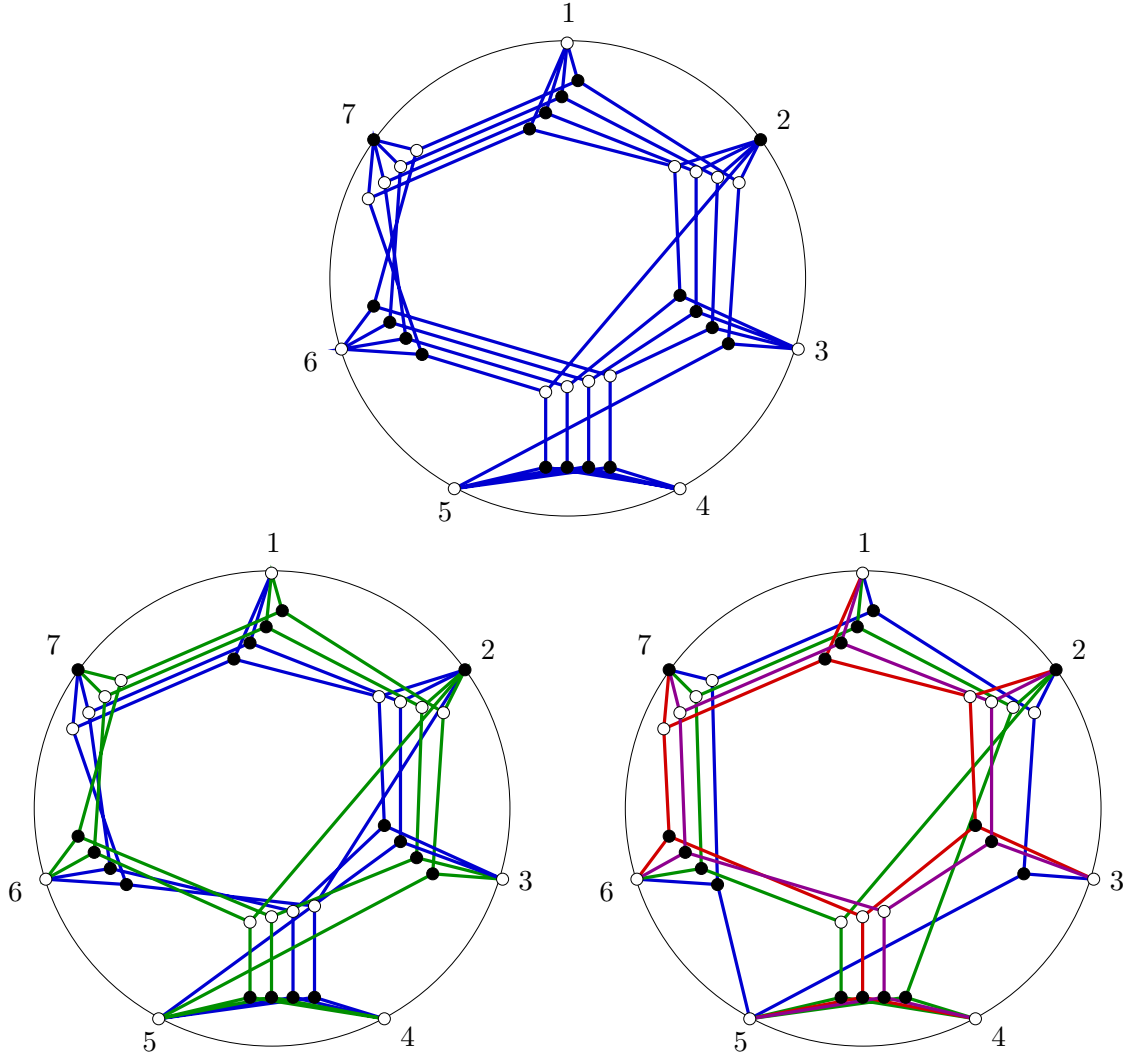

    \begin{center}
    \input{webs68a.pstex_t}
    \input{webs68b.pstex_t}
    \end{center}
    \caption{Let $z$ be the cluster variable shown at the top.  
Let \hbox{$z'=J_2^5$.} These cluster variables 
satisfy the 3-term relation 
$zz'=M_1+M_2$
where $M_1$ and $M_2$ are the web invariants shown at the bottom
(in arborized form).  
None of $zM_1$, $zM_2$, $z'M_1$, $z'M_2$ are web invariants.}
    \label{fig:webs68}
\end{figure}

\subsection{Negative structure constants}
\label{sec:negative-structure-const}

Figure~\ref{fig:webs69} presents an instance where a particular structure 
constant for the web basis is negative.
This example shows that even the product of two \emph{cluster variables} 
may expand in the web basis with coefficients some of which are negative. 

\begin{figure}[ht]
    \begin{center}
    \input{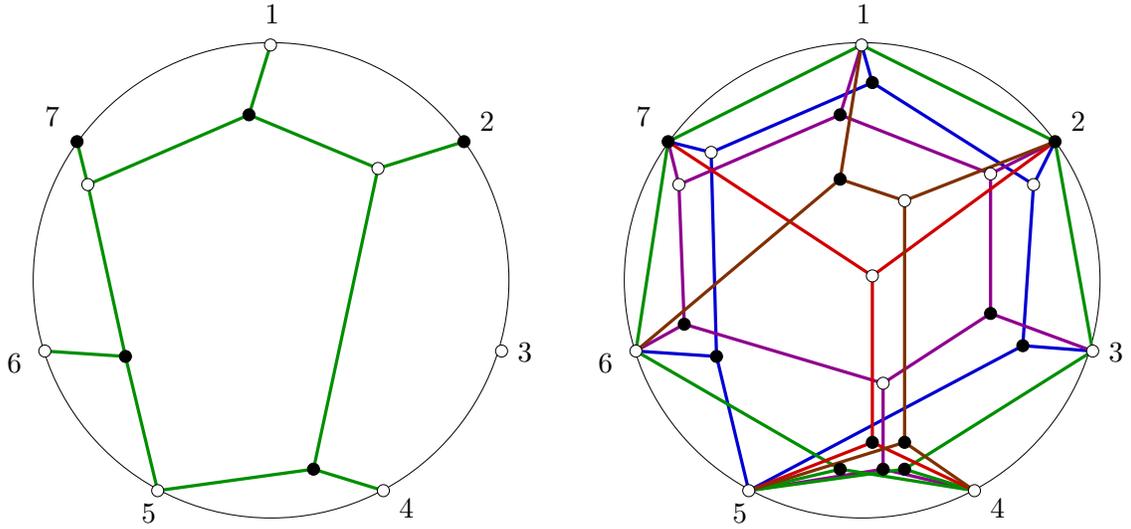}
    \end{center}
    \caption{Let $z$ be the cluster variable from Figure~\ref{fig:webs68},
and~let~$u$~be the cluster variable shown on the left.
The expansion of $zu$ in the web basis 
contains the web invariant shown on the right with coefficient~$-1$.}
    \label{fig:webs69}
\end{figure}

\newpage

\usection{Proofs}

\section{Properties of special invariants}
\label{sec:special-proofs}

\begin{proof}[Proof of Proposition~\ref{prop:specweb}] 
It is routine to examine the cases and check that 
any tensor diagram representing 
a special invariant ``planarizes'' 
(via repeated application of skein relations) into a single
non-elliptic web. 
The key relation used in these verifications is shown in 
Figure~\ref{fig:basic-step}. 
An example is given in Figure~\ref{fig:special-planarize}. 
\end{proof}

\vspace{-.05in}

\begin{figure}[ht]
\begin{center}
\scalebox{0.7}{
\input{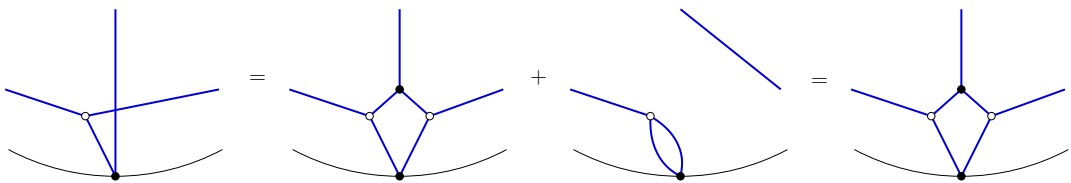}
}
\end{center} 
    \caption{Basic planarizing step. Cf.\ Figure~\ref{fig:webs27} on
    the left.}
    \label{fig:basic-step}
\end{figure}

\vspace{-.05in}

\begin{figure}[ht]
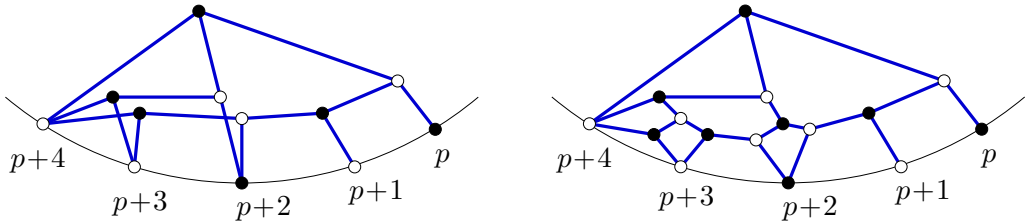

\begin{center}
\input{webs53a.pstex_t}
\qquad
\input{webs53b.pstex_t}
\end{center} 
    \caption{The special invariant $J^{p,p+2,p+4}$ is a web invariant.}
    \label{fig:special-planarize}
\end{figure}

\pagebreak[3]

The proof of Proposition~\ref{prop:special-are-irreducible}
will require some preparations. 

\begin{lemma} 
\label{lem:antis}
 Assume that an invariant $X\in R_\sigma(V)$
is antisymmetric and linear with respect to its arguments
$v_p$ and $v_{p+1}$ which have the same variance 
(i.e., are both contravariant or both covariant). 
 Then $X$ can be expressed as a linear combination of web invariants 
defined by webs which have a fork between vertices $p$ and~$p+1$. 
\end{lemma}

\begin{proof}
We can write $X$ as a linear
combination  of web invariants $X=\sum_i c_i Y_i$ 
where the web representing each $Y_i$ has a single edge incident to~$p$
(resp., to~$p+1$). 
Let us attach a crossing to this web at vertices $p$
and~$p+1$ to get an invariant~$\overline Y_i$.
(In other words, we redirect the edge incident to $p$ so that it
connects to~$p+1$, and vice versa.) 
Since $\overline X=\sum_i c_i \overline Y_i$ is obtained
from $X$ by interchanging $v_p$ and~$v_{p+1}$,
and since $X$ is antisymmetric in these arguments, 
it follows that $\overline X=-X$.
On the other hand, applying the skein relation at the crossing 
gives $\overline Y_i = Y_i + Z_i$, where $Z_i$ has
a fork between $p$ and~$p+1$. 
Taking linear combinations, we get $-X = X + \sum_i
c_i Z_i$, implying 
$X = -\frac{1}{2} \sum_i c_i Z_i$,
an expansion of the desired kind. 
\end{proof}

\begin{lemma} 
\label{lem:plantreeirr}
 A web invariant defined by a planar tree is irreducible.
\end{lemma}

Note that by Corollary~\ref{cor:planar-tree}, 
such an invariant is a cluster or
coefficient variable in~$R_{\sigma}(V)$. 
By \cite[Theorem~1.3]{gls-factorial},
every cluster variable in \emph{any} cluster algebra is
irreducible. 
At this point, we cannot of course rely on these statements, 
as we are still in the process of
proving that $R_{\sigma}(V)$ is a cluster algebra. 

\begin{proof} 

We proceed by induction on the size of the tree. 
In the base cases of Weyl generators,
irreducibility is well known (and easy to prove). 

Suppose a web invariant $X$ defined by a planar tree~$T$ 
has a nontrivial factorization $X=X_1X_2$. 
It is easy to see that $T$ must have a fork, 
say between vertices $p$ and $q$ (of the same color)
associated with (co)vectors $v_p$ and~$v_q$, respectively. 
Without loss of generality, we may assume that $q=(p+1)\bmod(a+b)$. 

As $X$ is multi-homogeneous, so must be $X_1$ and~$X_2$.  
Moreover $X$ is multilinear, implying the dichotomy: 
either one of the factors $X_i$ depends (linearly) on both 
$v_p$ and $v_q$ while the other depends on neither, 
or else one of the factors depends on $v_p$ but not $v_q$ while another 
depends on $v_q$ but not $v_p$. 

The latter option is ruled out by the fact that $X$ is antisymmetric 
in $v_p$ and $v_q$ (because of the fork), 
so it must vanish if we substitute $v_p = v_q$. 
On the other hand, neither factor $X_i$ 
vanishes under this substitution, and $R_\sigma$ is a domain. 

So $v_p$ and $v_q$ appear in the same factor, say~$X_1$. 
Then $X_2$ does not depend on $v_p$ and~$v_q$.  
Hence $X_1$ is antisymmetric in $v_p$ and~$v_q$.
Applying Lemma~\ref{lem:antis}, 
we express $X_1$ as a linear combination of web invariants 
whose webs have a fork between $p$ and $q=p+1$. 
Thus, $v_p$ and $v_q$ enter the identity $X=X_1X_2$ 
exclusively through their cross product $v_p\times v_q$.
This yields a nontrivial factorization of an invariant 
defined by a smaller planar tree,  
contradicting the induction assumption.
\end{proof}

\begin{proof}[Proof of Proposition~\ref{prop:special-are-irreducible}]

Let $X$ be a nonzero special invariant that does not have one of the
forms listed in Proposition~\ref{prop:special-reducible}.
We need to show that $X$ is irreducible. 
We describe the general idea of the proof, omitting some details. 
The proof is by induction on the number of internal vertices in the
tree diagram defining~$X$. 

Let $A\subset\{1,\dots,a+b\}$ denote the set of lower and upper
indices appearing in the 
original notation for~$X$ (cf.\ Definition~\ref{def:special-inv}). 
Thus the cardinality of $A$ is $2$, $3$, or~$4$. 

If $X$ is a planar tree, the statement reduces to
Lemma~\ref{lem:plantreeirr}.
Otherwise the ``tail'' involved in building the proxy for some of the
vertices in~$A$ (cf.\ Figure~\ref{fig:caterpillars}) is long 
enough to involve another such vertex.
It is not hard to see that we can find two
consecutive vertices $p,q\in A$ (going clockwise) 
such that the boundary segment between them is covered 
by only one such tail. Furthermore, for $X$ to be indecomposable,
 this segment has to be sufficiently long, by definition.
Assume without loss of generality that vertex $p$ is white;
then $p+1$ is black, $p+2$ is white, etc. 
We then define invariants $X'$ and $X''$ as follows.
(Consult Figure~\ref{fig:webs72}.)
For~$X'$, attach an outside fork to $p+1$ and identify 
one of its endpoints with~$p$.
(This changes the color of the entry $p+1$ of the signature.) 
For $X''$, attach an outside fork to $p$ and identify 
one of its endpoints with~$p+1$.
(This changes the color of the $p$'th entry.) 
Thus the invariants 
$X'=X'(\dots,v_p^*,v_{p+1}^*,v_{p+2}^*,\dots)$ and
$X''=X''(\dots,v_p,v_{p+1},v_{p+2}^*,\dots)$ 
are obtained from $X(\dots,v_p^*,v_{p+1},v_{p+2}^*,\dots)$ via the
substitutions
\begin{align}
\label{eq:X'-sub}
X'(\dots,v_p^*,v_{p+1}^*,\dots)
&=X(\dots,v_p^*,v_p^*\times v_{p+1}^*,\dots),\\
\label{eq:X''-sub}
X''(\dots,v_p,v_{p+1},\dots)
&=X(\dots,v_p\times v_{p+1},v_{p+1},\dots).
\end{align}


 \begin{figure}[ht]
    \begin{center}
    \input{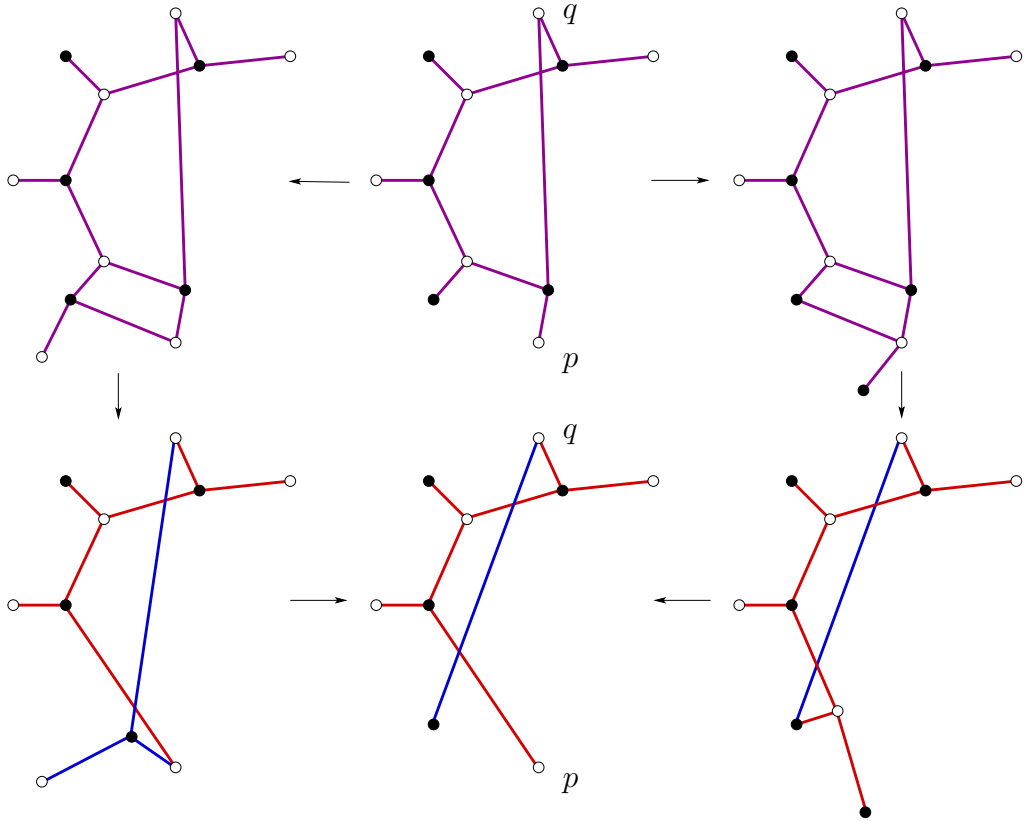}
    \end{center}
    \caption{Top row: invariants $X'$, $X$, $X''$. }
    \label{fig:webs72}
\end{figure}

Suppose $X$ has a nontrivial factorization $X=YZ$;
it then specializes into factorizations $X'=Y'Z'$ and~$X''=Y''Z''$. 
On the other hand, $X'$ and $X''$ factor as shown in
Figure~\ref{fig:webs72} (bottom row), one of the factors (shown in red)
being given by a planar tree 
and another (blue) factor being a special invariant of the same kind as~$X$. 
The red factor is irreducible by Lemma~\ref{lem:plantreeirr}
while the blue one is irreducible by the induction assumption: 
it is a special invariant of a kind not listed in
Proposition~\ref{prop:special-reducible}, and it has fewer
internal vertices. 
Since our rings of invariants are unique factorization domains,
we conclude that factorizations $X'=Y'Z'$ and~$X''=Y''Z''$
must coincide with the ones shown in Figure~\ref{fig:webs72}. 

Since $X$ is linear in~$v_{p+1}$, either $Y$ or~$Z$ (say~$Y$) 
does not depend on~$v_{p+1}$. Thus $Y$ is unaffected by the
substitution~\eqref{eq:X'-sub}, and $Y'=Y$ does not depend
on~$v_{p+1}^*$. Consequently, $Y$ must be the invariant 
shown in red in the lower-left corner 
of Figure~\ref{fig:webs72}. 
Similarly, one of the factors $Y$ and $Z$ does not depend on~$v_p$,
is unaffected by the substitution~\eqref{eq:X''-sub},
and appears in the factorization of $X''$ 
shown in Figure~\ref{fig:webs72} in the lower-right corner;
this must be the invariant shown in blue (and it must be~$Z$). 
In conclusion, the factors $Y$ and $Z$ must match 
the ones shown in the center of the bottom row. 
But the product of these two special invariants yields a sum of two
nonzero terms, one of them being~$X$. 
Thus $YZ\neq X$, a contradiction. 
\end{proof}

\begin{proof}[Proof of Proposition~\ref{prop:special-inv-factoring}]

In view of Proposition~\ref{prop:special-are-irreducible}, 
a nonzero special invariant can be factored into irreducible ones 
by repeated application of the rules
\eqref{enum:fact-step-1a}--\eqref{enum:fact-step-5}
in Proposition~\ref{prop:special-reducible}. 
An easy check shows that this process terminates. 
Uniqueness~of factorization 
follows from the fact that our ring is a UFD.
See Figure~\ref{fig:webs53} for an example. 
\end{proof}

\begin{figure}[ht]
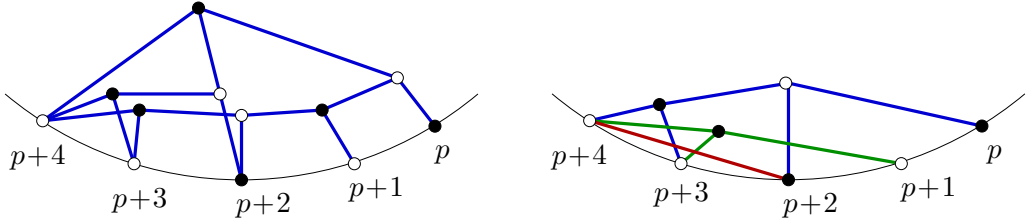

\begin{center}
\input{webs53a.pstex_t}
\qquad
\input{webs53c.pstex_t}
\end{center} 
    \caption{Factorization of 
$J^{p(p+2)(p+4)}$. 
The rules 
\eqref{enum:fact-step-1a}--\eqref{enum:fact-step-5}
can be applied in two different ways:
$J^{p(p+2)(p+4)} = J_{p+2}^{p+4} J_{p+3}^p = J_{p+2}^{p+4}
J_{p+3}^{p+1} J_p^{p+2}$
or
$J^{p(p+2)(p+4)} =
J_p^{p+2} J_{p+1}^{p+4} = J_p^{p+2} J_{p+3}^{p+1} J_{p+2}^{p+4}$,
yielding identical results.}
    \label{fig:webs53}
\end{figure}


\begin{proof}[Proof of Proposition~\ref{prop:coeff-special}]
First, we need to prove 
 that a non\-zero special invariant not on our list is not compatible with
 some other special invariant. This is checked on a case by case
 basis. For an invariant of the form~$J_p^q$, one finds an 
 incompatible special invariant $J_r^s$ such that the
 diagonals $pq$ and $rs$ cross each other. Similarly, if
 our invariant is $J_{pqr}$ or $J^{pqr}$, one can find
 an incompatible special invariant $J_s^t$ such that the
 diagonal $st$  crosses the triangle~$pqr$. 
(These are special cases of failure of \emph{weak separation}, 
cf.~\cite{leclerc-zelevinsky}.)  
Special invariants $J_{pq}^{rs}$ are handled analogously. 

Another claim to check is that a nonzero invariant of the form
$J_p^{p \pm 1}$ is compatible with any web invariant. 
Coefficient invariants come in two flavors: short and long, 
see Figure \ref{fig:webs54}. The
short ones are obviously compatible with any web invariant. 
The relevant calculation for a long invariant 
is shown in Figure~\ref{fig:webs55}. Applying it
around every vertex establishes the claim. 
\end{proof}

\begin{figure}[ht]
\scalebox{0.9}{\input{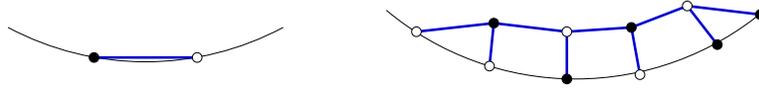}}
    \caption{Two kinds of coefficient invariants: short and long.}
    \label{fig:webs54}
\end{figure}

\begin{figure}[ht]
\scalebox{0.7}{\input{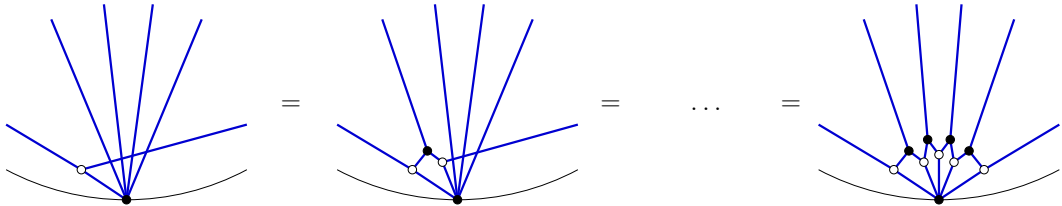}}
    \caption{Compatibility of 
    coefficient invariants with web invariants is verified by 
iterating the basic relation shown in Figure~\ref{fig:basic-step}.} 
    \label{fig:webs55}
\end{figure}


\begin{proof}[Proof of Proposition~\ref{prop:3-term-skein-special}
]
Each identity 
can be obtained by repeated 
application of skein relations. 
A~key role is played by the relation shown in
Figure~\ref{fig:basic-step}. An example of a computation verifying the 
first relation is given in Figure \ref{fig:webs67aa}.
\begin{figure}[ht]
\scalebox{0.8}{\input{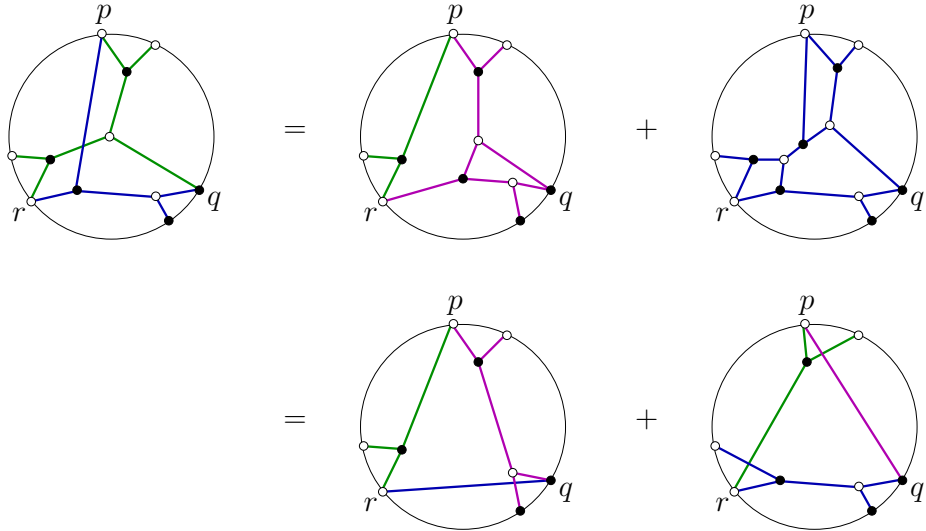}}
    \caption{The verification of the relation $J_{pqr} J^{pqr} = J_r^p J_q^r J_p^q + J_r^q J_p^r J_q^p$.} 
    \label{fig:webs67aa}
\end{figure}
\end{proof}

\begin{proof}[Proof of Proposition~\ref{pr:thin-triangles}]

Without loss of generality, we can assume that side $qr$ is exposed,
so that $r=(q+1)\bmod(a+b)$. 
One then checks that $J^{pq,q+1}=J_{q}^p$ if $q$ is white,
and $J^{pq,q+1}=J_{q+1}^p J_{q}^{q+1}$ if $q$ is black. 
\end{proof}

\begin{proof}[Proof of Proposition~\ref{pr:thin-quads}]

If $p$ is white and $p+1$ is black, then
$J^{sp}_{p+1,p+2}=J_p^s$. 
Relation~\eqref{eq:3-term-3} then gives
\[
J_{p+2}^p J_{p+1}^s 
= J_{p+1}^p J_{p+2}^s + J_{p+1,p+2}^{sp}
= J_{p+1}^p J_{p+2}^s + J_p^s\,.
\]
If $p$ is black and $p+1$ is white, then
$J_{sp}^{p+1,p+2}=J_s^p$, 
and \eqref{eq:3-term-3} gives
\[
J_p^{p+2} J_s^{p+1}
= J_s^{p+2} J_p^{p+1} + J_{sp}^{p+1,p+2}
=J_s^{p+2} J_p^{p+1} + J_s^p\,. 
\qedhere
\]
\end{proof}

\section{Properties of special seeds}
\label{sec:properties-of-special-seeds}

\begin{proof}[Proof of Theorem~\ref{th:z(T)}]

Denote $N=a+b$. 
Since the theorem can be verified by direct calculation
for $N\le 6$, we assume that $N\ge 7$ from now on.

By Proposition~\ref{prop:coeff-special}, 
each side of the $N$-gon $P_\sigma$ produces 
one nonzero special invariant 
which by construction belongs to~$\zz(T)$;
those are exactly the coefficient invariants.

Define the \emph{length} of a diagonal $pq$ in~$P_\sigma$
as the number $\min(|p-q|,|N-p-q|)$;
this is the length of the shortest path from $p$ to~$q$ along the
perimeter of~$P_\sigma$. 

\begin{lemma} 
\label{lem:decomp}
Assume that $N\ge 7$. 
Let $P'$ denote the interior of $P_\sigma$
with all diagonals of length~$2$ in~$T$ removed. 
The set $P'$ uniquely decomposes into a disjoint union of 
``fundamental regions'' of types $A, B, C, D, E, F$ 
shown in Figure~\ref{fig:webs73}.
\end{lemma}

An example is shown in Figure~\ref{fig:webs74}.
We note that the lemma fails for $N=6$ 
if $T$ contains a diagonal of length~$3$. 
(The two regions of type~$E$ are not
disjoint as they share this diagonal.)

\begin{figure}[ht]
    \begin{center}
\scalebox{0.8}{
   \input{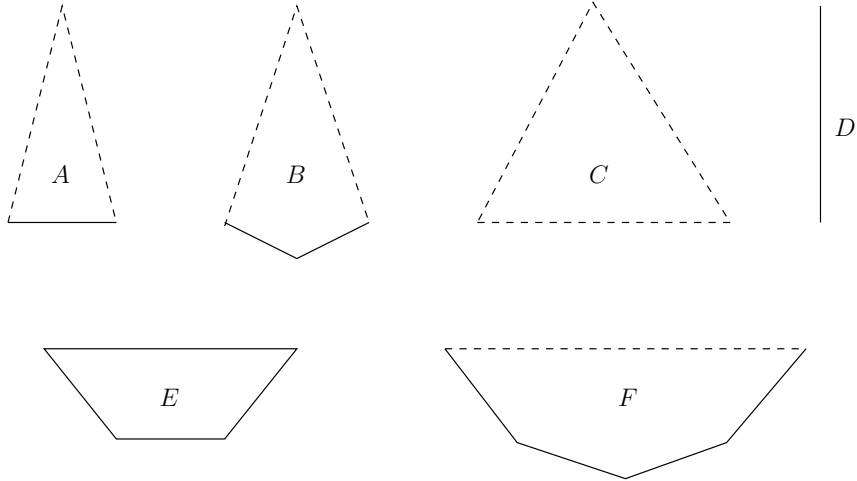}
}
   \end{center}
\caption{Fundamental regions. A~dashed line means that the region
does not include this diagonal. 
All dashed diagonals are of length at least~$3$. 
Type~$D$ includes diagonals of length at least~$4$. 
The short solid segments in $A$,
$B$, $E$ and $F$ are sides of the polygon~$P_\sigma$.}
   \label{fig:webs73}
\end{figure}

\begin{figure}[ht]
    \begin{center}
   \input{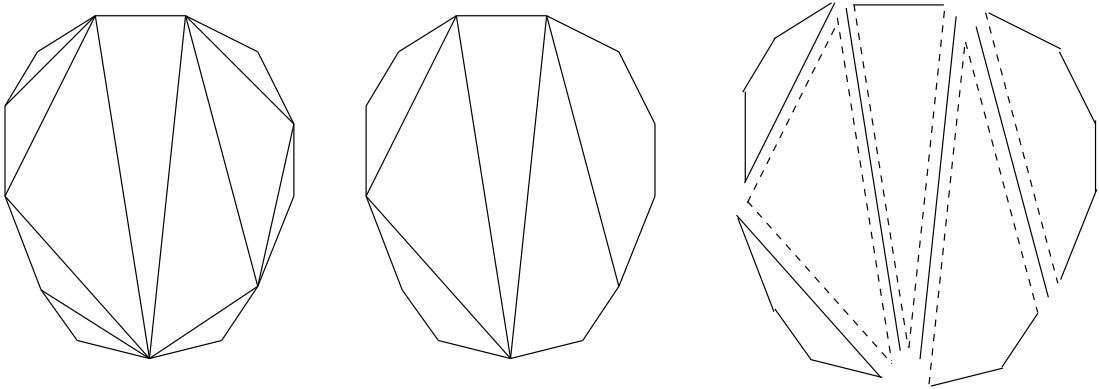}
   \end{center}
\vspace{-.2in}
   \caption{Decomposition of $P_\sigma$ into fundamental regions.}
   \label{fig:webs74}
\end{figure}

\begin{lemma}
Special invariants associated to each fundamental region
contribute the following number of non-coefficient
irreducible factors to the cluster~$\xx(T)$: 
\[
\begin{array}{|c|l|l|l|l|l|l|}
\hline
\text{\rm type of a region} & A & B & C & D & E & F\\
\hline
\text{\rm contribution to $\zz(T)$} & 0 & 2 & 1 & 2 & 3 & 3\\
\hline
\end{array}
\]
The contributions of different fundamental regions are disjoint. 
To clarify, if a region $R$ does not include a (dashed) diagonal~$pq$
lying on its boundary, then the contributions of $R$ 
exclude all irreducible factors appearing in $J_p^q$ and~$J_q^p$.
\end{lemma}

\begin{proof}
Case by case consideration depending on the signature.
One typical example involving a region of type~$E$
(with two possible colorings) is shown in Figure~\ref{fig:webs75}. 
\end{proof}

\begin{figure}[ht]
    \begin{center}
   \input{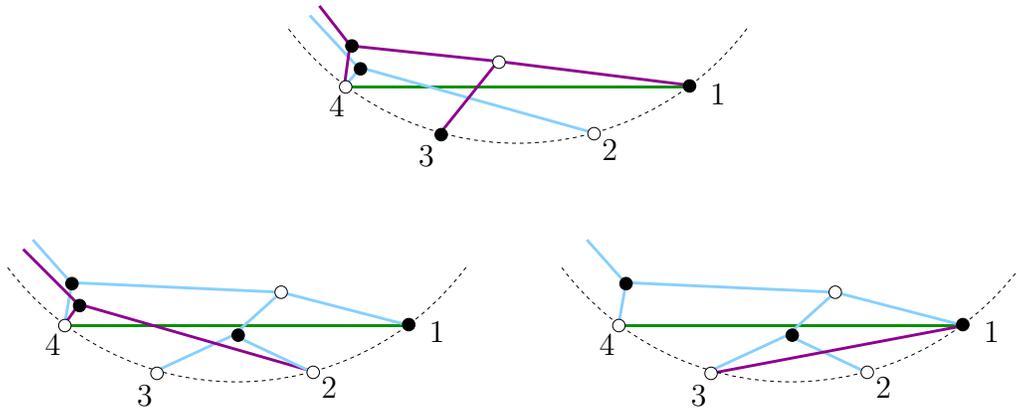}
   \end{center}
   \caption{
Two possible colorings for a region of type~$E$. 
In the first case, regardless of which of the two short diagonals 
($13$ or~$24$) is chosen, we
obtain the same three variables $J_1^4$, $J_1^3$ and $J_4^2$. In the
second case, depending on which short diagonal is used, we get either
$\{J_1^4, J_4^1, J_4^2\}$ or $\{J_1^4, J_4^1, J_1^3\}$. 
Each case yields exactly three non-coefficient variables. 
The fact that they cannot come from other
fundamental regions can be seen by performing a similar analysis 
for those regions.
}
   \label{fig:webs75}
\end{figure}

We continue with the proof of Theorem~\ref{th:z(T)}. 
It is a simple counting exercise to check that the total contribution
from all fundamental regions is exactly $2N-8$ non-coefficient
variables, as claimed.
For example, in Figure \ref{fig:webs74} we get 
\[
3+3+1+0+2+3+2+2+2=18=2\cdot 13-8.
\]

It remains to check that the special invariants in $\zz(T)$ are pairwise
compatible. 
For invariants coming from the same fundamental region,
this is done by direct case by case inspection. 
For invariants from different regions, 
the argument goes as follows. 
The ``tails'' of the special invariants 
(cf.\ Figure~\ref{fig:caterpillars})
do not create any obstructions to compatibility:  
just iterate the basic
planarizing steps of Figure~\ref{fig:basic-step}. 
The remaining pieces lie in different regions 
and thus do not intersect.
\end{proof}


\begin{proof}[Proof of Proposition~\ref{pr:x(T)=x(T')}]

The proof reduces to a direct verification comparing 
triangulations $T$ and~$T'$ which differ by a flip 
that switches the diagonals $(q,q+2)$ and $(q+1,q+3)$. 
A~typical example is shown in Figure~\ref{fig:webs15}. 
\end{proof} 

\begin{proof}[Proof of Proposition~\ref{prop:special-seeds-exchanges}]

The proof is a case by case verification, organized as follows.
For each region~$R$, 
the variables (i.e., irreducible factors) 
contributed by $R$ are exchanged using only
the variables defined within~$R$ 
and/or several adjacent regions.
Specifically:
\begin{itemize}
\item 
regions of type $A$ do not contribute any variables, so there is
nothing to check; 
\item 
for regions of type $B$ (resp.,~$C$,~$F$), 
also consider two (resp., three, one) adjacent region(s) of types
$D$ or~$E$; 
  \item 
for regions of type $D$ (resp.,~$E$),
also consider two (resp., one)
adjacent region(s) of types $A$, $B$, $C$ or~$F$, as well as the regions
 of types $D$ or $E$ which are adjacent to those.
\end{itemize}
In each of these cases, the proof consists of a local verification 
(sometimes tedious but always straightforward) 
ranging over a finite 
list of possible patterns. 
\end{proof}


\section{Building a quiver 
}
\label{sec:building-a-quiver}

In this section, we provide a blueprint for building the quiver~$Q(T)$
associated with an arbitrary triangulation~$T$ of the polygon~$P_\sigma$.
As mentioned earlier, a detailed description for 
a particular choice of~$T$ can be found in Section~\ref{sec:T-fan}.

The simplest part of the recipe
concerns the portions of $Q(T)$ coming from triangles $pqr$ in~$T$ 
which have no exposed sides. 
We draw the vertex~of~$Q(T)$ representing
the special invariant $J_{pqr}$ inside the triangle~$pqr$. 
The vertices representing cluster variables coming from 
$J_p^q$ and $J_q^p$ are placed on
the diagonal~$pq$ (and similarly for $qr$ and~$pr$),
with the former closer to~$q$, and the latter closer to~$p$. 
%
See Figure~\ref{fig:webs43}. 

\begin{figure}[ht]
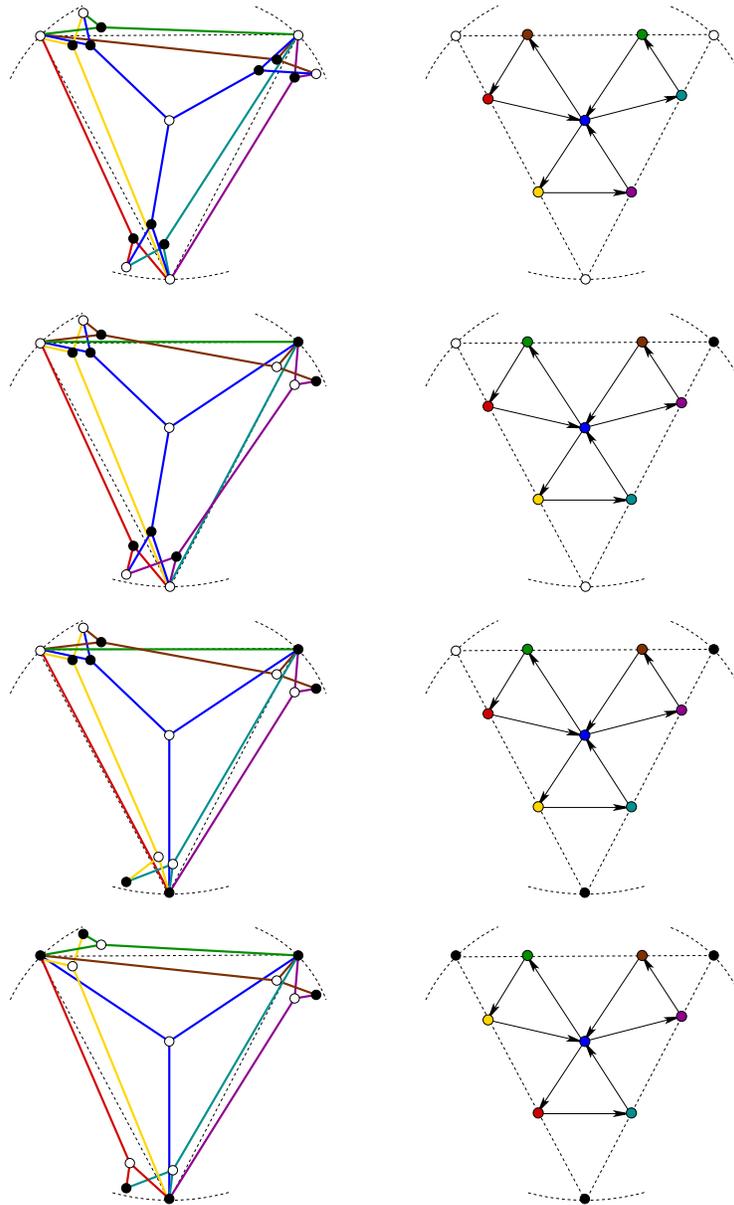


    \begin{center}
\input{webs4.pstex_t} 

\bigskip

    \input{webs41.pstex_t} 

\bigskip

    \input{webs42.pstex_t} 

\bigskip

    \input{webs43.pstex_t} 
    \end{center} 
    \caption{Special invariants around a triangle in a
      triangulation~$T$, 
and the corresponding portion of the quiver~$Q(T)$.
}
    \label{fig:webs43}
\end{figure}

The above recipe may require adjustments if some of the sides of the
triangle $pqr$ are too short, so that the corresponding special
invariants factor; cf., e.g., Figure~\ref{fig:webs15}. 



For triangles with one or two sides on the boundary of~$P_\sigma$,
the basic principles remain the same, but the recipe changes somewhat.  
Figures~\ref{fig:webs11} and~\ref{fig:webs12} treat ``generic'' cases of one
and two exposed sides, respectively. 

\begin{figure}[ht]
    \begin{center}
\input{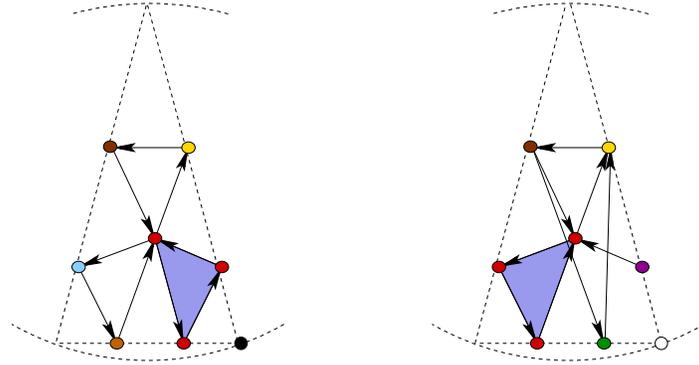} 
    \end{center} 
    \caption{Portion of the quiver $Q(T)$ 
in a triangle $pq(q+1)$ of~$T$.
There are two cases, depending on the color of~$q$. 
In each case, three vertices 
collapse into a single vertex of~$Q(T)$, represented by the solid triangle. 
All edges that used to go to/from
those three vertices are still going to/from this new vertex.}
    \label{fig:webs11}
\end{figure}

\begin{figure}[ht]
    \begin{center}
\input{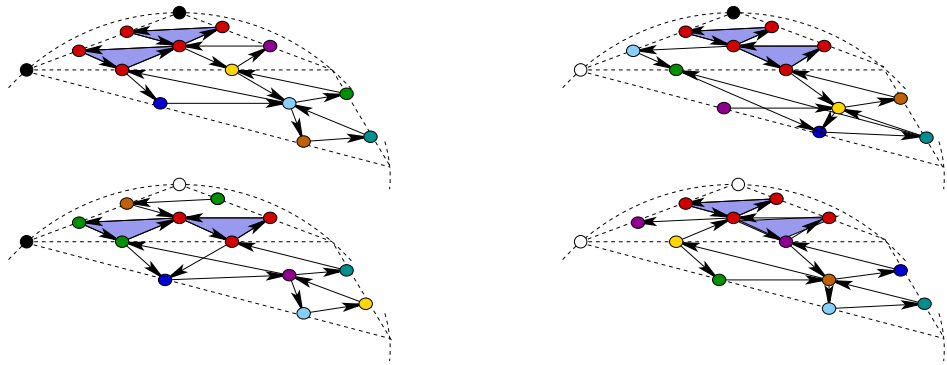}
    \end{center} 
    \caption{Portion of the quiver in a triangle with two exposed
      sides. The pattern only depends on
the coloring of two of the vertices.
}
    \label{fig:webs12}
\end{figure}


In a few exceptional cases, additional adjustments have to be made to
the construction of the quiver $Q(T)$ and/or to the rules of assigning
special invariants to its vertices. 
These adjustments are caused by nontrivial factoring of the special
invariants involved. 
One such case is illustrated in Figure~\ref{fig:webs14},
which is in turn a special case of the pattern shown in 
Figure~\ref{fig:webs12} on the upper right. 
\begin{figure}[ht]
    \begin{center}
\vspace{.2in}
\scalebox{0.9}{
    \input{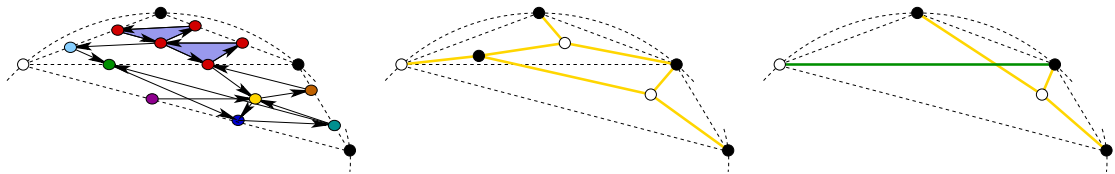} 
}
    \end{center} 
\caption{The special invariant 
corresponding to the yellow vertex~$Y$ of the quiver shown on the left
is given by the yellow web in the middle; it factors 
into a product of two special invariants shown on the right.
The green factor is already associated to the corresponding diagonal. 
We let the yellow factor be the cluster variable 
associated with~$Y$. 
}
    \label{fig:webs14}
\end{figure}
Another example is shown in Figure~\ref{fig:webs15}. 
\begin{figure}[ht]
    \begin{center}
    \input{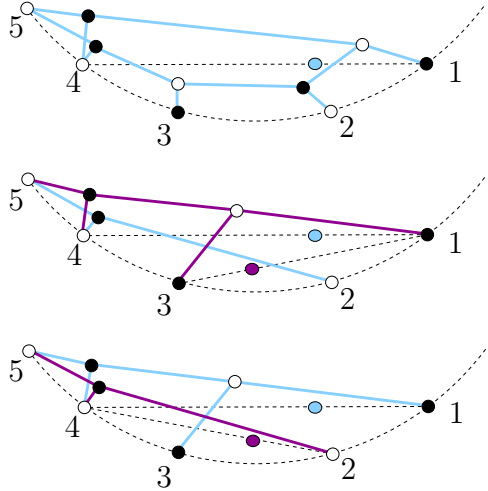} 
    \end{center} 
\caption{The special invariant~$J_4^1$ factors:
$J_4^1=J_4^2 J_1^3\,$. 
Which of the two factors should be 
associated to the diagonal~$14$ 
depends on whether triangulation~$T$ contains 
diagonal $13$ or diagonal~$24$. 
The clusters $\xx(T)$ and $\xx(T')$
associated to triangulations $T$ and~$T'$ which differ by a 
flip replacing $13$ by~$24$ are the same.
Cf.\ Proposition~\ref{pr:x(T)=x(T')}.}
    \label{fig:webs15}
\end{figure}
There are only a finite number of such exceptional situations, 
all of them arising when some sides of the relevant 
triangle(s) are short enough to make certain special
invariants factor nontrivially. 
We refrain from exhaustively describing all these exceptional cases and
the corresponding quiver-building instructions. 

\pagebreak[3]

\section{Proof of the main theorem 
}
\label{sec:proof-main}

\begin{lemma}
\label{lem:mut-flips}
The seeds $(Q(T),\zz(T))$ are mutation equivalent. 
\end{lemma}

Strictly speaking, we should not use the terms ``seed'' and
``mutation'' since we have not established yet that the elements of
$\zz(T)$ are algebraically independent.
What we mean in Lemma~\ref{lem:mut-flips} is that any two quivers
$Q(T)$ and $Q(T')$ are related by a sequence of mutations,
and applying the corresponding exchange transformations to the
collection $\zz(T)$ produces the collection~$\zz(T')$. 

\begin{proof}
Since any two triangulations of the polygon~$P_\sigma$ can be
connected by a sequence of \emph{flips} 
(each flip replacing a single diagonal by another one),
it suffices to show that for two triangulations that differ by a flip, 
the corresponding seeds are related by mutations.
Typically, the required number of mutations is four,
although in some cases it can be three, two, one, and even zero,
see Proposition~\ref{pr:x(T)=x(T')}.

The verification of the claim is done on a case by case basis,
depending on the color pattern of the vertices involved;
as before, there are several exceptional cases where proximity of 
vertices to each other 
results in nontrivial factorization of the corresponding special invariant. 
Once again, we do not list all cases exhaustively,
presenting instead a couple of 
``generic'' cases which illustrate the checks
one needs to perform; this is done in the next paragraph. 
But first, let us explain how one would exhaustively examine
all possible ``non-generic'' cases. 
For each side of a quadrilateral, there are four 
options to consider: 
the side may be of length~$1$, $2$, $3$, or~$\ge 4$. 
Once each of these four options has been selected for each side, 
one considers possible colorings of the vertices along the sides. 
For the sides of length~$\ge 4$, the coloring of the 
intermediate vertices does not matter. 
In each case, it is possible to exhibit a sequence of $\le 4$
mutations which connects the two seeds under consideration. 
Details are omitted.
We note that in the case of Grassmannians (when all boundary vertices
are of the same color), an exhaustive examination of all possibilities
is given in Section~\ref{sec:other-proofs}. 

The general rule (barring aforementioned exceptions) is as follows. 
Suppose that triangulations $T$ and~$T'$ are related by a flip that
replaces diagonal~$pr$ by diagonal~$qs$;
thus $T$ has triangles $pqr$ and $rsp$ while $T'$ has triangles $qrs$
and~$pqs$.
Then the seed $(Q(T'),\zz(T'))$ is 
obtained from $(Q(T),\zz(T))$
by a sequence of four mutations:
\begin{itemize}
\item
replace $J_r^p$ by $J_{pqs}$, and 
replace $J_p^r$ by $J_{qrs}$ (these two mutations commute);
\item
replace $J_{prs}$ by $J_q^s$, and 
replace $J_{pqr}$ by $J_s^q$ (these two mutations commute).
\end{itemize}
The corresponding exchange relations are of the
form~\eqref{eq:3-term-2}. 
Figures~\ref{fig:webs10} and~\ref{fig:webs8}
illustrate these sequences of mutations for two different color
patterns. 
\end{proof}

In Section~\ref{sec:other-proofs}
(see the proof of Theorem~\ref{th:we=scott}), 
we examine exceptional instances of mutation sequences associated with
diagonal flips 
in the special case of a monochromatic signature 
(equivalently, the case of a Grassmannian $\operatorname{Gr}_{3,b}$). 




\begin{figure}[ht]
    \begin{center}
    \input{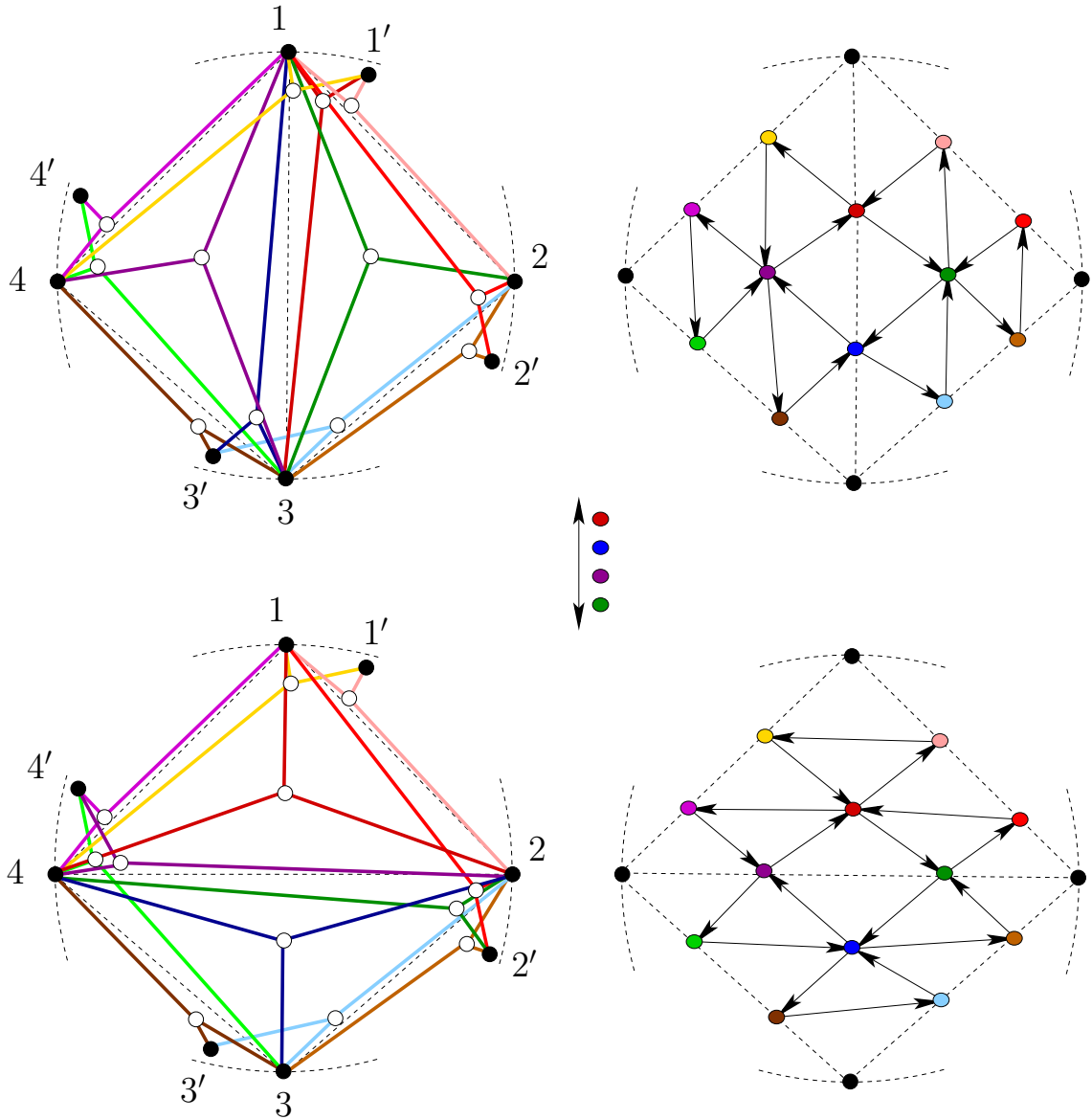} 
    \end{center} 
\caption{Diagonal flip in a quadrilateral with four black
vertices. 
These seeds are related by 
a sequence of four mutations: 
$J_3^1\to J_{124}$, 
$J_1^3\to J_{234}$, 
$J_{134}\to J_2^4$, 
$J_{123}\to J_4^2$. 
}
    \label{fig:webs10}
\end{figure}

\begin{figure}[ht]
    \begin{center}
    \input{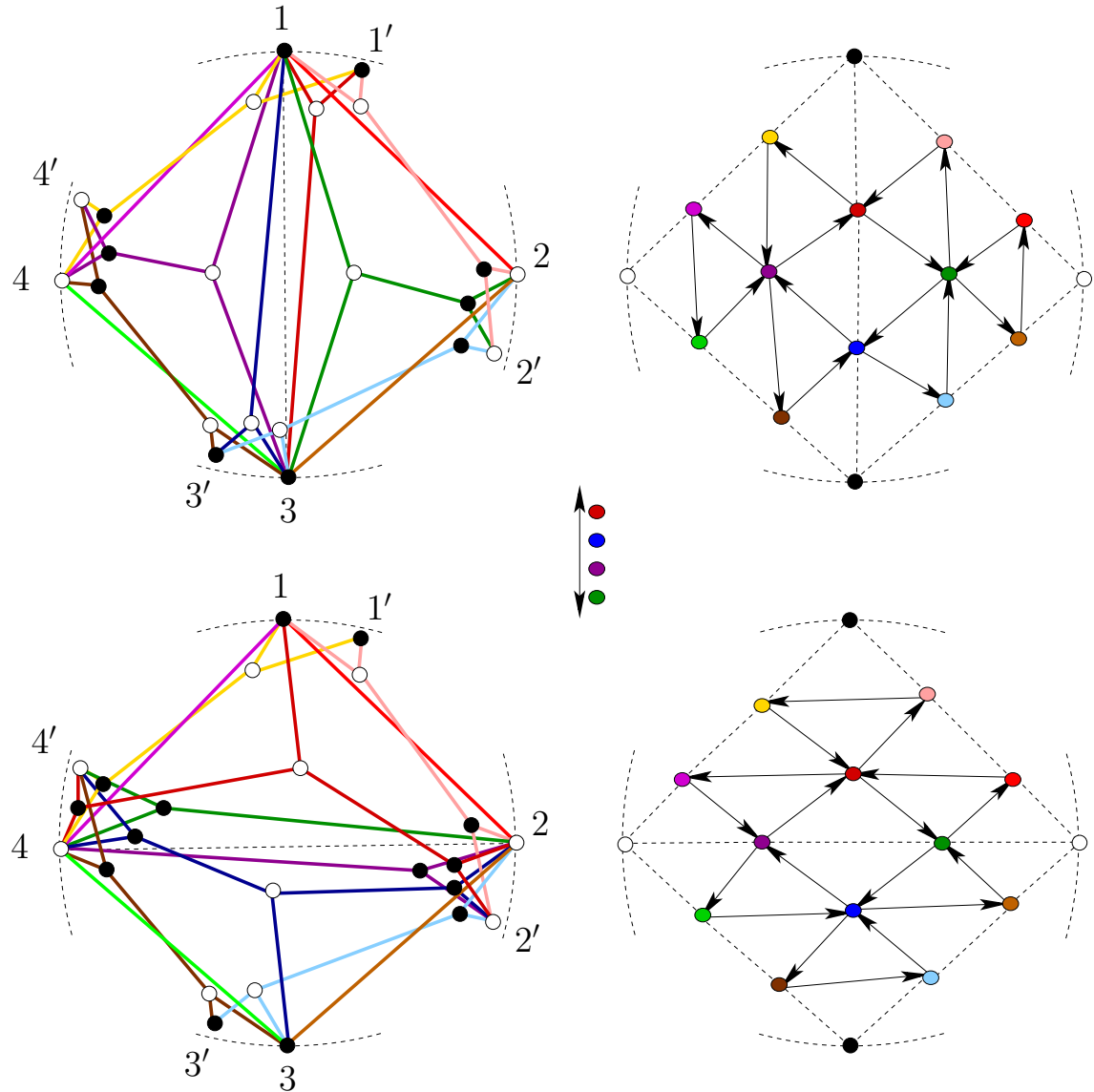} 
    \end{center} 
    \caption{Diagonal flip in a quadrilateral with two black and two
      white vertices, with opposite vertices of the same color.
Mutations to perform: 
$J_3^1\to J_{124}$, $J_1^3\to J_{234}$,
$J_{134}\to J_2^4$, $J_{123}\to J_4^2$. 
}
    \label{fig:webs8}
\end{figure}

\begin{lemma}
\label{lem:alg-indep}
The elements of each set $\zz(T)$ 
(cf.\ Definition~\ref{def:z(T)})
are algebraically independent. 
\end{lemma}

\begin{proof}
Note that $\operatorname{Spec}(R_{a,b}(V))$
has the same dimension as $\operatorname{Spec}(R_{0,a+b}(V))$,
the affine cone over the Grassmannian $\operatorname{Gr}_{3,a+b}\,$;
the latter dimension is equal to $3(a+b-3)+1=3a+3b-8$.
To prove the equality of dimensions, 
assume without loss of generality that $a\ge 3$, 
and use the $\SL(V)$-equivariant birational isomorphism 
that sends an $a$-tuple of covectors 
$(u_1^*,u_2^*\dots,u_a^*)$ to the $a$-tuple of vectors 
$(u_1^*\times u_2^*,u_2^*\times u_3^*,\dots,u_a^*\times u_1^*)$. 
In view of Lemma~\ref{lem:mut-flips}, the $(3a+3b-8)$-tuples $\zz(T)$ 
(cf.\ Theorem~\ref{th:z(T)}) are
birationally related to each other, and collectively
generate the field of fractions of~$R_\sigma$ 
(as they contain all Weyl generators).
It follows that each of these tuples is algebraically independent. 
\end{proof}

\begin{proof}[Proof of Theorem~\ref{th:main}]

Lemma~\ref{lem:alg-indep} means that $(Q(T),\zz(T))$ is a seed. 
Lemma~\ref{lem:mut-flips} means that~the cluster algebra
$\Acal(Q(T),\zz(T))$ does not depend on~$T$. 
Now we are going to use Corollary~\ref{cor:cluster-criterion}, 
which is applicable by Lemma~\ref{lem:properties-of-RabV}. 
The elements of~$\zz(T)$ are irreducible. 
By construction in 
Proposition \ref{prop:special-seeds-exchanges},
all exchange relations from the special seeds
$(Q(T),\zz(T))$ involve exclusively special invariants.
One can then verify using 
Lemma~\ref{prop:special-are-irreducible} that 
the cluster variables appearing in seeds adjacent to the special ones 
are irreducible as well. 
Then Corollary~\ref{cor:cluster-criterion} yields the inclusion 
$R_\sigma(V)\!\supset\!\Acal(Q(T),\zz(T))$. 

The proof of the reverse inclusion is based on the second half of
Corollary~\ref{cor:cluster-criterion}. 
Here we use Lemma~\ref{lem:mut-flips}, 
the First Fundamental Theorem of invariant theory,
and the fact that all Weyl generators show up in extended
clusters~$\zz(T)$.  
Hence $R_\sigma(V)\!=\!\Acal(Q(T),\zz(T))$, and the theorem is
proved. 
\end{proof}

\section{Other proofs}
\label{sec:other-proofs}

\begin{proof}[Proof of Theorem~\ref{th:we=scott}]

It is straightforward to verify that the construction described in~\cite{scott}
arises as a special case of our setup for the ``zig-zag''
triangulation that includes all diagonals of the form $(i,b-i)$ and 
$(i,B+1-i)$. 

For the sake of completeness, we include a case-by-case verification
that in the case of monochromatic signature, the seeds $Q(T),\zz(T)$ 
related by a single flip are indeed obtained from each other
by mutations. 
We stated this claim in Section~\ref{sec:proof-main} for an
arbitrary signature, but did not conduct an exhaustive examination of all
possible non-generic cases. 
Such an examination (for $a=0$) is presented in 
Figures~\ref{fig:webs49b}--\ref{fig:webs48b}.
\end{proof}

\begin{figure}[h!]
    \begin{center}
\scalebox{0.6}{ \input{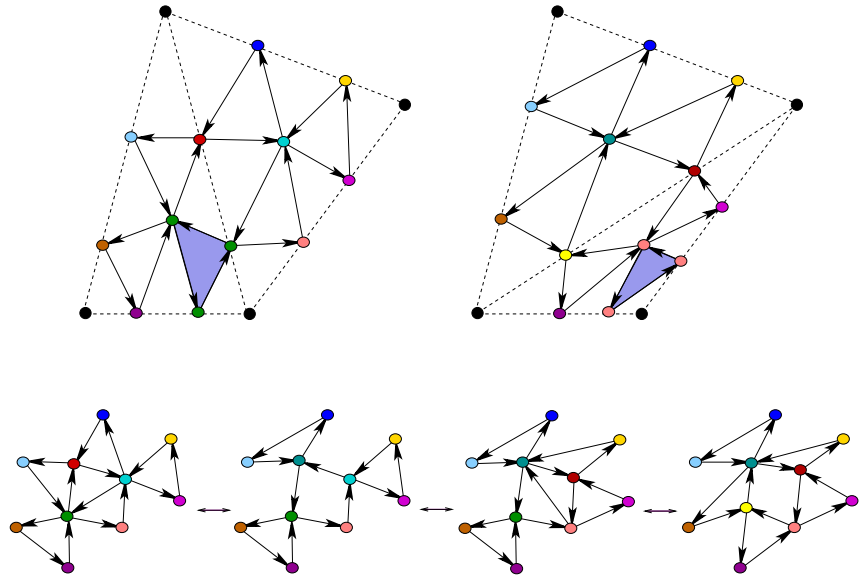}}
    \end{center} 
    \caption{A flip in $R_{0,b}(V)$: one exposed side, three mutations.}
    \label{fig:webs49b}
\end{figure}

\begin{figure}[h!]
    \begin{center}
\scalebox{0.6}{ \input{webs48a.pstex_t}}
    \end{center} 
    \caption{A flip in $R_{0,b}(V)$: two adjacent exposed
    sides, two mutations.}
    \label{fig:webs48a}
\end{figure}

\begin{figure}[h!]
    \begin{center}
\scalebox{0.6}{ \input{webs49a.pstex_t}}
    \end{center} 
    \caption{A flip in $R_{0,b}(V)$: two non-adjacent exposed
    sides, two mutations.} 
    \label{fig:webs49a}
\end{figure}

\begin{figure}[h!]
    \begin{center}
\scalebox{0.6}{ \input{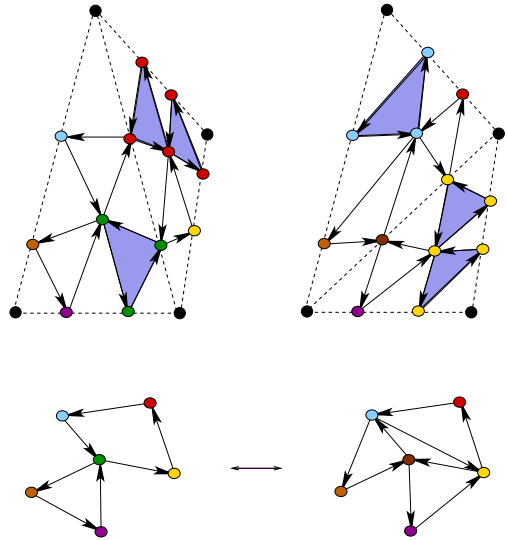}}
    \end{center} 
    \caption{A flip in $R_{0,b}(V)$: three exposed sides, one mutation.}
    \label{fig:webs48b}
\end{figure}

\newpage

\begin{proof}[Proof of Theorem~\ref{th:drop-vertex}]

Our goal is to identify the cluster structure after dropping a
boundary vertex with a part of the original cluster structure. 
There are several cases to consider. 
Figure~\ref{fig:webs20} explains how to handle the case where the
color of the dropped vertex differs from both of its neighbors.
Other cases are treated in a similar way. 
\end{proof}

\begin{figure}[ht]
    \begin{center}
    \input{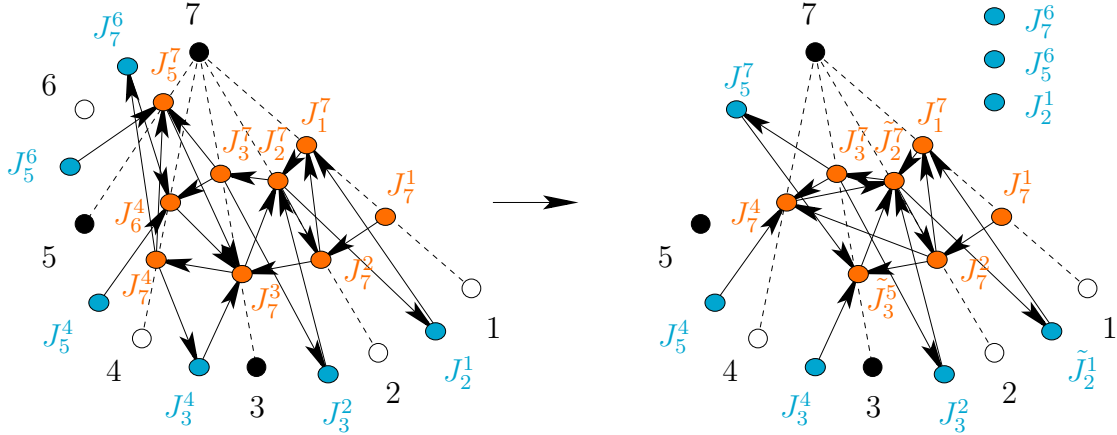}
    \end{center}
    \caption{Vertex $6$ is being dropped. 
Notation $\tilde J$ refers to 
special invariants after the removal.
The identification is achieved by exchanging
$J_6^4$ for $J_{357} = {\tilde {J_3^5}}$, 
then $J_7^3$ for ${\tilde{J_2^7}}$,
and finally $J_2^7$ for~${\tilde{J_2^1}}$;
then freezing $J_5^7$ and ${\tilde{J_2^1}}$
and subsequently throwing away $J_2^1$, $J_5^6$, and~$J_7^6$.  
%
%
The relevant exchange relations are:
$J_6^4 J_{357} = J_7^3 + J_3^7 J_7^6 J_5^4$,
$J_7^3 {\tilde{J_2^7}} = J_2^7 J_{357} J_7^4 + J_7^6 J_5^7 J_5^4 J_3^4 J_7^2
J_3^7$, and 
$J_2^7 {\tilde{J_2^1}} =
{\tilde{J_2^7}} J_2^1 + J_7^6 J_5^7 J_5^4 J_3^4 J_3^2 J_1^7$.}
    \label{fig:webs20}
\end{figure}

\begin{proof}[Proof of Theorem~\ref{th:fork}]

We follow the pattern of the proof of
Theorem~\ref{th:drop-vertex},
examining one characteristic example 
(see Figure~\ref{fig:webs19})
instead of conducting a formal case by case analysis. 
The size of this example is large enough to demonstrate 
what happens in general. 
\end{proof}

%

\begin{figure}[ht]
    \begin{center}
    \input{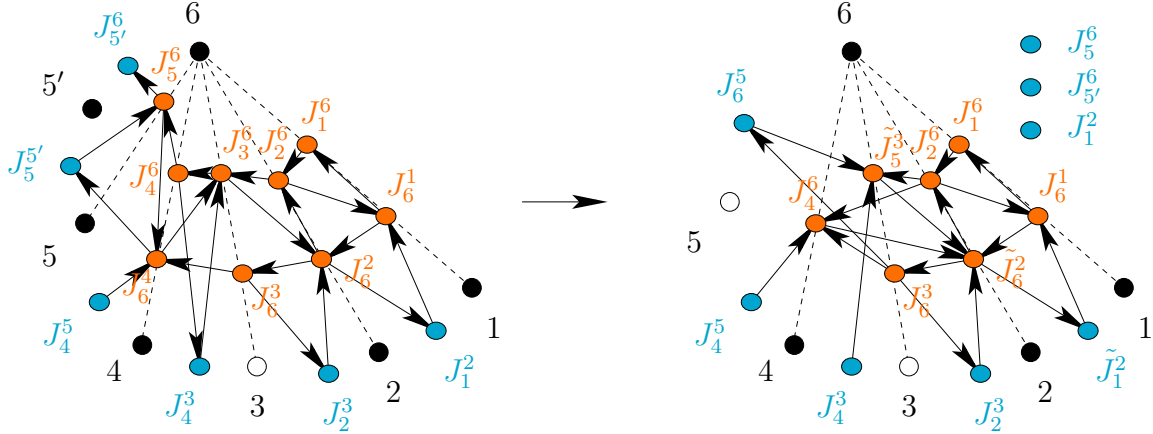}
    \end{center}
    \caption{Replacing black vertices $5$ and~$5'$ 
by the white vertex~$5$. 
Notation $\tilde J$ refers to
special invariants after the replacement.
%
We go counterclockwise along the boundary until the colors 
repeat (at vertices $1,2$). 
The~identification
is achieved by exchanging
$J_6^4$ for~${\tilde {J_5^3}}$, 
then $J_3^6$ for~${\tilde{J_6^2}}$,
and finally $J_6^2$ for~${\tilde{J_1^2}}$;
then freezing $J_5^6$ and~${\tilde{J_1^2}}$, and throwing out
$J_1^2$, $J_{5'}^6$, and~$J_5^6$. 
The exchange relations are:
$J_6^4 {\tilde {J_5^3}} = J_5^{5'} J_3^6 + J_4^5 J_6^3\,$, 
$J_3^6 {\tilde{J_6^2}} = J_4^6 {\tilde {J_5^3}} J_6^2 + J_5^6 J_6^3 J_2^6
J_4^3 J_4^5\,$, and
$J_6^2 {\tilde{J_1^2}} =
{\tilde{J_6^2}} J_1^2 + J_5^6 J_6^1 J_4^5 J_4^3 J_2^3\,.$
}
    \label{fig:webs19}
\end{figure}

\pagebreak[3]

\begin{proof}[Proof of Corollary~\ref{cor:planar-tree}]
Any planar tree can be grown by repeated application of the operations
of adding a fork and dropping a vertex. 
Make sure to never pass 
through an alternating signature; this can indeed be done. 
We omit the details. 
\end{proof}

\begin{proof}[Proof of Theorem~\ref{th:swap-colors}]

The construction of an extended cluster $\zz(T)$ can be carried out
verbatim with the black and white colors swapped.
Let us denote the resulting set by~$\zz'(T)$.  
Note that the definition of a proxy vertex remains the same. 
The only difference between the two constructions is that for each
triangle $pqr$ in~$T$, we use (the factors of) 
the special invariant $J_{pqr}$ to build~$\zz(T)$ 
whereas the construction of $\zz'(T)$ involves~$J^{pqr}$. 
In each instance where the latter choice yields an irreducible
factor that does not come from any special invariant
contributing to~$\zz(T)$, one can use the relation \eqref{eq:3-term-1} to
mutate between the two seeds. 
This can be verified on a case by case basis,
treating each fundamental region separately.
\end{proof}

\begin{proof}[Proof of Lemma~\ref{lem:arborizing-step}]
A skein relation yields the identity shown in Figure~\ref{fig:webs27b}. 
It remains to note that the second term vanishes 
as it is both symmetric and 
antisymmetric with respect to the (co)vectors corresponding to the two
sibling vertices. 
\end{proof}

\begin{figure}[ht]
    \begin{center}
    \input{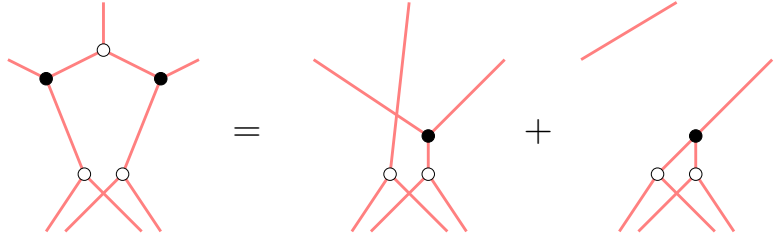} 
    \end{center} 
    \caption{Skein relation justifying an arborizing step.}
    \label{fig:webs27b}
\end{figure}

\begin{proof}[Proof of Theorem~\ref{th:arborization-confluent}]
The proof is a straightforward application of the Diamond Lemma (see,
e.g., \cite[Lemma~1.4.1]{pm-cohn}). 
The five possible nontrivial diamonds are shown in Figure~\ref{fig:webs28}.
\end{proof}

\begin{figure}[ht]
    \begin{center}
  \scalebox{0.85}{  
    \input{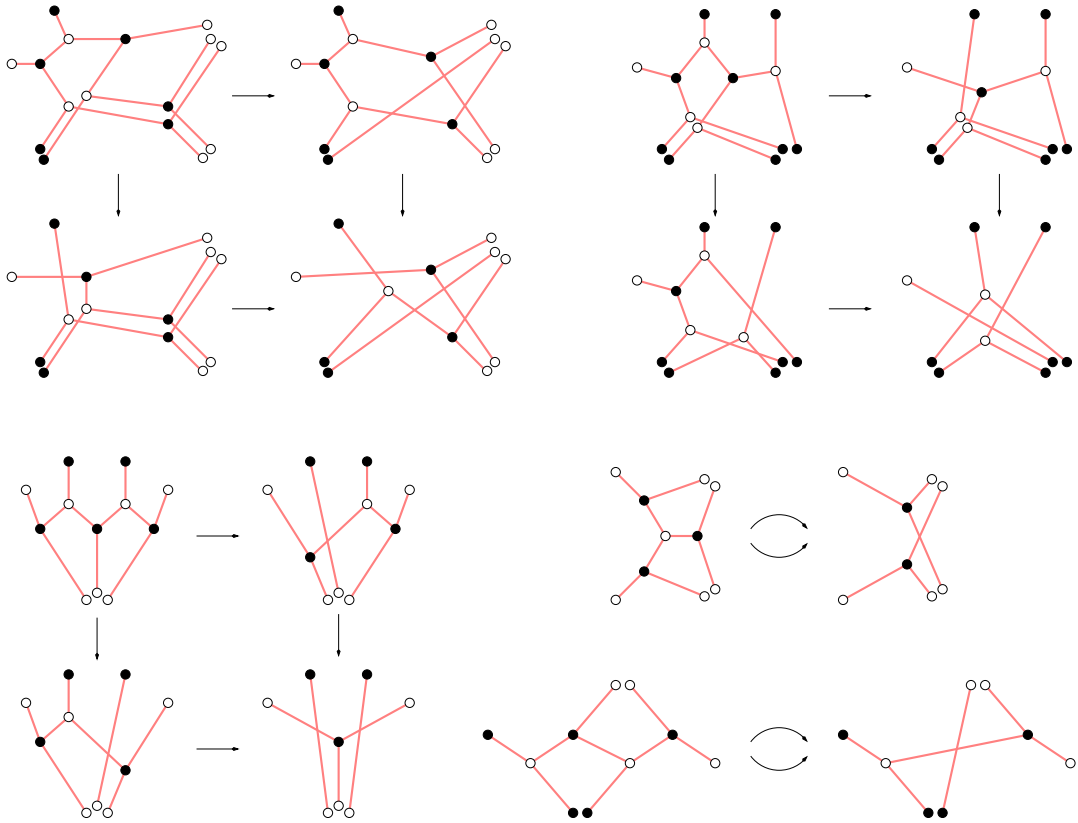} 
}
    \end{center} 
    \caption{Proving the confluence of the arborization algorithm.}
    \label{fig:webs28}
\end{figure}


\begin{proof}[Proof of Theorem~\ref{th:thickening}]


The idea of the proof is as follows. 
Assume that $W$ arborizes to a tree diagram~$D$. 
Let us reverse the arborization process: starting
with~$D$, we step by step ``planarize'' it (cf.\ Figure~\ref{fig:basic-step}), 
each time creating a four-edge fragment of one of the two kinds described in
Definition~\ref{def:arborizing-step}. 
Now make a tensor diagram for~$z^k$ by bundling together $k$ copies
of~$D$, and apply a planarization procedure following the same steps, 
each time planarizing all $k$ copies. 
At the end of this process, we are going to obtain the $k$-thickening of~$W$,
thereby proving the theorem. 

We illustrate how this works  using the example in
Figure~\ref{fig:webs31}. 
The left picture shows an intermediate stage of the planarization
process, with $k=3$. 
The six white vertices correspond to two sibling vertices in~$D$. 
Remember that we are transforming our tensor diagram from the inside
out, i.e., moving towards the boundary. 
The portions attached underneath those six vertices all look
identical: 
they are patterned after the same subtree of~$D$. 
We proceed in two stages as shown in the figure. 
At the first stage, we planarize 
around the two triples of black vertices, using skein relations
(including the square move). 
Each time, all terms but one vanish because they can be seen
to be both symmetric and antisymmetric with respect to a pair of 
arguments. 
At the second stage, 
we planarize the intersection of $k$-thickenings of two edges of~$D$
using a similar calculation. 

\end{proof}

\begin{figure}[ht]
    \begin{center}
  \scalebox{0.9}{  
\input{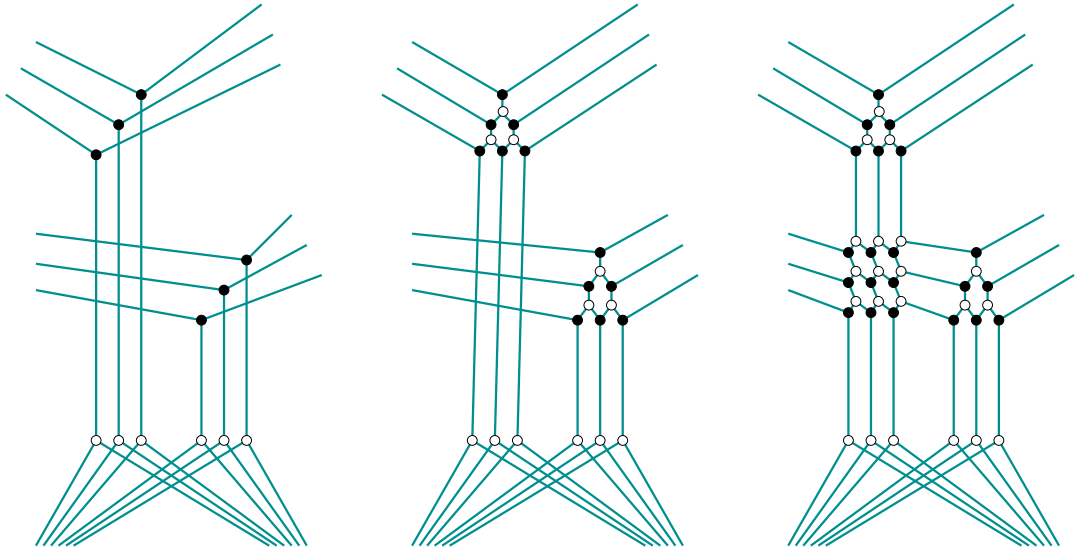} 
}
    \end{center} 
    \caption{ L'Ouvrier et la Kolkhozienne}
    \label{fig:webs31}
\end{figure}

\newpage

\usection{Appendices}

\section{Special choice of an initial seed}
\label{sec:T-fan}

The cluster structure in the ring of invariants~$R_\sigma(V)$
is in principle determined by a single initial seed $(Q(T),\zz(T))$. 
Consequently, one may be interested
in identifying a particular choice of a triangulation~$T$ for which 
the seed $(Q(T),\zz(T))$ has a simpler and more 
explicit description than the general case as presented in
Sections~\ref{sec:special-invariants}--\ref{sec:special-seeds}. 
One~such choice is discussed in this section. 

 Denote $N\!=\!a+b$. 
Since the signature $\sigma$ is non-alternating, we can assume without loss of
generality that the vertices $1$ and~$2$ are black. 
(If we can only find two adjacent white vertices, then 
use the same recipe with the colors swapped, cf.\ Theorem~\ref{th:swap-colors}.) 

The case $a=0$ has been covered in Figure~\ref{fig:webs47}. 
So let $\sigma$ be non-monochromatic.
Without loss of generality, we can assume that the
vertex~$N$ is white. 

Let $T_1$ be the triangulation of the $N$-gon~$P_\sigma$ 
obtained by drawing all diagonals with an endpoint at the vertex~$1$. 
The quiver $Q(T_1)$ is constructed as follows. 

\begin{enumerate}
 \item Place two (mutable) vertices of $Q(T_1)$ on each diagonal of~$T_1$.
Place one frozen vertex on each side of~$P_\sigma$. 
 \item Fill each triangle in $T_1$ with arrows as shown in
   Figure~\ref{fig:webs44}. 
The choice of a pattern is determined by the color of the right endpoint
   of the base side.
 \item Assemble a quiver from these pieces,
removing $2$-cycles if necessary. 
 \item On the diagonal $13$, freeze the vertex closer to~$1$,
and identify it with the frozen vertex~$12$.
\item
On the diagonal $1(N\!-\!1)$, freeze one vertex, namely the one closer
to~$1$ (resp.,~$N\!-\!1$) if $N\!-\!1$ is white (resp., black). 
Identify it with $(N\!-\!1)N$
if 
$N-1$ is white, 
and with $1N$ 
if it is black. 
 \item Add arrows connecting the mutable vertices on diagonals 
$13$ and $1(N\!-\!1)$ 
to the frozen vertices $12$, $23$, $(N\!-\!1)N$, and~$1N$ as shown
in Figure~\ref{fig:webs44} on the~right. 
Remove 2-cycles if necessary;
this will happen if vertex $3$ is black.
\item
\label{enum:except-Q}
If 
$N-1$ is black, 
replace the portions of the
resulting quiver inside the pentagon $1(N-3)(N-2)(N-1)N$ as shown in 
Figure~\ref{fig:webs45}.
\end{enumerate}

\begin{figure}[h!]
    \begin{center}
\input{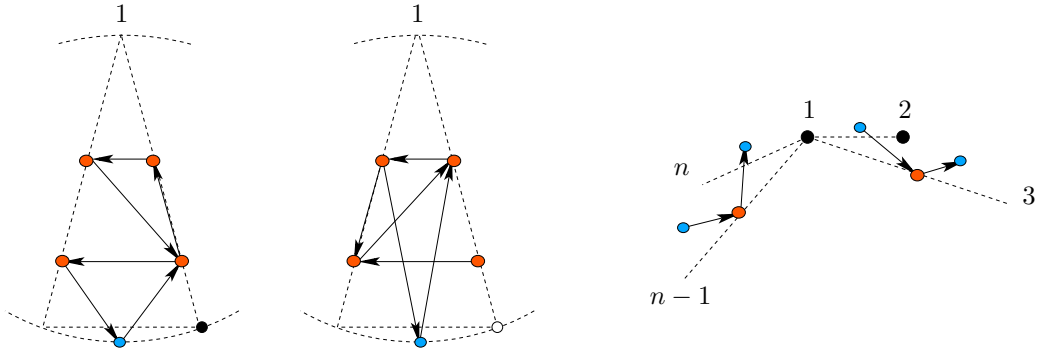}
\end{center}
\caption{Constructing the quiver $Q(T_1)$: 
Building blocks}
    \label{fig:webs44}
\end{figure}




\begin{figure}[h!]
    \begin{center}
\scalebox{1.1}{
    \input{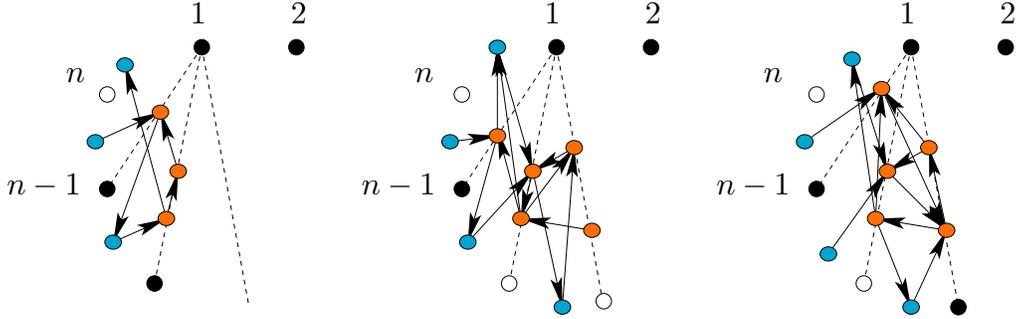}
}
\end{center}
\caption{Constructing the quiver $Q(T_1)$: 
Exceptional cases}
    \label{fig:webs45}
\end{figure}

Figure~\ref{fig:webs46} shows an example of applying this
construction.

\begin{figure}[h!]
    \begin{center}
    \input{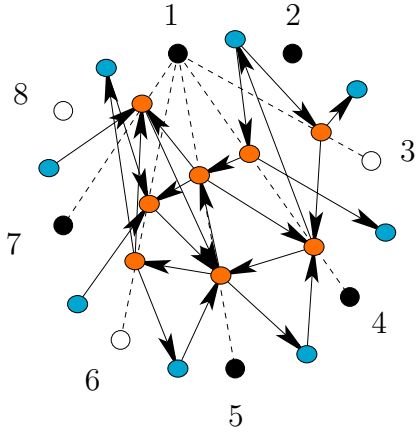} 
    \end{center}
    \caption{Constructing the quiver $Q(T_1)$. 
In this example, one of the exceptional cases
in step~\eqref{enum:except-Q} applies.}
    \label{fig:webs46}
\end{figure}

When the vertex $N-1$ is white,
the above construction of the quiver $Q(T_1)$ simplifies substantially
as we do not have to make the adjustments described in 
step~\eqref{enum:except-Q}.
If one is only interested in the cluster type of~$R_\sigma(V)$,
the description simplifies further:

\begin{proposition}
Let $\sigma=[\sigma_1\cdots\sigma_N]$, $N\ge 6$, be a signature 
such that $\sigma_1=\sigma_2=\bullet$ and
$\sigma_{N-1}=\sigma_N=\circ$. 
Let $T_1$ be the triangulation of $P_\sigma$ by the diagonals incident
to the vertex~$1$. 
Then the mutable part of the quiver $Q(T_1)$ is obtained by 
taking the vertices $s_3,s_4,\dots,s_{N-1}$ and
$t_4,t_5,\dots,t_{N-2}$
together with the following edges: 
\begin{itemize}
\item
for $3\le i\le N-2$, draw an edge $s_i\to s_{i+1}\,$;
\item
for $4\le i\le N-3$, draw an edge $t_i\to t_{i+1}\,$;
\item
for $3\le i\le N-3$, if $\sigma_i=\bullet$, then draw an edge
$t_{i+1}\to s_i\,$; 
\item
for $3\le i\le N-3$, if $\sigma_i=\circ$, then draw an edge
$t_{i+1}\to s_{i+1}\,$; 
\item
for $4\le i\le N-2$, if $\sigma_i=\bullet$, then draw an edge
$s_i\to t_i\,$; 
\item
for $4\le i\le N-2$, if $\sigma_i=\circ$, then draw an edge
$s_{i+1}\to t_i\,$.
\end{itemize} 
\end{proposition}

\pagebreak

\section{Additional examples of seeds}
\label{sec:examples-of-seeds}

\begin{figure}[ht]
    \begin{center}
    \input{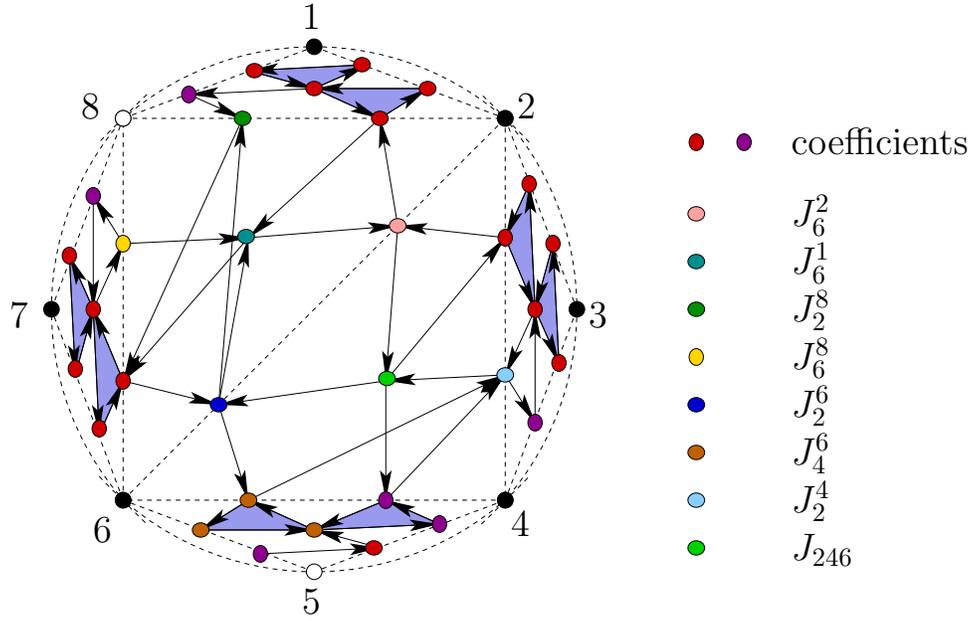} 
    \end{center} 
    \caption{An example of a seed $(Q(T),\zz(T))$ for 
$\sigma=[\bullet\bullet\bullet\bullet\circ\bullet\bullet\,\circ]$. 
}
    \label{fig:webs13}
\end{figure}

\begin{figure}[ht]
    \begin{center}
    \input{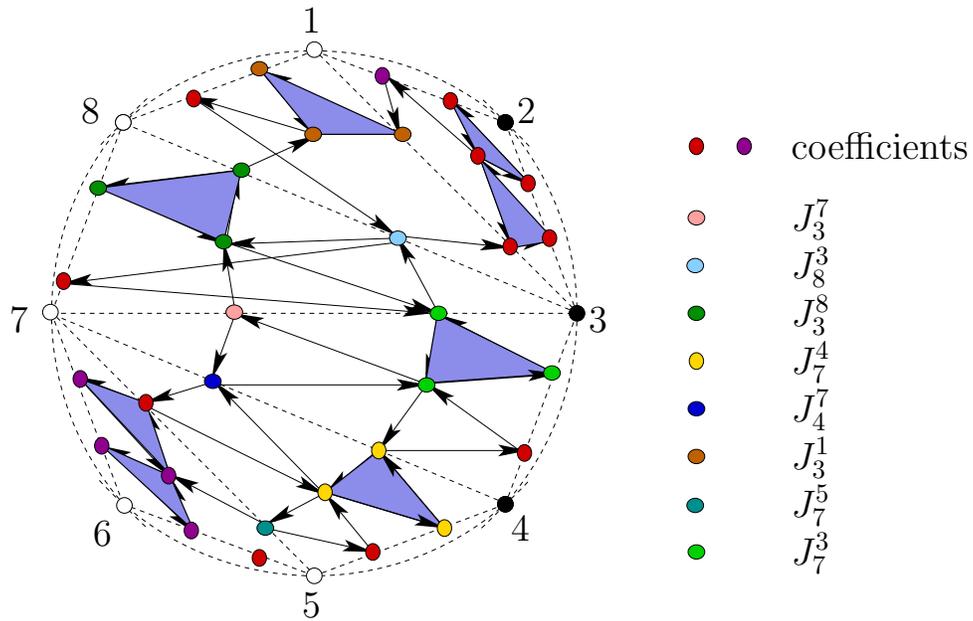} 
    \end{center} 
    \caption{An example of a seed $(Q(T),\zz(T))$ for 
$\sigma=[\circ\circ\circ\circ\circ\bullet\bullet\,\bullet]$.
}
    \label{fig:webs16}
\end{figure}

\end{document}